\documentclass[a4paper,10.5pt]{article}
\usepackage{amsmath,amssymb,amsthm}
\theoremstyle{definition}
 \newtheorem{dfn}{Definition}
 \newtheorem{remark}[dfn]{Remark}

\theoremstyle{plain}
 \newtheorem{thm}[dfn]{Theorem}
 
 \newtheorem{lem}[dfn]{Lemma}

\numberwithin{equation}{section}
\newcommand{\bn}{{\bold n}}
\newcommand{\ba}{{\bold a}}
\newcommand{\bb}{{\bold b}}
\newcommand{\bd}{{\bold d}}
\newcommand{\bu}{{\bold u}}
\newcommand{\bv}{{\bold v}}
\newcommand{\bw}{{\bold w}}
\newcommand{\bff}{{\bold f}}
\newcommand{\be}{{\bold e}}
\newcommand{\bg}{{\bold g}}
\newcommand{\bh}{{\bold h}}
\newcommand{\bk}{{\bold k}}
\newcommand{\bp}{{\bold p}}
\newcommand{\bA}{{\bold A}}
\newcommand{\bB}{{\bold B}}

\newcommand{\bD}{{\bold D}}

\newcommand{\bF}{{\bold F}}
\newcommand{\bG}{{\bold G}}
\newcommand{\bH}{{\bold H}}
\newcommand{\bI}{{\bold I}}

\newcommand{\bK}{{\bold K}}

\newcommand{\bV}{{\bold V}}

\newcommand{\DV}{{\rm Div}\,}
\newcommand{\dv}{{\rm div}\,}

\newcommand{\BR}{{\Bbb R}}
\newcommand{\BC}{{\Bbb C}}

\newcommand{\BN}{{\Bbb N}}

\newcommand{\BT}{{\Bbb T}}

\newcommand{\BZ}{{\Bbb Z}}

\newcommand{\CA}{{\mathcal A}}
\newcommand{\CB}{{\mathcal B}}

\newcommand{\CD}{{\mathcal D}}

\newcommand{\CF}{{\mathcal F}}
\newcommand{\CI}{{\mathcal I}}

\newcommand{\CL}{{\mathcal L}}
\newcommand{\CM}{{\mathcal M}}

\newcommand{\CR}{{\mathcal R}}
\newcommand{\CS}{{\mathcal S}}
\newcommand{\CT}{{\mathcal T}}
\newcommand{\CH}{{\mathcal H}}

\newcommand{\CP}{{\mathcal P}}

\newcommand{\CX}{{\mathcal X}}

\newcommand{\fp}{{\frak p}}
\newcommand{\fq}{{\frak q}}
\newcommand{\fr}{{\frak r}}

\newcommand{\pd}{\partial}

\newenvironment{cases*}%
{%
\left\{
\begin{array}{@{}r@{\;}l@{\quad}l@{}}
}%
{\end{array}\right.}
\begin{document}
\title{On periodic solutions for one-phase and  
two-phase problems of the Navier-Stokes equations}
\author{
Thomas EITER \thanks{
\small{Fachbereich Mathematik, 
Technische Universit\"at Darmstadt, 
Schlossgartenstr.~7, 64289 Darmstadt, Germany, 
{\small e-mail address: eiter@mathematik.tu-darmstadt.de}}}, 
Mads KYED \thanks{Hochschule Flensburg, 
Kanzleistra\ss e 91--93, 24943 Flensburg,
Germany, 
\small{e-mail address: mads.kyed@hs-flensburg.de}}, 
and 
Yoshihiro SHIBATA
\thanks{Partially supported by  
JSPS Grant-in-aid for Scientific Research (A) 17H0109,
and Top Global University Project
\endgraf
Department of Mathematics, Waseda University,
Ohkubo 3-4-1, Shinjuku-ku, Tokyo 169-8555, Japan, 
\small{e-mail address: yshibata325@gmail.com}
\endgraf
Adjunct Faculty member in the Department of Mechanical Engineering and 
Materials Science, University of Pittsburgh.}}
\maketitle
\begin{abstract}
This paper is devoted to proving the existence of time-periodic solutions
of  one-phase or two-phase  problems for the Navier-Stokes equations 
with small periodic external forces when the reference domain is close to 
a ball. Since our problems are  formulated in  time-dependent 
unknown domains, the problems are reduced to quasiliner systems of parabolic equations with
non-homogeneous boundary conditions or transmission conditions
in fixed domains by using the so-called
Hanzawa transform.   
We separate solutions into the stationary part and the oscillatory part.
The linearized equations for the stationary part have eigen-value $0$,
which is avoided by changing the equations with the help of the necessary conditions 
for the existence of solutions to the original problems.  To treat the 
oscillatory part, we establish the maximal $L_p$-$L_q$ regularity theorem
of the periodic solutions 
for the system of parabolic equations with non-homogeneous boundary conditions or
transmission conditions, which is obtained
by the systematic use of $\CR$-solvers developed in \cite{S1, S2, S4, S3} to the resolvent problem for the linearized
equations and the transference theorem obtained in \cite{EKS1}
for the $L_p$ boundedness of 
operator-valued Fourier multipliers. These approaches are the novelty of this paper. 

\end{abstract}

\section{Introduction}\label{sec:1}

This paper is concerned with time-periodic solutions of one-phase and two-phase problems for 
the Navier-Stokes equations. The periodic solutions for the Navier-Stokes
equations have been studied in many articles \cite{GK, GS, GHN, HMT, HNS, KN, 
M91, MP, OT, Salvi, Serrin, Ta99, Y00} and references therein.  One well-known
approach to prove the existence of periodic solutions is  the utilization of the Poincar\'e
operator, which maps an initial value into the solution of the
PDE at time $\CT$, where $\CT$ is the period of the data. A fixed point of the
Poincar\'e operator yields an initial value that induces a $\CT$-time-periodic 
solution.  Such a utilization of the Poincar\'e operator is naturally carried out under the 
global well-posedness of the corresponding
initial-boundary value problem for the bounded data on the right hand side of the equations.
In the bounded domain case, this is deeply related with the situation where
$0$ does not belong to the spectrum of the system of the linearized equations. 
However, in many interesting problems in mathematical physics, we meet the situation that
$0$ is in the spectrum.  One-phase or two-phase problems for
the Navier-Stokes equations are typical examples. As explained in Sections \ref{sec:1}
and \ref{sec:2}
below, the one-phase and two-phase problems we treat in this paper are formulated by
the Navier-Stokes equations with free boundary condition or transmission condition
on the interface in a time-dependent domain $\Omega_t$, which is also unknown.
Usually, $\Omega_t$ is transformed to a fixed domain $\Omega$ by introducing 
an unknown function representing the boundary or the interface of $\Omega_t$.
Thus, the problem treated here becomes a quasilinear system of equations with nonlinear
boundary or transmission conditions. The first of our key approaches is to separate
solutions into stationary part and oscillatory part. Then, the zero eigen-value 
of the linearized equations appears only in the equations for the stationary problem. 
We change the linearized equations
by using some necessary conditions for the unique existence of solutions to avoid
eigen-value $0$ for the linearized problem.  This technique is possible under the 
separation of the stationary part and the oscillatory part, which does not appear when working with
the Poincar\'e operator. 
The second is to introduce a systematic approach to the maximal $L_p$-$L_q$ regularity for the 
oscillatory part based solely on the $\CR$-solver for the resolvent problem of the
linearized equations developed in \cite{S1, S2, S3, S4, SS1} 
and a transference theorem for the $L_p$ boundedness of the
operator-valued Fourier multiplier due to Eiter, Kyed and Shibata  in \cite{EKS1}.  
The $L_p$-$L_q$ maximal regularity for 
the oscillatory part of solutions is necessary because our problem is 
a quasilinear system with non-homogeneous boundary conditions.  Since
the maximal regularity for the oscillatory part of the periodic solutions does not
seem to be well-studied,  our systematic approach gives a quite important contribution to
the study of systems of parabolic equations with non-homogeneous boundary conditions,
which is  the novelty of this paper.

\subsection{One-phase problem} 
Let $\Omega_t$ be a time-dependent domain in the $N$-dimensional 
Euclidean space $\BR^N$ ($N \geq 2$). 
Let $\Gamma_t$ be the boundary of $\Omega_t$ and $\bn_t$ the unit
outer normal to $\Gamma_t$.  We assume that $\Omega_t$ is occupied by some
incompressible viscous fluid of unit mass density whose viscosity coefficient is
a positive constant $\mu$.  Let $\bu = {}^\top(u_1(x, t), \ldots, u_N(x,t))$,
$x = (x_1, \ldots, x_N)\in \Omega_t$, and $\fp=\fp(x, t)$ be
the velocity field and the pressure field in $\Omega_t$, respectively, where ${}^\top M$ denotes
the transposed of $M$. We consider the Navier-Stokes equations in $\Omega_t$
with free boundary condition as follows:
\begin{equation}\label{eq:1p}\left\{\begin{aligned}
&\pd_t\bu + \bu\cdot\nabla\bu - \DV(\mu\bD(\bu) - \fp\bI)  = \bff
&\quad&\text{in $\Omega_t$}, \\
&\qquad \dv \bu  = 0 
&\quad&\text{in $\Omega_t$}, \\
&(\mu\bD(\bu)-\fp\bI)\bn_t  = \sigma H(\Gamma_t)\bn_t 
&\quad&\text{on $\Gamma_t$}, \\
&\qquad V_{\Gamma_t} = \bu\cdot\bn_t
&\quad&\text{on $\Gamma_t$}
\end{aligned} \right.\end{equation}
for $t \in \BR$.  Here, $\bff = \bff(x, t)$ 
is a prescribed time-periodic external force with period $2\pi$;  
$H(\Gamma_t)$ denotes the $(N-1)$-fold mean curvature of 
$\Gamma_t$ which is given by $H(\Gamma_t)\bn_t = \Delta_{\Gamma_t}x$
for $x \in \Gamma_t$, where $\Delta_{\Gamma_t}$ is the Laplace-Beltrami
operator on $\Gamma_t$; $V_{\Gamma_t}$ is the evolution speed of $\Gamma_t$
along $\bn_t$; 
$\sigma$ is a positive 
constant representing  the surface tension coefficient; $\bD(\bu)$ is the doubled
deformation tensor given by $\bD(\bu) = \nabla\bu + {}^\top\nabla\bu$; and $\bI$ is the 
$(N\times N)$-identity matrix.  Moreover, for any $(N\times N)$-matrix of 
functions $\bK$ whose $(i, j)^{\rm th}$ component is $K_{ij}$, 
$\DV K$ is an $N$-vector whose $i^{\rm th}$ component is 
$\sum_{j=1}^n\pd_jK_{ij}$ and for any $N$-vector of functions 
$\bv = {}^\top(v_1, \ldots, v_N)$, $\bv\cdot\nabla\bv$ is an 
$N$-vector of functions whose $i^{\rm th}$ component is 
$\sum_{j=1}^Nv_j\pd_j v_i$, where $\pd_j=\pd/\pd x_j$.

Our problem is to find  $\Omega_t$,  $\Gamma_t$, $\bu$ and $\fp$ satisfying the periodic
condition:
\begin{equation}\label{eq:p1}\Omega_t = \Omega_{t+2\pi}, \quad \Gamma_t = \Gamma_{t+2\pi},
\quad 
\bu(x, t) = \bu(x, t+2\pi), \quad \fp(x, t) = \fp(x, t+2\pi)
\end{equation}
for any $t \in \BR$. 

To state the main result, we introduce assumptions and some functional spaces. Let $\bp_i = \be_i
={}^T(0, \ldots, 0, \overset{\rm i -th}1, 0, \ldots, 0)$ for $i=1, \ldots, N$ and $\bp_\ell$ ($\ell
=N+1, \ldots, M)$ be one of $x_i\be_j - x_j\be_i$ ($1 \leq i< j \leq N$).  Notice that 
$\bp_\ell$ forms a basis of the rigid space $\{ \bv \mid \bD(\bv) = 0\}$ and the number $M$ is 
its dimension.  We will construct
$\Omega_t$ satisfying the following two conditions:
\begin{gather}
\det\Bigl(\int^{2\pi}_0(\bp_\ell, \bp_m)_{\Omega_t}\,dt\Bigr)_{\ell, m = 1, \ldots, M}
\not=0, \label{assump:1} \\
\int^{2\pi}_0\Bigl(\int_{\Omega_t} x\,dx\Bigr)\,dt = 0,  \label{assump:3}\\
|\Omega_t| = |B_R| \quad\text{for any $t \in (0, 2\pi)$}. \label{assump:5}
\end{gather}
Here and in the following, $(M_{\ell,m})_{\ell, m=1,\ldots, N}$ denotes an $(N\times N)$-matrix whose
$(\ell, m)^{\rm th}$ component is $M_{\ell, m}$; 
for any domain $G$ and $(N-1)$-dimensional hypersurface $S$, we let 
$$(f, g)_{G} = \int_G f(x)\cdot\overline{g(x)}\,dx, \quad 
(f, g)_{S}= \int_{S} f(x)\cdot\overline{g(x)}\,d\sigma,$$
where $\overline{g(x)}$ denotes the complex conjugate of $g(x)$,  and $d\sigma$ the surface element
of $S$. $|G|$ denotes the Lebesgue measure of a Lebesgue measurable set $G$ of $\BR^N$; and
$B_R$ is the ball with radius $R$ centered at the origin. 
For $1 < p < \infty$ and any Banach space $X$ with norm $\|\cdot\|_X$, let 
\begin{align*}
L_{p, {\rm per}}((0, 2\pi), X) & = \{f : \BR \to X  \mid \text{$\|f(\cdot)\|_X \in L_{1, {\rm loc}}(\BR)$},  \\
&f(t+ 2\pi) = f(t)  \quad\text{for any $t \in \BR$}, \quad
 \|f\|_{L_p((0, 2\pi), X)} = \Bigl(\int^{2\pi}_0\|f(t)\|_X^p\,dt\Bigr)^{1/p} < \infty\}, \\
H^1_{p, {\rm per}}(0, 2\pi), X) & = \{f: \BR \to X \mid \text{$\|f(t)\|_X \in L_{1, {\rm loc}}(\BR)$ 
and $\|\dot f(t)\|_X \in L_{1, {\rm loc}}(\BR)$}, \\
&f(t) = f(t+2\pi), \enskip
\dot f(t) = \dot f(t+2\pi) \quad\text{for any $t \in \BR$}, \\
&\|f\|_{H^1_p((0, 2\pi), X)} = \Bigl(\int^{2\pi}_0(\|f(t)\|_X^p
+ \|\dot f(t)\|_X^p)\,dt\Bigr)^{1/p} < \infty\},
\end{align*}
where $\dot f$ denotes the derivative of $f$ with respect to $t$. Let 
$$\|f\|_{L_p((0, 2\pi), X)} = \Bigl(\int^{2\pi}_0\|f(t)\|_X^p\,dt\Bigr)^{1/p},
\quad 
\|f\|_{H^1_p((0, 2\pi), X)} = \|f\|_{L_p((0, 2\pi), X)} + \|\dot f\|_{L_p((0, 2\pi), X)}.
$$
For any domain $G$ in $\BR^N$ and $1 \leq q \leq \infty$, 
 $L_q(G)$, $H^m_q(G)$,  and 
$B^s_{q,p}(G)$  
denote the standard Lebesgue, Sobolev, and 
Besov spaces on $G$, 
and  $\|\cdot\|_{L_q(G)}$, 
$\|\cdot\|_{H^m_q(G)}$, and  
$\|\cdot\|_{B^s_{q,p}(G)}$ denote their respective norms. 
For any  integer $d$, $X^d$ denotes the $d$-fold product of the space $X$, that is
$X^d = \{\bg = {}^\top(g_1, \ldots, g_d) \mid g_j \in X \enskip(j=1, \ldots, d)\}$, while 
the norm of $X^d$ is denoted by $\|\cdot\|_X$ instead of $\|\cdot\|_{X^d}$ for simplicity.

The following theorem is our main result concerning time-periodic solutions of the one-phase problem 
for the Navier-Stokes equations.
\begin{thm}\label{thm:main1} Let $1 < p, q < \infty$ and $2/p + N/q < 1$.  Let $D\subset B_R$ 
be a domain. Then, there exists 
a positive constant $\epsilon$ and an injective map $x = \Phi(y, t) :  
B_R \to \BR^N$ for each $t \in (0, 2\pi)$ with
$$\Phi  \in L_{p, {\rm per}}((0, 2\pi), H^3_q(B_R)^N) \cap H^1_{p, {\rm per}}((0, 2\pi),
H^2_q(B_R))$$
for which the following assertion holds: 
If $\bff \in L_{p, {\rm per}}((0, 2\pi), L_q(D)^N)$ satisfies the support condition: 
${\rm supp}\,\bff(\cdot, t) \subset D$ for any $t \in (0, 2\pi)$, 
the orthogonal condition
\begin{equation}\label{assump:4}
\int^{2\pi}_0 (\bff(\cdot, t), \bp_\ell)_D \,dt = 0 \quad\text{for $\ell=1, \ldots, M$}, 
\end{equation}
and the smallness condition:
$\|\bff\|_{L_p((0, 2\pi), L_q(D)^N)} \leq \epsilon$, then there exist  
 $\bv(y, t)$,  $\fq(y, t)$, and $\rho(y, t)$ with
\begin{equation}\label{reg:1}\begin{aligned}
\bv &\in L_{p, {\rm per}}((0, 2\pi), H^2_q(B_R)^N) \cap H^1_{p, {\rm per}}((0, 2\pi), L_q(B_R)^N), \\
\fq & \in L_{p, {\rm per}}((0, 2\pi), H^1_q(B_R)), \\
\rho & \in L_{p, {\rm per}}((0, 2\pi), W^{3-1/q}_q(B_R)^N) \cap 
H^1_{p, {\rm per}}((0, 2\pi), W^{2-1/q}_q(S_R)),
\end{aligned}\end{equation}
such that  
\begin{gather*}
\Omega_t = \{x =\Phi(y, t) \mid y \in B_R\}, \quad 
\bu(x, t) = \bv(\Phi^{-1}(x, t), t), \quad \fp(x, t) = \fq(\Phi^{-1}(x, t), t),
\end{gather*}
where $\Phi^{-1}(x, t)$ is the inverse map of the correspondence: 
$x = \Phi(y, t)$ for any $t \in (0, 2\pi)$, 
are solutions of equations \eqref{eq:p1} satisfying the periodicity condition
\eqref{eq:p1}, and  $\Gamma_t$ is given by 
$$\Gamma_t = \{x = y + R^{-1}\rho(y, t) y+ \xi(t) \mid y \in S_R\},$$
where $\xi(t)$ is the barycenter point of $\Omega_t$ defined by setting 
$$\xi(t) = \frac{1}{|\Omega_t|}\int_{\Omega_t} x\,dx.$$
Moreover, $\bv$ and $\rho$ satisfy the estimate:
\begin{equation}\label{est:int1}\begin{aligned}
&\|\bv\|_{L_p((0, 2\pi), H^2_q(B_R))} + \|\pd_t\bv\|_{L_p((0, 2\pi), L_q(B_R))} \\
&\quad+ \|\rho\|_{L_p((0, 2\pi), W^{3-1/q}_q(S_R))} + \|\pd_t\rho\|_{L_p((0, 2\pi), W^{2-1/q}_q(S_R))}
+\|\pd_t \rho\|_{L_\infty((0,2\pi),W^{1-1/q}_q(S_R))} 
\leq C\epsilon
\end{aligned}\end{equation}
for some constant $C$ independent of $\epsilon$.

\end{thm}
\begin{remark} In the construction of the map $\Phi$, we see that 
$\Phi(y,t) = y + R^{-1}\rho(y, t) + \xi(t)$ for $y \in S_R$.
\end{remark}
\subsection{Two-phase problem}
  Let $\Omega_{+t}$ be a time-dependent 
domain in the $N$-dimensional Euclidean space $\BR^N$. 
Let $\Gamma_t$ be the 
boundary of $\Gamma_t$ and $\bn_t$ its unit outer normal.
Let $\Omega$ be a bounded domain in $\BR^N$ and $S$  the boundary of 
$\Omega$. We assume that $\Omega_{+t} \subset \Omega$ and $\Gamma_t
 \cap S = \emptyset$. 
 Let $\Omega_{-t} = \Omega\setminus(\Omega_{+t}
\cup\Gamma_t)$ and set $\Omega_t = \Omega_{+t} \cup \Omega_{-t}$.
We assume that $\Omega_{\pm t}$ be occupied by some incompressible 
viscous fluids of unit mass densities whose viscosity coefficients are 
positive constants $\mu_\pm$. 
Let 
$\bu = {}^\top(u_1, \ldots, u_N)$ and $\fp$ be the velocity field and 
the pressure field  on $\Omega_t$, respectively.   
 We consider the following Navier-Stokes equations with transmission condition
on $\Gamma_t$ and no-slip condition on $S$:
\begin{equation}\label{eq:2phase}\left\{\begin{aligned}
&\pd_t\bu_\pm + \bu\cdot\nabla\bu_\pm - \DV(\mu\bD(\bu_\pm) - \fp_\pm\bI)  = \bff
&\quad&\text{in $\Omega_{\pm t}$}, \\
&\qquad \dv\bu_\pm = 0 &\quad&\text{in $\Omega_{\pm t}$}, \\
&[[\mu\bD(\bu) - \fp\bI]]\bn_t = \sigma H(\Gamma_t)\bn_t,
\quad [[\bu]]=0
&\quad&\text{on $\Gamma_t$}, \\
&\qquad V_{\Gamma_t} = \bu_+\cdot\bn_t &\quad&\text{on $\Gamma_t$},\\
&\qquad \bu_- = 0 & \quad&\text{on $S$}
\end{aligned}\right.\end{equation}
for $t\in\BR$,
where $\bff=\bff(x, t)$ is a prescribed time-periodic external force
with period $2\pi$;  $\mu$ is the viscosity coefficient given by 
$$\mu = \begin{cases} \mu_+ \quad&\text{in $\Omega_{+t}$}, \\
\mu_-\quad&\text{in $\Omega_{-t}$};
\end{cases}
$$
and  $[[f]]$ denotes the jump 
of $f_\pm$ defined on $\Omega_\pm$ along $\bn_t$ defined by setting
$$[[f]](x_0) = \lim_{x\to x_0 \atop x \in \Omega_{+t}} f_+(x) 
- \lim_{x\to x_0 \atop x \in \Omega_{-t}}
f_-(x)\quad\text{for $x_0 \in \Gamma_t$}.
$$ 
The purpose of this paper is also to find $\Omega_{\pm t}$, $\Gamma_t$, $\bu_\pm$ and $\fp_\pm$ 
which satisfy the periodicity condition:
\begin{equation}\label{eq:p2}
\Omega_{\pm t} = \Omega_{\pm t+2\pi}, \enskip
\Gamma_t = \Gamma_{t+2\pi}, \enskip
\bu_{\pm }(x,t) = \bu_{\pm}(x, t+2\pi), \enskip
\fp_{\pm }(x,t) = \fp_{\pm}(x, t+2\pi).
\end{equation}

To state a main result, we introduce the assumptions about $\Omega_t$ as follows.
We assume that $\Omega \supset B_R$ for some $R > 0$, and that 
\begin{gather} 
\int^{2\pi}_0\Bigl(\int_{\Omega_{+t}} x\,dx\Bigr)\,dt = 0,  \label{assump:6}\\
|\Omega_{+t}| = |B_R| \quad\text{for any $t \in (0, 2\pi)$}. \label{assump:7}
\end{gather}
The following theorem is our main result concerning time-periodic solutions of the two-phase problem 
for the Navier-Stokes equations.
\begin{thm}\label{thm:main2} Let $1 < p, q < \infty$ and $2/p + N/q < 1$.  
$\Omega_+ = B_R$ and $\Omega_- = \Omega\setminus(B_R\cup S_R)$.  Then, there exist 
a positive constant $\epsilon$ and a bijective map $x = \Phi(y, t)$ from $\Omega$ onto itself
such that for any $\bff \in L_{p, {\rm per}}((0, 2\pi), L_q(\Omega)^N)$ satisfying the smallness condition:
$\|\bff\|_{L_p((0, 2\pi), L_q(\Omega))} \leq \epsilon$, there exist $\bv_\pm(y, t)$, $\fq_\pm(y, t)$
and $\rho(y, t)$ with
\begin{equation}\label{reg:2} \begin{aligned}
\bv_\pm &\in L_{p, {\rm per}}((0, 2\pi), H^2_q(\Omega_\pm)^N) \cap 
H^1_{p, {\rm per}}((0, 2\pi), L_q(\Omega_\pm)^N), \\
\fq_\pm & \in L_{p, {\rm per}}((0, 2\pi), H^1_q(\Omega_\pm)), \\
\rho & \in L_{p, {\rm per}}((0, 2\pi), W^{3-1/q}_q(S_R)) \cap H^1_{p, {\rm per}}((0, 2\pi),
W^{2-1/q}_q(S_R))
\end{aligned}\end{equation}
such that 
\begin{gather*}
\Omega_{\pm t} = \{x = \Phi(y, t) \mid y \in \Omega_\pm\}, 
\quad
\bu_\pm(x, t) = \bv_\pm(\Phi^{-1}(x, t), t), \quad \fp_\pm(x, t) = \fq_\pm(\Phi^{-1}(x, t), t), 
\end{gather*}
where $y = \Phi^{-1}(x, y)$ is the inverse map
of $x = \Phi(y, t)$, 
are solutions of problem \eqref{eq:2phase},  and $\Gamma_t$ is given by 
$$\Gamma_t = \{x = y + R^{-1}\rho(y, t) + \xi(t) \mid y \in S_R\},$$
where $\xi(t)$ is the barycenter point of $\Omega_+$ defined by setting  
$$\xi(t) = \frac{1}{|\Omega_{+t}|} \int_{\Omega_{+t}} x\,dx.$$
Moreover, $\bv_\pm$ and $\rho$ satisfy the estimate:
\begin{equation}\label{est:int2}\begin{aligned}
\sum_{\pm}&(\|\bv_\pm\|_{L_p((0, 2\pi), H^2_q(\Omega_\pm))} + 
\|\pd_t\bv_\pm\|_{L_p((0, 2\pi), L_q(\Omega_\pm))} ) \\
&+ \|\rho\|_{L_p((0, 2\pi), W^{3-1/q}_q(S_R))}
+ \|\pd_t\rho\|_{L_p((0, 2\pi), W^{2-1/q}_q(S_R))}
+\|\pd_t \rho\|_{L_\infty((0,2\pi),W^{1-1/q}_q(S_R))} 
\leq C\epsilon
\end{aligned}\end{equation}
for some constant $C$ independent of $\epsilon$. 
\end{thm}

{\bf Method} Since the domain $\Omega_t$ is unknown, using the Hanzawa
transform, we reduce the equations onto a fixed domain, 
which results in a system of quasilinear equations.  Thus, we cannot use the analytic 
$C_0$-semi-group approach.  Our main tool is to use the $L_p$-$L_q$ maximal
regularity for periodic solutions to the linearized equations, which 
can be obtained by using the $\CR$-solver to the generalized resolvent
problem and applying the transference theorem (\cite{L, EKS1})  to the solution formula
represented by the $\CR$-solver. 
This is a quite new and more direct approach and a completely
different idea than exploiting the Poincar\'e operator. 

{\bf Further notation}~ This section is ended by explaining further notation
used in this paper. We denote the sets of all complex numbers, real numbers,
integers, and natural numbers by $\BC$, $\BR$, $\BZ$, and $\BN$, respectively.
Let $\BN_0 = \BN \cup \{0\}$. Let $X$ be a Banach space with norm $\|\cdot\|_X$.  
For any $X$-valued function $f:\BR\to X$ the functions $\CF[f]$ and $\CF^{-1}[f]$ denote
the Fourier transform and the inverse Fourier transform of $f$, respectively, defined by 
setting
$$\CF[f](\tau) = \frac{1}{2\pi}\int_{\BR} e^{-i\tau t}f(t)\,dt, \quad 
\CF^{-1}[f](t) = \int_{\BR} e^{it\tau} f(\tau)\,d\tau.
$$
Let $g:\BT\to X$ be an $X$-valued function defined on the torus $\BT = \BR/ 2\pi\BZ$.
We define the Fourier transform $\CF_\BT$ acting on $g$ by setting
$$\CF_\BT[g](k) = \frac{1}{2\pi}\int^{2\pi}_0 e^{-ikt}g(t)\,dt,$$
which is regarded as a correspondence $g \mapsto (\CF_\BT[g](k))
= \{\CF_\BT[g](k) \in X \mid k \in \BZ\}$.  For any  sequence $(a_k)
= \{a_k \in X \mid k \in \BZ\}$, we define the inverse Fourier transform 
$\CF^{-1}_\BT$ acting on $(a_k)$ by setting
$$\CF^{-1}_\BT[(a_k)](t) = \sum_{k \in \BZ} e^{ikt}a_k.
$$
For any $X$-valued periodic function $f$ with period $2\pi$, we set
$$f_S = \frac{1}{2\pi}\int^{2\pi}_0f(t)\,dt, \quad f_\perp = f - f_S.$$
The $f_S$ and $f_\perp$ are called stationary part and 
oscillatory part of $f$, respectively. 

For $1 \leq p \leq \infty$,  $L_p(\BR, X)$ and 
$H^1_p(\BR, X)$ denote the standard  Lebesgue and Sobolev spaces of
$X$-valued functions defined on $\BR$, 
and $\|\cdot\|_{L_p(\BR, X)}$,  
$\|\cdot\|_{H^1_p(\BR, X)}$ denote their respective norms. 
For $\theta \in (0, 1)$, $H^\theta_{p, {\rm per}}((0,2\pi), X)$ denotes the 
$X$-valued Bessel potential space of periodic functions defined by 
\begin{align*}
H^\theta_{p, {\rm per}}((0, 2\pi), X) &= \{f \in L_{p, {\rm per}}((0, 2\pi), X) \mid 
\|f\|_{H^\theta_p((0, 2\pi), X)} < \infty\}, \\
\|f\|_{H^\theta_p((0, 2\pi), X)} &= \Bigl(\int^{2\pi}_0
\|\CF_\BT^{-1}[(1+k^2)^{\theta/2}\CF_\BT[f](k)](t)\|_X^p\,dt
\Bigr)^{1/p}.
\end{align*}
As usual, we set $L_{p, {\rm per}}((0, 2\pi), X)=H^0_{p, {\rm per}}((0, 2\pi), X)$.

For any multi-index $\alpha = (\alpha_1, \ldots, \alpha_N) 
\in \BN_0^N$ we set 
$\partial_x^\alpha h = 
\partial_1^{\alpha_1}\cdots \partial_N^{\alpha_N} h$
with $\partial_i = \partial/\partial x_i$.  
For any scalar function $f$, we write
\begin{align*}
&\nabla f= (\pd_1f, \ldots, \pd_Nf), \quad 
\bar\nabla f=(f, \pd_1f, \ldots, \pd_Nf),\\
&\nabla^nf = (\pd_x^\alpha f \mid |\alpha|=n), \quad 
\bar\nabla^n f = (\pd_x^\alpha f \mid |\alpha| \leq n)
\quad(n \geq 2),
\end{align*}
where $\pd_x^0f = f$. For any $m$-vector of functions $\bff
={}^\top(f_1, \ldots, f_m)$, we write 
\begin{align*}
&\nabla \bff= (\nabla f_1, \ldots, \nabla f_m), \quad 
\bar\nabla \bff=(\bar\nabla f_1, \ldots, \bar\nabla f_m),\\
&\nabla^n\bff = (\nabla^n f_1, \ldots, \nabla^nf_m), \quad 
\bar\nabla^n\bff = (\bar\nabla^nf_1,\ldots, \bar\nabla^nf_m). 
\end{align*}
For any $N$-vector of functions,
$\bu={}^\top(u_1, \ldots, u_N)$, sometimes $\nabla\bu$ is regarded as
an $(N\times N)$-matrix of functions whose $(i, j)^{\rm th}$ component is
$\pd_ju_i$.
For any $m$-vector $V=(v_1, \ldots, v_m)$ and $n$-vector
$W=(w_1, \ldots, w_n)$, $V\otimes W$ denotes an $(m\times n)$ matrix 
whose $(i, j)^{\rm th}$ component is $V_iW_j$.  For any 
$(mn\times N)$-matrix $A=(A_{ij, k} \mid i=1, \ldots, m, j=1, \ldots, n,
k=1, \ldots, N)$, 
$AV\otimes W$ denotes an $N$-column vector whose $k^{\rm th}$ component is
the quantity: $\sum_{j=1}^m\sum_{j=1}^n A_{ij, k}v_iw_j$.

Let $\ba\cdot \bb =<\ba, \bb>= \sum_{j=1}^Na_jb_j$
for any $N$-vectors $\ba=(a_1, \ldots, a_N)$  and $\bb
=(b_1, \ldots, b_N)$.  
For any $N$-vector $\ba$, let $\Pi_0\ba = \ba_\tau
: = \ba - <\ba, \bn>\bn$.  For any two $(N\times N)$-matrices 
$\bA=(A_{ij})$ and $\bB=(B_{ij})$, the quantity $\bA:\bB$ is defined by
$\bA:\bB 
= \sum_{i,j=1}^NA_{ij}B_{ji}$. 
For any domain $G$ with boundary $\pd G$, 
we set
$$(\bu, \bv)_G = \int_G\bu(x)\cdot\overline{\bv(x)}\,dx, 
\quad (\bu, \bv)_{\pd G} = \int_{\pd G}\bu\cdot\overline{\bv(x)}\,d\sigma,
$$
where $\overline{\bv(x)}$ is the complex conjugate of $\bv(x)$ and 
$d\sigma$ denotes the surface element of $\pd G$. 
Given  
$1 < q < \infty$, let $q' = q/(q-1)$. 
For $L > 0$, let $B_L = \{x \in \BR^N \mid |x| < L\}$ and
$S_L = \{x \in \BR^N \mid |x| = L \}$.

For two Banach spaces $X$ and $Y$, $X+Y = \{x + y \mid x \in X, y\in Y\}$,
$\CL(X, Y)$ denotes the 
set of all bounded linear operators from $X$ into $Y$ and 
$\CL(X, X)$ is written simply as $\CL(X)$. 
Moreover, let 
$\CR_{\CL(X, Y)}(\{\CT(\lambda) \mid \lambda \in I\})$
be the $\CR$-bound of the operator family 
$\{\CT(\lambda) \mid \lambda \in I\}\subset \CL(X, Y)$ (see also Definition \ref{def:4.1}).
Let 
\begin{alignat*}2
i\BR &= \{i\lambda \in \BC \mid \lambda \in \BR\}, 
&\quad 
i\BR_{\lambda_0} &= \{i\lambda \in i\BR \mid |\lambda| \geq \lambda_0\}.
\end{alignat*} 

The letter $C$ denotes a
generic constant and $C_{a,b,c,\ldots}$ denotes that the 
constant $C_{a,b,c,\ldots}$ depends 
on $a$, $b$, $c, \ldots$;
the value of 
$C$ and $C_{a,b,c,\ldots}$ may change from line to line.

\section{Linearization principle}\label{sec:2}
We now formulate the problems \eqref{eq:1p} and \eqref{eq:2phase} in a fixed domain and 
state main results in this setting. 
Theorems \ref{thm:main1} and \ref{thm:main2}  follow from the main theorems of this section.  
\subsection{One-phase problem}
Let $\Omega_t$, $\bu$ and $\fp$ satisfies equations \eqref{eq:1p} and the periodicity condition
\eqref{eq:p1}. We have
\begin{align*}
&((\mu\bD(\bu)-\fp\bI)\bn_t, \be_i)_{\Gamma_t} 
=\sigma(\Delta_{\Gamma_t}x, \be_i)_{\Gamma_t}
= -\sigma(\nabla_{\Gamma_t}x, \nabla_{\Gamma_t}\be_i)_{\Gamma_t}=0;
\\
&((\mu\bD(\bu)-\fp\bI)\bn_t, x_i\be_j-x_j\be_i)_{\Gamma_t} 
=\sigma(\Delta_{\Gamma_t}x, x_i\be_j-x_j\be_i)_{\Gamma_t} \\
&\quad =-\sigma(\nabla_{\Gamma_t}x_j, \nabla_{\Gamma_t}x_i)_{\Gamma_t}
+\sigma (\nabla_{\Gamma_t}x_i, \nabla_{\Gamma_t}x_j)_{\Gamma_t}
=0.
\end{align*}
Multiplying the first equation in \eqref{eq:1p} with $\bp_\ell$ and 
integrating the resultant formula on $\Omega_t$ and using the divergence
theorem of Gauss give that 
$$\frac{d}{dt}(\bu, \bp_\ell)_{\Omega_t} = (\bff, \bp_\ell)_{\Omega_t}.
$$
In fact, we have used the fact that
$$\frac{d}{dt}\int_{\Omega_t}\bu(x, t)\cdot\bp_\ell(x)\,dx
= \int_{\Omega_t}<\pd_t\bu + \bu\cdot\nabla\bu, \bp_\ell>\,dx, 
$$
which follows from the Reynolds transport theorem%
\footnote{For any $f(x, t)$ defined on $\Omega_t$, we have
$$\frac{d}{dt}\int_{\Omega_t}f(x, t)\,dx = \int_{\Omega_t}
(\pd_tf + \dv(f\bu))\,dx,$$
which is called the Reynolds transport theorem.}
 and that $\dv \bu=0$ in
$\Omega_t$. Thus, the periodicity condition \eqref{eq:p1} yields that
\begin{equation}\label{eq:2.1}
\int^{2\pi}_0\Bigl(\int_D \bff(x, \cdot)\cdot\bp_\ell(x)\,dx\Bigr)\,dt = 0
\quad\text{for $\ell=1, \ldots, M$}, 
\end{equation}
where we have used the assumption that ${\rm supp}\, \bff(\cdot, t) 
\subset D$ for any $t \in \BR$. Thus, the condition \eqref{assump:4} is 
a necessary one to prove Theorem \ref{thm:main1}. From this observation,
instead of problem \eqref{eq:p1}, we consider the following 
equations:
\begin{equation}\label{1main:eq}\left\{\begin{aligned}
&\pd_t\bu + \bu\cdot\nabla\bu - \DV(\mu\bD(\bu) - \fp\bI)  
 +\sum_{k=1}^M\int^{2\pi}_0(\bu(\cdot, t), \bp_k)_{\Omega_t}\,
dt \,\bp_k= \bff
&\quad&\text{in $\Omega_t$}, \\
&\qquad \dv \bu  = 0 
&\quad&\text{in $\Omega_t$}, \\
&(\mu\bD(\bu)-\fp\bI)\bn_t  = \sigma H(\Gamma_t)\bn_t 
&\quad&\text{on $\Gamma_t$}, \\
&\qquad V_{\Gamma_t} = \bu\cdot\bn_t
&\quad&\text{on $\Gamma_t$}
\end{aligned} \right.\end{equation}
for $t \in \BR$.  In fact, if $\Omega_t$, $\bu$ and $\fp$ 
satisfy equations \eqref{1main:eq}, then we have
$$\frac{d}{dt}(\bu(\cdot, t),\bp_\ell)_{\Omega_t} + 
\sum_{k=1}^M\int^{2\pi}_0(\bu(\cdot, t), \bp_k)_{\Omega_t}\,dt
(\bp_k, \bp_\ell)_{\Omega_t}
=(\bff, \bp_\ell)_{\Omega_t},
$$
which, combined with the periodicity condition \eqref{eq:p1}, the assumption \eqref{assump:1} and
\eqref{eq:2.1}, leads to 
$$
\int^{2\pi}_0(\bu(\cdot, t), \bp_k)_{\Omega_t}\,dt= 0
\quad\text{for $k=1, \ldots, M$}.
$$
Thus, $\Omega_t$, $\bu$ and $\fp$ satisfy the first equation in
\eqref{eq:1p}. 
Therefore, under the stated assumptions, a solution to problem \eqref{1main:eq} 
is a solution to the original problem \eqref{eq:1p}.
However, as we shall see below, the condition \eqref{eq:2.1} is not necessary to find a solution 
to \eqref{1main:eq}.

From now on, we consider problem \eqref{1main:eq}.  We reduce problem
\eqref{1main:eq} to some nonlinear equations on $B_R$ by using the 
Hanzawa transform, which we explain below.
Let $\xi(t)$ be the barycenter point of $\Omega_t$ defined by setting
\begin{equation}\label{bary:1}
\xi(t)= \frac{1}{|B_R|}\int_{\Omega_t} x\,dx,
\end{equation}
where we have used the fact that $|\Omega_t| = |B_R|$, which follows from 
the assumption \eqref{assump:5}. By the Reynolds transport theorem, we see
that
\begin{equation}\label{bary:2}
\frac{d}{dt} \xi(t) = \frac{1}{|B_R|} \int_{\Omega_t} (\pd_tx + \bu\cdot\nabla x)\,dx
= \frac{1}{|B_R|} \int_{\Omega_t} \bu(x, t)\,dx
\end{equation}
because $\dv\bu=0$. Let $\rho(y, t)$ be an unknown time-periodic function with
period $2\pi$ such that 
$$\Gamma_t = \{x = y + \rho(y, t)\bn + \xi(t) \mid y \in S_R\},
$$
where $S_R = \{x \in \BR^N \mid |x| = R\}$ and $\bn$ is the unit outer normal
to $S_R$, that is $\bn = x/|x|$ for $x \in S_R$.  Let $H_\rho$ be a suitable 
extension of $\rho$ to $\BR^N$, and then by the $K$-method
in the theory of real interpolation \cite{Lu, Tri},  we see that there exist
constants $C_1$ and $C_2$ such that 
\begin{gather}\label{ext:1}
C_1\|H_\rho(\cdot, t)\|_{H^k_q(\BR^N)} \leq \|\rho(\cdot, t)\|_{W^{k-1/q}_q(S_R)}
\leq C_2\|H_\rho(\cdot, t)\|_{H^k_q(\BR^N)} 
\quad\text{for $k=1, 2, 3$},  \nonumber\\
C_1\|\pd_t H_\rho(\cdot, t)\|_{H^k_q(\BR^N)} \leq \|\pd_t\rho(\cdot, t)\|_{W^{k-1/q}_q(S_R)}
\leq C_2\|\pd_tH_\rho(\cdot, t)\|_{H^k_q(\BR^N)} 
\quad\text{for $k=1, 2$}, 
\label{ext:1} 
\end{gather}
for any $t \in(0, 2\pi)$.  In the following, we fix the method of this extension.  For example, 
$\hat H_\rho$ is the unique solution of the Dirichlet problem:
$$(1-\Delta)\hat H_\rho = 0 \quad\text{in $\BR^N\setminus S_R$}, \quad \hat H_\rho|_{S_R} = \rho.
$$
Let $\varphi$ be a $C^\infty(\BR^N)$ function which equals one for $x \in B_{2R}$ and
zero for $x \not\in B_{3R}$, and we set $H_\rho = \varphi \hat H_\rho$. 
We assume that 
\begin{equation}\label{small:1}
\sup_{t \in \BR} \|\nabla H_\rho(\cdot, t)\|_{H^1_\infty(\BR^N)} \leq \delta
\end{equation}
with some small constant $\delta > 0$.  Notice that 
$y/|y| = R^{-1}y$ for $y \in S_R$  is the unit outer normal to $S_R$.
Let $\Phi(y, t) = y + R^{-1}H_\rho(y, t)y+ \xi(t)$.  
We choose $\delta > 0$ so small that 
the map $x = \Phi(y,t)$ is injective.  In fact, for any $y_1$ and $y_2$ 
$$|\Phi(y_1, t) - \Phi(y_2, t)| \geq |y_1-y_2| - \sup_{t \in \BR}\|\nabla H_\rho(\cdot, t)\|_{H^1_\infty(\BR^N)}
|y_1-y_2| \geq (1-\delta)|y_1-y_2|,
$$
which leads to the injectivity of the transformation $x = \Phi(y, t)$ for any $t \in \BR$ provided
that $0 < \delta < 1$.
Moreover, using the inverse mapping theorem, 
we see that the map $x=\Phi(y, t)$ is surjective from $\BR^N$ onto $\BR^N$.

Let 
\begin{equation}\label{domain:1}\begin{aligned}
\Omega_t & = \{x = y+ R^{-1}H_\rho(y, t)y + \xi(t) \mid y \in B_R\}, \\
\Gamma_t & = \{x = y + R^{-1}\rho(y, t)y + \xi(t) \mid y \in S_R\}.
\end{aligned}\end{equation}
Let $\bu(x, t)$ and $\fp(x, t)$ satisfy equations \eqref{eq:1p}, and let 
$\bv(y, t) = \bu(x, t)$ and $\fq(y, t) = \fp(x, t)$.  
We  derive an equation for $\bv$ and $\rho$ from the kinematic condition: 
$V_{\Gamma_t} = \bu\cdot\bn_t$ on
$\Gamma_t$. From the definition: 
$$V_{\Gamma_t} = \frac{\pd x}{\pd t}\cdot\bn_t = (\frac{\pd \rho}{\pd t} \bn + \xi'(t))\cdot\bn_t.
$$
To represent $\xi'(t)$, we introduce the Jacobian $J(t)$ of the transformation 
$x = \Phi(y, t)$, which is written as
$J(t) = 1 + J_0(t)$ with
$$J_0(t) = \det\bigl(\delta_{ij} +R^{-1} \frac{\pd}{\pd y_i}(H_\rho(y, t)y_j)\bigr)_{i,j=1, \ldots, N} - 1.$$
Choosing $\delta > 0$ small enough in \eqref{small:1}, we have
\begin{equation}\label{jacob:1}\begin{aligned}
|J_0(t)|& \leq C\|\nabla H_\rho(\cdot, t)\|_{L_\infty(B_R)}.
\end{aligned}\end{equation}
From \eqref{bary:2} it follows that
\begin{equation}\label{bary:3}
\xi'(t) = \frac{1}{|B_R|}\int_{B_R} \bv(y, t)\,dy + \frac{1}{|B_R|}\int_{B_R} \bv(y, t)J_0(t)\,dy,
\end{equation}
and so noting that $\bn\cdot\bn=1$, we have the kinematic equation: 
\begin{equation}\label{kin:1}
\pd_t\rho - (\bv - \frac{1}{|B_R|}\int_{B_R}\bv(y, t)\,dy)\cdot\bn = d(\bv, \rho)
\end{equation}
with
\begin{equation}\label{kin:1*}
d(\bv, \rho) = \frac{1}{|B_R|}\int_{B_R}\bv(y, t)J_0(t)\,dy \cdot(\bn-\bn_t)+ 
\frac{\pd \rho}{\pd t}\bn\cdot(\bn-\bn_t) + \bv\cdot(\bn_t-\bn).
\end{equation}
As will be seen in Sect. \ref{sec:3},  we have $<H(\Gamma_t)\bn_t, \bn_t>
= (\Delta_{S_R} + (N-1)/R^2)\rho -(N-1)/R +$  nonlinear terms,  and
$-(N-1)/R^2$ is the first eigen-value of the Laplace-Beltrami operator
$\Delta_{S_R}$ on $S_R$ with eigen-functions $y_j/R$ for $y=(y_1,
\ldots, y_N) \in S_R$. We need to derive some auxiliary equations
to avoid the zero and first eigen-values of $\Delta_{S_R}$. 
From the assumption \eqref{assump:5} and the representation formulas 
of $\Omega_t$ and $\Gamma_t$ in \eqref{domain:1}, by using polar
coordinates we have
\begin{align*}
|B_R| &= |\Omega_t|=\int_{S_R}\Bigl(\int^{1+R^{-1}\rho(\omega, t)}_0 r^{N-1}\,dr\Bigr)\,d\omega
= \frac{1}{N}\int_{S_R} (1+R^{-1}\rho(\omega,t))^N\,d\omega  \\
&= |B_R| + R^{-1}\int_{S_R} \rho \,d\omega + \sum_{k=2}^N\frac{{}_NC_k}{N}
R^{-k}\int_{S_R} \rho^k\,d\omega, 
\end{align*}
and so we have
\begin{equation}\label{eigen:1}
\int_{S_R}\rho\,d\omega + \sum_{k=2}^N\frac{{}_NC_k}{N}R^{1-k}\int_{S_R}\rho^k\,d\omega=0
\end{equation}
where $d\omega$ denotes the surface element of $S_R$.  Moreover, from \eqref{bary:1}
and the assumption \eqref{assump:5}, 
using polar coordinates centered at $\xi(t)$, we have
\begin{align*}
0 &= \frac{1}{|B_R|}\int_{\Omega_t}(x-\xi(t))\,dx = \frac{1}{|B_R|}
\int_{S_R}\Bigl(\int^{1+R^{-1}\rho(\omega, t)}_0 r^N\omega\,dr\Bigr)\,d\omega \\
&= \frac{1}{|B_R|}\frac{1}{N+1}\int_{S_R}(1+R^{-1}\rho(\omega, t))^{N+1}\omega\,d\omega
 = \frac{1}{|B_R|}\Bigl(R^{-1}\int_{S_R} \rho\omega\,d\omega 
+ \sum_{k=2}^{N+1}\frac{{}_{N+1}C_k}{N+1}R^{-k}\int_{S_R}\rho^k
\omega\,d\omega\Bigr),
\end{align*}
from which it follows that
\begin{equation}\label{eigen:2}
\int_{S_R} \rho\omega_j\,d\omega + 
\sum_{k=2}^{N+1}\frac{{}_{N+1}C_k}{N+1}R^{1-k}\int_{S_R} \rho^k\omega_j\,d\omega
=0
\end{equation}
for $j=1, \ldots, N$.  Thus, under the assumption \eqref{assump:5}
and the representation of $\Gamma_t$ and $\Omega_t$ in \eqref{domain:1}, 
the kinematic condition \eqref{kin:1} is equivalent to the equation
\begin{equation}\label{kin:2}
\pd_t\rho + \int_{S_R} \rho\,d\omega
+ \sum_{k=1}^N \Bigl(\int_{S_R}\rho \omega_k\,d\omega \Bigr)y_k
-\Bigl(\bv - \frac{1}{|B_R|}\int_{B_R}\bv\,dy\Bigr)\cdot\bn
= \tilde d(\bv, \rho)
\quad\text{on $S_R\times(0, 2\pi)$}
\end{equation}
with
\begin{equation}\label{kin:3}\tilde d(\bv, \rho) = d(\bv, \rho) 
- \sum_{k=2}^N\frac{{}_NC_k}{N}R^{1-k}\int_{S_R}\rho^k\,d\omega 
- \sum_{k=2}^{N+1}\frac{{}_{N+1}C_k}{N+1}R^{1-k}\Bigl(\int_{S_R}\rho^k
\omega\,d\omega\Bigr)y_k.
\end{equation}
Therefore, to prove the existence of $(\Omega_t, \bu, \fp)$, we shall 
prove the well-posedness of the following equations:
\begin{equation}\label{main:eq1}\left\{\begin{aligned}
&\pd_t\bv + \CL\bv_S- \DV(\mu(\bD(\bv) - \fq\bI) = \bG + \bF(\bv, \rho)
&\quad&\text{in $B_R\times(0, 2\pi)$}, \\
&\dv \bv = g(\bv, \rho) = \dv \bg(\bv, \rho)
&\quad&\text{in $B_R\times(0, 2\pi)$}, \\
&\pd_t\rho + \CM\rho 
-\CA\bv\cdot\bn
= \tilde d(\bv, \rho) 
&\quad&\text{on $S_R\times(0, 2\pi)$}, \\
&(\mu\bD(\bv)-\fq)\bn - (\CB_R\rho) \bn
= \bh(\bv, \rho)
&\quad&\text{on $S_R\times(0, 2\pi)$},
\end{aligned}\right.\end{equation}
where we have set 
\allowdisplaybreaks{
\begin{equation}\label{op:1}\begin{aligned}
\CL\bv_S &= \sum_{k=1}^M(\bv_S, \bp_k)_{B_R}\,\bp_k; \quad 
\CA\bv  = \bv - \frac{1}{|B_R|}\int_{B_R}\bv\,dy; \\
\CM\rho & = \int_{S_R} \rho\,d\omega
+ \sum_{k=1}^N \Bigl(\int_{S_R}\rho \omega_k\,d\omega \Bigr)y_k; \\
\CB_R\rho &= (\Delta_{S_R}+\frac{N-1}{R^2})\rho
=R^{-2}(\Delta_{S_1} +(N-1))\rho,
\end{aligned}\end{equation}}%
where $\Delta_{S_1}$ is the Laplace-Beltrami operator on the unit 
sphere $S_1$.  For the functions on the right side of equations \eqref{main:eq1},
$\bG(y, t)$ and $\bF(\bv, \rho)$ are given in \eqref{right:1} in Sect. \ref{sec:3}
below,  
$g(\bv, \rho)$ and $\bg(\bv, \rho)$ given in \eqref{form:g} in
Sect. \ref{sec:3} below, 
$\tilde d(\bv, \rho)$ has been given in \eqref{kin:3} and 
$\bh(\bv, \rho) = (\bh'(\bv, \rho), h_N(\bv, \rho))$ is given 
in \eqref{58} and \eqref{non:g.5} in Sect. \ref{sec:3} below. 

The following theorem is the unique existence theorem of 
$2\pi$-periodic solutions of problem \eqref{main:eq1}.
\begin{thm}\label{main:thm1} Let $1 < p, q < \infty$ and $2/p+N/q < 1$.
Then, there exists a small constant $\epsilon > 0$ such that if 
$\bff$ satisfies the assumption \eqref{assump:4} and the smallness
condition: $\|\bff\|_{L_p((0, 2\pi), L_q(D))} \leq \epsilon$, then 
problem \eqref{main:eq1} admits $2\pi$-periodic
solutions $\bv$, $\fq$,  and $\rho$ satisfying the regularity condition
\eqref{reg:1} and the estimate \eqref{est:int1} in Theorem \ref{thm:main1}.
\end{thm}

{\bf Proof of Theorem \ref{thm:main1}} ~ We prove Theorem \ref{thm:main1} with the help
of Theorem \ref{main:thm1}. 
Let $\xi(t)$ be defined by 
$$\xi(t) = \int^t_0\xi'(s)\,ds + c$$
where $c$ is chosen in such a way that
\begin{equation}\label{compati:1}
\int^{2\pi}_0 \xi(s)\,ds = 0.
\end{equation}
Here, $\xi'(t)$ is given by the formula in \eqref{bary:3}.  Then, 
 we define $\Omega_t$ and $\Gamma_t$ by the formulas
in \eqref{domain:1}.  Let   $\Phi(y, t)=y+ R^{-1}H_\rho y+ \xi(t)$. 
By choosing $\epsilon$ sufficiently small, estimates \eqref{est:int1}
and \eqref{ext:1} ensure that the condition \eqref{small:1} is satisfied
with small $\delta > 0$. This yields the existence of the inverse map $y= \Phi^{-1}(x, t)$ of the 
map: $x = \Phi(y, t)$. Thus, the velocity field $\bu(x, t)$ and the pressure $\fp(x, t)$ 
on $\Omega_t$ are well-defined
by setting $\bu(x, t) = \bv(y, t)$ and $\fp(x, t) = \fq(y,t)$.  Since $\dv \bu=0$ in $\Omega_t$, 
$|\Omega_t|$ is a constant, and so $|\Omega_t| = |B_R|$ by assumption \eqref{assump:5}.  
Moreover, if we set
$$\eta(t) = \frac{1}{|B_R|}\int_{\Omega_t} x\,dx,$$
then 
$$\eta'(t) = \frac{1}{|B_R|}\int_{\Omega_t} \bu(x, t)\,dx = \xi'(t),
$$
and so $\eta(t) = \xi(t) + d$ with some constant $d$.  We assume that the assumption 
\eqref{assump:3} holds, and then by \eqref{compati:1} we have 
$$0 = \int^{2\pi}_0 \eta(t)\,dt = 2\pi d + \int^{2\pi}_0\xi(t)\,dt = 2\pi d,
$$
which leads to $d=0$, that is 
$$\xi(t) = \frac{1}{|B_R|} \int_{\Omega_t} x\,dx.
$$
Combining this with \eqref{assump:5} gives that
$$\int_{S_R}(R+\rho)^N\,d\omega = 0, \quad \int_{S_R}(R+\rho)^{N+1}\,d\omega=0,
$$
which yields that $\rho$ satisfies the equation:
$$\pd_t\rho - \CA\bv\cdot\bn = d(\bv, \rho) \quad\text{on $S_R$}.
$$
Therefore, the kinematic equation: $V_{\Gamma_t} = \bu\cdot\bn_t$ holds on $\Gamma_t$. 
So far, we see that $\Omega_t$, $\bu$ and $\fp$ satisfy equations \eqref{1main:eq}. 
Since $D \subset B_R$, there exists a constant $\epsilon_0 > 0$ for which
$D \subset B_{R-3\epsilon_0}$.  Since $\Omega_t$ is a small perturbation of
$B_R$, choosing $\epsilon > 0$ smaller if necessary, we may assume that
$B_{R-\epsilon_0} \subset \Omega_t$, and so by \eqref{assump:4} we have
\begin{equation}\label{assump:4*}
\int^{2\pi}_0(\bff(\cdot, t), \bp_\ell)_{\Omega_t}\,dt = 0 \quad\text{for $\ell=1, \ldots, N$}.
\end{equation}
Multiplying the first equation in \eqref{1main:eq} with $\bp_\ell$,
integrating the resultant formulas with respect to $x$  on $\Omega_t$ and
with respect to $t$ on $(0, 2\pi)$,   and 
using the periodicity \eqref{eq:p1} and \eqref{assump:4*}
we have
\begin{equation}\label{assump:4**}
\sum_{k=1}^M \int^{2\pi}_0(\bu(\cdot, t), \bp_k)_{\Omega_t}\,dt\int^{2\pi}_0(\bp_k, \bp_\ell)_{\Omega_t}
= \int^{2\pi}_0(\bff(\cdot, t),  \bp_\ell)_{\Omega_t}\,dt =0
\end{equation}
for $\ell=1, \ldots, M$.   Since $\Omega_t$ is
a small perturbation of $B_R$, we may assume that the assumption \eqref{assump:1}
holds, and so by \eqref{assump:4**} we have
$$\int^{2\pi}_0(\bu(\cdot, t), \bp_\ell)_{\Omega_t}\,dt = 0
\quad\text{for $\ell=1, \ldots, M$}. $$
Therefore, $\Omega_t$, $\bu$ and $\fp$ satisfy  equations \eqref{eq:1p},
and so we see that
Theorem \ref{thm:main1} follows immediately from 
Theorem \ref{main:thm1}. 


\subsection{Two-phase problem}

We now formulate problem \eqref{eq:2phase} in the fixed domain. 
The idea is essentially the same as in the one-phase case. 
Let $\dot\Omega = \Omega\setminus S_R$, $\Omega_+=B_R$ and $\Omega_-=\Omega\setminus\overline{B_R}$. 
We define the barycenter point, $\xi(t)$,
of $\Omega_{+t}$  by setting
\begin{equation}\label{bary:21}
\xi(t)= \frac{1}{|B_R|}\int_{\Omega_{+t}} x\,dx,
\end{equation}
where we have used the fact that $|\Omega_{+t}| = |B_R|$, which follows from 
the assumption \eqref{assump:7}. By the Reynolds transport theorem, we see
that
\begin{equation}\label{bary:22}
\frac{d}{dt} \xi(t) 
= \frac{1}{|B_R|} \int_{\Omega_t} \bu(x, t)\,dx.
\end{equation}
 Let $\rho(y, t)$ be an unknown periodic function with
period $2\pi$ such that 
$$\Gamma_t = \{x = y + \rho(y, t)\bn + \xi(t) \mid y \in S_R\},
$$
where $S_R = \{x \in \BR^N \mid |x| = R\}$ and $\bn$ is the unit outer normal
to $S_R$, that is $\bn = y/|y|$ for $y \in S_R$.  

In the following, we fix the method how to extend this to a transformation 
from $\dot\Omega$ to $\Omega_t$.  
Let $H$ be 
 a unique solution of the Dirichlet problem:
$$(1-\Delta) H_\rho = 0 \quad\text{in $\BR^N\setminus S_R$}, \quad H_\rho|_{S_R} = \rho.
$$
Let $L$ be a large number for which $\Omega \subset B_L$. 
From the $K$-method in real interpolation theory \cite{Lu, Tri}, we see that 
\begin{gather}
C_1\|H_\rho(\cdot, t)\|_{H^k_q(\BR^N)} \leq \|\rho(\cdot, t)\|_{W^{k-1/q}_q(S_R)}
\leq C_2\|H_\rho(\cdot, t)\|_{H^k_q(\BR^N)} 
\quad\text{for $k=1, 2, 3$},  \nonumber\\
C_1\|\pd_t H_\rho(\cdot, t)\|_{H^k_q(\BR^N)} \leq \|\pd_t\rho(\cdot, t)\|_{W^{k-1/q}_q(S_R)}
\leq C_2\|\pd_tH_\rho(\cdot, t)\|_{H^k_q(\BR^N)} 
\quad\text{for $k=1, 2$}, \label{ext:21}
\end{gather}
for any $t \in(0, 2\pi)$.  
We may assume that there exists
a small number $\omega > 0$ for which $B_{R+3\omega} \subset \Omega$. 
Let $\varphi$ be a function in $C^\infty(\BR^N)$ for which equals one
for $x \in B_{R+\omega}$ and zero for $x\not\in B_{R+2\omega}$. 
Let $\Phi(y, t) = y+ \varphi(y)(R^{-1}H_\rho(y, t)y+ \xi(t))$.  Notice that  
 $\Phi(y, t) = y + R^{-1}H_\rho(y, t)y + \xi(t)$ for $y \in B_R$. 
Setting $\Psi(y, t) = \varphi(y)(R^{-1}H_\rho(y, t)y+ \xi(t))$, we assume that 
\begin{equation}\label{small:21}
\sup_{t \in \BR} \|\Psi(\cdot, t)\|_{H^1_\infty(\BR^N)} \leq \delta
\end{equation}
with some small constant $\delta > 0$.  We choose $\delta > 0$ so small that 
the map: $y \mapsto x= \Phi(y, t)$ is bijective from $\Omega$ onto itself. 
In fact, 
for any $y_1$ and $y_2$ 
$$|\Phi(y_1, t) - \Phi(y_2, t)| \geq |y_1-y_2| - \sup_{t\in \BR}\|\nabla \Psi(\cdot, t)\|_{H^1_\infty(\BR^N)}
|y_1-y_2| \geq (1-\delta)|y_1-y_2|,
$$
which leads to the injectivity of the map: $x = \Phi(y, t)$ for any $t \in \BR$ provided
that $0 < \delta < 1$.  Moreover,  using the fact that $x = \Phi(y, t) = y$ 
for $y \in \Omega\setminus B_{R+2\omega}$, 
and the inverse mapping theorem, 
we  see that the map $x=\Phi(y, t)$ is surjective from $\Omega$ onto itself. 
Let 
\begin{equation}\label{domain:2}\begin{aligned}
\Omega_{+t} & = \{x = \Phi(y, t) = y+ R^{-1}H_\rho(y, t)y + \xi(t) \mid y \in B_R\}, \\
\Omega_{-t} & = \{x = \Phi(y, t) = y + \varphi(y)(R^{-1}H_\rho(y, t)y + \xi(t))
\mid y \in \Omega\setminus (S_R\cup B_R)\}, \\
\Gamma_t & = \{x = y + R^{-1}\rho(y, t)y + \xi(t) \mid y \in S_R\},
\end{aligned}\end{equation}
Notice that  $R^{-1}y$ is the unit outer normal to $S_R$ for  $y \in S_R$. 
In the following, 
the jump quantity of $f$ defined on $\Omega\setminus S_R$ is also denoted by
$[[f]]$, which is defined by setting
$$[[f]](x_0,t) = \lim_{y\to x_0 \atop y \in \Omega_+} f(y, t) 
- \lim_{y\to x_0 \atop y \in \Omega_-}f(y,t)\quad\text{for $x_0 \in S_R$}, 
$$
where we have set $\Omega_+ = B_R$ and $\Omega_- = \Omega\setminus(B_R \cup S_R)$. 
Let $\dot\Omega = \Omega_+ \cup \Omega_-$, and  for $f$ defined on $\dot\Omega$,
we write $f_\pm = f|_{\Omega_\pm}$.  On the other hand,  for $f_\pm$ defined on 
$\Omega_\pm$, we define $f$ by 
$f|_{\Omega_\pm} = f_\pm$.

Let $\bu(x, t)$ and $\fp(x, t)$ satisfy the equations \eqref{eq:2phase}, and let 
$\Phi^{-1}(x, t)$ be the inverse map of $x = \Phi(y, t)$.  Let 
$\bv_\pm(y, t) = \bu_\pm(\Phi^{-1}(y, t), t)$ and $\fq_\pm(y, t) = \fp_\pm(\Phi^{-1}(y, t), t)$
for $y \in \Omega_{\pm t}$.  
We  derive an equation  for $\bv_+$ and $\rho$ from the kinematic condition
 $V_{\Gamma_t} = \bu\cdot\bn_t$ on
$\Gamma_t$. Noting that $[[\bu]] = 0$ on $\Gamma_t$, we may also assume that
$[[\bv]]=0$ on $S_R$, and so $\bv_+ = \bv_-$ on $S_R$.

From the definition it follows that 
$$V_{\Gamma_t} = \frac{\pd x}{\pd t}\cdot\bn_t = (\frac{\pd \rho}{\pd t} \bn + \xi'(t))\cdot\bn_t,
$$
Here and in the following, the unit outer normal to $S_R$ is denoted by 
$\bn$, which is given by $\bn(y) = R^{-1}y$ for $y \in S_R$. 
To represent the time derivative of $\xi(t)$ given in \eqref{bary:21}, 
we introduce the Jacobian $J_+(t)$ of the transformation: 
$x = y + R^{-1}H_\rho y + \xi(t)$ for $y \in B_R$, which is written as
$J_+(t) = 1 + J_{0,+}(t)$ with
$$J_{0,+}(t) = \det\bigl(\delta_{ij} +R^{-1} \frac{\pd}{\pd y_i}(H_\rho(y, t)y_j)\bigr)_{i,j=1, \ldots, N} - 1
\quad\text{for $y \in B_R$}. $$
Choosing $\delta > 0$ small enough in \eqref{small:21}, we have
\begin{equation}\label{jacob:21}\begin{aligned}
\|J_{0,+}(t)\|_{L_\infty(B_R)} & \leq C\|\nabla H_\rho(\cdot, t)\|_{L_\infty(B_R)}.
\end{aligned}\end{equation} 
From \eqref{bary:21} it follows that
\begin{equation}\label{bary:23}
\xi'(t) = \frac{1}{|B_R|}\int_{B_R} \bv_+(y, t)\,dy + \frac{1}{|B_R|}\int_{B_R} \bv_+(y, t)J_{0,+}(t)\,dy,
\end{equation}
and noting that $\bn\cdot\bn=1$, on $S_R$  we have the kinematic equation: 
\begin{equation}\label{kin:21}
\pd_t\rho - (\bv - \frac{1}{|B_R|}\int_{B_R}\bv_+(y, t)\,dy)\cdot\bn = d(\bv_+, \rho)
\end{equation}
with
$$d(\bv_+, \rho) = \frac{1}{|B_R|}\int_{B_R}\bv_+(y, t)J_{0,+}(t)\,dy \cdot(\bn-\bn_t) + 
\frac{\pd \rho}{\pd t}\bn\cdot(\bn-\bn_t) + \bv_+\cdot(\bn_t-\bn).
$$
As was already discussed in Subsec.~2.1,  from  the assumption \eqref{assump:7} and the representation formulas 
of $\Omega_{+t}$ and $\Gamma_t$ in \eqref{domain:2},  we have \eqref{eigen:1} in Subsec.~2.1, too. 
 Moreover, from \eqref{bary:21}
and the assumption \eqref{assump:7},  we have \eqref{eigen:2} in Subsec.~2.1, 
too.  Thus, under the assumption \eqref{assump:7}
and the representation of $\Gamma_t$ and $\Omega_{+t}$ in \eqref{domain:2}, 
the kinematic condition is equivalent to the equation:
\begin{equation}\label{kin:2.22}
\pd_t\rho + \int_{S_R} \rho\,d\omega
+ \sum_{k=1}^N \Bigl(\int_{S_R}\rho \omega_k\,d\omega \Bigr)y_k
-\Bigl(\bv_+ - \frac{1}{|B_R|}\int_{B_R}\bv_+\,dy\Bigr)\cdot\bn
= \tilde d(\bv_+, \rho)
\quad\text{on $S_R\times(0, 2\pi)$}
\end{equation}
with
\begin{equation}\label{kin:3.22}\tilde d(\bv_+, \rho) = d(\bv_+, \rho) 
- \sum_{k=2}^N\frac{{}_NC_k}{N}\int_{S_R}R^{1-k}\rho^k\,d\omega 
- \sum_{k=2}^{N+1}\frac{{}_{N+1}C_k}{N+1}R^{1-k}\Bigl(\int_{S_R}\rho^k
\omega\,d\omega\Bigr)y_k.
\end{equation}
And then, to prove Theorem \ref{thm:main2}, we shall 
prove the global well-posedness of the following equations:
\begin{equation}\label{main:eq2}\left\{\begin{aligned}
&\pd_t\bv_\pm - \DV(\mu_\pm(\bD(\bv_\pm) - \fq_\pm) = \bG_\pm + \bF_\pm(\bv, \rho)
&\quad&\text{in $\Omega_\pm\times(0, 2\pi)$}, \\
&\dv \bv_\pm = g_\pm(\bv, \rho) = \dv \bg_\pm(\bv, \rho)
&\quad&\text{in $\Omega_\pm \times(0, 2\pi)$}, \\
&\pd_t\rho + \CM\rho 
-\CA\bv_+\cdot\bn
= \tilde d(\bv_+, \rho)
&\quad&\text{on $S_R\times(0, 2\pi)$}, \\
&[[\mu_\pm\bD(\bv_\pm)-\fq_\pm]]\bn - (\CB_R\rho) \bn
= \tilde\bh(\bv, \rho)
&\quad&\text{on $S_R\times(0, 2\pi)$}, \\
&[[\bv]]= 0 &\quad&\text{on $S_R\times(0, 2\pi)$}, \\
&\bv_- = 0 &\quad&\text{on $S\times(0, 2\pi)$}, 
\end{aligned}\right.\end{equation}
where we have set 
\allowdisplaybreaks{
\begin{equation}\label{op:2}
\CA\bv_+  = \bv_+ - \frac{1}{|B_R|}\int_{B_R}\bv_+\,dy
\end{equation}
and 
$\CM\rho$ and $\CB_R\rho$ are the same as in \eqref{op:1} in 
Subsec.~2.1.  For the functions on the right side of equations \eqref{main:eq2},
$\bG_\pm$ and $\bF_\pm(\bv, \rho)$ are defined in \eqref{right:1*} 
of Sect.~\ref{sec:3} below, $g_\pm(\bv, \rho)$ and $\bg_\pm(\bv, \rho)$
are defined in \eqref{form:g*} of Sect.~\ref{sec:3} below, 
and $\tilde\bh(\bv, \rho)$ is defined in \eqref{3.58} 
of Sect.~\ref{sec:3} below. 

The following theorem is the unique existence theorem of 
$2\pi$-periodic solutions of problem \eqref{main:eq2}.
\begin{thm}\label{main:thm2} Let $1 < p, q < \infty$ and $2/p+N/q < 1$.
Then, there exists a small constant $\epsilon > 0$ such that 
for any $\bff \in L_{p, {\rm per}}((0, 2\pi), L_q(\Omega)^N)$
satisfying the smallness condition: $\|\bff\|_{L_p((0, 2\pi), L_q(\Omega))}
\leq \epsilon$, problem \eqref{main:eq2} admits solutions 
$\bv_\pm$, $\fq_\pm$, and $\rho$ satisfying the regularity condition
\eqref{reg:2} and the estimate \eqref{est:int2} in Theorem \ref{thm:main2}.
\end{thm}

Employing the same argument as in the proof of Theorem \ref{thm:main1}
in  Subsec.~2.1, we see that Theorem \ref{thm:main2} immediately 
follows from Theorem \ref{main:thm2}.


\section{Derivation of nonlinear terms}\label{sec:3}
\subsection{One-phase problem case} 
First, we consider the one-phase problem case 
and  we consider the map
\begin{equation}\label{map:1} x = y + \Psi(y, t),
\end{equation}
where
$\Psi(y, t)= R^{-1}H_\rho(y,t) y + \xi(t)$ and $H_\rho$ satisfies the condition
\eqref{ext:1} and 
\eqref{small:1}. Recall that $H_\rho(y, t) = \rho(y, t)$ for $y \in S_R$. 
Let $\Omega_t$, $\Gamma_t$, $\bu(x, t)$ and $\fp(x, t)$ satisfy the equations
\eqref{eq:1p} and 
$$\Omega_t = \{x = y + \Psi(y, t) \mid y \in B_R\}, \quad
\Gamma_t = \{x = y + R^{-1}\rho(y, t)y + \xi(t) \mid y \in S_R\}.
$$
Choose $\delta > 0$ small in such a way that there exists an inverse
map: $y = \Phi^{-1}(x, t)$ of the map: $x = \Phi(y,t)=y+\Psi(y, t)$. 
Let $\bv(y, t) = \bu(\Phi^{-1}(y, t), t)$ and $\fq(y, t) = 
\fp(\Phi^{-1}(y, t), t)$.  
By the chain rule, we have 
\begin{equation}\label{change:1}
\nabla_x = (\bI + \bV_0(\bk))\nabla_y, 
\quad \frac{\pd}{\pd x_i} = \frac{\pd}{\pd y_i} + \sum_{j=1}^N V_{0ij}(\bk)\frac{\pd}{\pd y_j}
\end{equation}
where $\nabla_z = {}^\top(\pd/\pd z_1, \ldots, \pd/\pd z_N)$ for $z \in \{x, y\}$
and $\bk =(k_0, k_1, \ldots, k_N) =  (H_\rho, \nabla H_\rho)$.  Here, 
$\bV_0(\bk)$ is an $(N\times N)$-matrix of $C^\infty$ functions defined for $|\bk| \leq \delta$
with $\bV_0(0) = 0$ and $V_{0ij}(\bk)$ is the $(i, j)^{\rm th}$ component of 
$\bV_0(\bk)$. 
By \eqref{change:1}, we can write $\bD(\bu)$ as  
$\bD(\bu) = \bD(\bv) + \CD_\bD(\bk)\nabla\bv$ with 
\begin{equation} \label{div:1}\begin{split}
\bD(\bv)_{ij} &= \frac{\pd v_i}{\pd y_j} + \frac{\pd v_j}{\pd y_i},
\\
(\CD_\bD(\bk)\nabla\bv)_{ij} &= \sum_{k=1}^N
\Bigl(V_{0jk}(\bk)\frac{\pd v_i}{\pd y_k}
+ V_{0ik}(\bk)\frac{\pd v_j}{\pd y_k}\Bigr).
\end{split}\end{equation}
We next consider $\dv\bv$.  By \eqref{change:1}, we have
\begin{equation}\label{div:1-1}\dv_x\bu = \sum_{j=1}^N\frac{\pd u_j}{\pd x_j}
= \sum_{j,k=1}^N(\delta_{jk} + V_{0jk}(\bk))\frac{\pd v_j}{\pd y_k}
= \dv_y\bv + \bV_0(\bk):\nabla \bv.
\end{equation}
Let $J$ be the Jacobian of the transformation \eqref{map:1}.  
Choosing $\delta > 0$ small enough, we may assume 
that $J = J(\bk) = 1 + J_0(\bk)$, 
where $J_0(\bk)$ is a $C^\infty$ function defined for $|\bk| < \sigma$ such 
that $J_0(0) = 0$.  

To obtain another representation formula of $\dv_x\bu$,
we use the inner product $(\cdot, \cdot)_{\Omega_t}$.  For any 
test function $\varphi \in C^\infty_0(\Omega_t)$, 
we set $\psi(y) = \varphi(x)$. We then have 
\begin{align*}
&(\dv_x\bu, \varphi)_{\Omega_t} = -(\bu, \nabla\varphi)_{\Omega_t}
= -(J\bv, (\bI + \bV_0)\nabla_y\psi)_\Omega \\
&= (\dv((\bI + {}^\top\bV_0)J\bv), \psi)_\Omega
= (J^{-1}\dv((\bI + {}^\top\bV_0)J\bv), \varphi)_{\Omega_t},
\end{align*}
which, combined with \eqref{div:1-1}, leads to 
\begin{equation}\label{div:2}
\dv_x\bu= \dv_y\bv + \bV_0(\bk):\nabla\bv
= J^{-1}(\dv_y\bv + \dv_y(J{}^\top\bV_0(\bk)\bv)).
\end{equation}
Recalling that $J = J(\bk) = 1 + J_0(\bk)$, we define $g(\bv, \rho)$ and
$\bg(\bv, \rho)$ by letting  
\begin{equation}\label{form:g}\begin{split}
g(\bv, \rho) &= -(J_0(\bk)\dv\bv + (1+J_0(\bk))\bV_0(\bk):\nabla\bv), \\
\bg(\bv, \rho) &= -(1+J_0(\bk)){}^\top\bV_0(\bk)\bv,
\end{split}\end{equation}
and then by \eqref{div:2} we see that 
the divergence free condition: $\dv\bu=0$ is transformed to
 the second equation in the equations \eqref{main:eq1}.  
In particular, it follows from \eqref{div:2} that 
\begin{equation}\label{div:3}
J_0(k)\dv\bv + J(k)\bV_0(\bk):\nabla\bv
= \dv(J(k){}^\top\bV_0(\bk)\bv).
\end{equation}

To derive  $\bF(\bv, \rho)$, 
we first observe that
\begin{align}
&\sum_{j=1}^N\frac{\pd}{\pd x_j}(\mu\bD(\bu)_{ij} - \fp \delta_{ij})
\nonumber \\
&= \sum_{j,k=1}^N\mu(\delta_{jk} + V_{0jk})\frac{\pd}{\pd y_k}
(\bD(\bv)_{ij} + (\CD_\bD(\bk)\nabla\bv)_{ij})
-\sum_{j=1}^N(\delta_{ij} + V_{0ij})\frac{\pd \fq}{\pd y_j},
\label{change:5}
\end{align}
where we have used \eqref{div:1}.  Since 
$$\frac{\pd}{\pd t}[u_i(y + \Psi(y, t), t)]
= \frac{\pd u_i}{\pd t}(x, t) + 
\sum_{j=1}^N\frac{\pd\Psi_{j}}{\pd t}
\frac{\pd u_i}{\pd x_j}(x, t),
$$
we have 
$$\frac{\pd u_i}{\pd t} = \frac{\pd v_i}{\pd t}-\sum_{j,k=1}^N
\frac{\pd \Psi_{j}}{\pd t}
(\delta_{jk} + V_{0jk})\frac{\pd v_i}{\pd y_k},
$$
and therefore, 
\begin{equation}\label{change:3}
\frac{\pd u_i}{\pd t} + \sum_{j=1}^N u_j\frac{\pd u_i}{\pd x_j}
= \frac{\pd v_i}{\pd t}
 + \sum_{j,k=1}^N(v_j - \frac{\pd\Psi_{j}}{\pd t})
(\delta_{jk} + V_{0jk}(\bk))\frac{\pd v_i}{\pd y_k}.
\end{equation}
Putting \eqref{change:5} and \eqref{change:3} together gives 
\begin{align*}
f_i(x, t) & = \Bigl(\frac{\pd v_i}{\pd t} 
+ \sum_{j,k=1}^N(v_j - \frac{\pd\Psi_{j}}{\pd t})
(\delta_{jk} + V_{0jk}(\bk))\frac{\pd v_i}{\pd y_k}\Bigr) \\
&\quad - \mu\sum_{j,k=1}^N(\delta_{jk} + V_{0jk}(\bk))
\frac{\pd}{\pd y_k}(\bD(\bv)_{ij} + (\CD_\bD(\bk)\nabla\bv)_{ij}) \\
&\quad - \sum_{j=1}^N(\delta_{ij}+ V_{0ij}(\bk))\frac{\pd\fq}{\pd y_j}.
\end{align*}
Since $(\bI + \nabla\Psi)(\bI + \bV_0) = (\pd x/\pd y)(\pd y/\pd x) = \bI$, that is,
\begin{equation}\label{change:3*}
\sum_{i=1}^N(\delta_{mi} + \pd_m\Psi_{i})
(\delta_{ij} + V_{0ij}(\bk)) = \delta_{mj},
\end{equation}
we have
\begin{align*}
&\sum_{i=1}^N(\delta_{mi} + \pd_m\Psi_i)f_i(\Psi(y, t), t)\\
& \quad=\sum_{i=1}^N(\delta_{mi} + \pd_m\Psi_{i})
\Bigl(\frac{\pd v_i}{\pd t} 
+ \sum_{j,k=1}^N(v_j - \frac{\pd\Psi_{i}}{\pd t})
(\delta_{jk}+V_{0jk}(\bk))\frac{\pd v_i}{\pd y_k}\Bigr)
\\
&\quad-\mu\sum_{i,j,k=1}^N(\delta_{mi} + \pd_m\Psi_{i})
(\delta_{jk}+V_{0jk}(\bk))\frac{\pd}{\pd y_k}
(\bD(\bv)_{ij} +(\CD_\bD(\bk)\nabla\bv)_{ij})
-\frac{\pd \fq}{\pd y_m}.
\end{align*}
Thus, changing $i$ to $\ell$ and $m$ to $i$ in the formula above, 
we define an $N$-vector
of functions $\bF_1(\bv, \rho)$ by letting 
\begin{align}
&\bF_1(\bv, \rho)|_i  = -\sum_{j,k=1}^N(v_j 
- \frac{\pd\Psi_{j}}{\pd t})
(\delta_{jk} + V_{0jk}(\bk))\frac{\pd v_i}{\pd y_k} \nonumber\\
&\quad-\sum_{\ell=1}^N\pd_i\Psi_{\ell}
\Bigl(\frac{\pd v_\ell}{\pd t} 
+ \sum_{j,k=1}^N(v_j - \frac{\pd\Psi_{j}}{\pd t})
(\delta_{jk} + V_{0jk}(\bk))\frac{\pd v_\ell}{\pd y_k}\Bigr) \nonumber \\
&\quad + \mu\Bigl(\sum_{j=1}^N \frac{\pd}{\pd y_j}(\CD_\bD(\bk)\nabla\bv)_{ij}
+ \sum_{j,k=1}^NV_{0jk}(\bk)\frac{\pd}{\pd y_k}(\bD(\bv)_{ij}
+ (\CD_\bD(\bk)\nabla\bv)_{ij})  \nonumber \\ 
&\quad + \sum_{j,k,\ell=1}^N\pd_i\Psi_{\ell}(\delta_{jk} + V_{0jk}(\bk))
\frac{\pd}{\pd y_k}(\bD(\bv)_{\ell j} + (\CD_\bD(\bk)\nabla\bv)_{\ell j})
\Bigr), \label{form:f}
\end{align}
where $\bF_1(\bu, \rho)|_i$ denotes the $i^{\rm th}$ component of 
$\bF_1(\bu, \rho)$.  

Moreover, 
\begin{align*}
&(\bI + \nabla\Psi)\sum_{k=1}^M\int^{2\pi}_0(\bu(\cdot, t), \bp_k(\cdot))_{\Omega_t}\,dt\, \bp_k(x)\\
&\quad = (\bI + \nabla\Psi)\sum_{k=1}^M\int^{2\pi}_0\int_{B_R}(\bv(y, t)\cdot
\bp_k(y + \Psi(y, t))(1 + J_0(t))\,dydt\,\bp_k(y+\Psi(y, t))\\
& = \CL\bv_S + \bF_2(\bv, \rho)
\end{align*}
with
\begin{equation}\label{form:f*}\begin{aligned}
\bF_2(\bv, \rho) & = \sum_{k=1}^M\Bigl\{\int^{2\pi}_0\int_{B_R}(\bv(y, t)\cdot(\bp_k(y)J_0(t)
+ \tilde\bp_k(\Psi(y, t))(1+J_0(t))\,dydt\bp_k(y)\\
& + \int^{2\pi}_0\int_{B_R}\bv(y, t)\cdot\bp_k(y + \Psi(y, t))(1 + J_0(t))\,dydt
\,\tilde\bp_k(\Psi(y, t)) \\
& + \nabla\Psi\int^{2\pi}_0\int_{B_R}\bv(y, t)\cdot\bp_k(y+\Psi(y, t))(1 + J_0(t))\,dydt
\,\bp_k(y + \Psi(y, t)),
\end{aligned}\end{equation}
where we have set
$$\tilde\bp_k(\Psi(y, t)) = \begin{cases} 0 & \quad\text{for $k=1, \ldots, N$}, \\
c_{ij}(\Psi_i(y, t)\be_j - \Psi_j(y, t)\be_i) & \quad \text{for $k = N+1, \ldots, M$}.
\end{cases}
$$
Thus, setting 
\begin{equation}\label{right:1}
\bG(y, t) = (\bI + \nabla\Psi(y, t))\bff(y+\Psi(y, t), t), \quad
\bF(\bv, \rho) = \bF_1(\bv, \rho) + \bF_2(\bv, \rho),
\end{equation}
we have the first equation in equations \eqref{main:eq1}. 

We next consider the transformation of the boundary conditions.
Recall that $\Gamma_t$ is represented by 
$x = y + \rho(y, t)\bn(y) + \xi(t)$ for $y \in S_R$ with $\bn(y)=y/|y|$. 
Let $x_0$ be any point on $S_R$ and let $\Phi(p)$ be a $C^\infty$ diffeomorphism
on $\BR^N$ such that---up to a rotation---it holds 
$$B_R \cap B_\omega(x_0) =\Phi( \{p \in \BR^N \mid 0 < p_N < \omega, 
\quad \mid |p'| < \omega \}) \cap B_\omega(x_0),
$$
where we have set $B_\omega(x_0) = \{y \in \BR^N \mid |y - x_0| < \omega\}$ and 
$p' = (p_1, \ldots, p_{N-1})$. 
Notice that $y = \Phi(p', 0) \in S_R \cap B_\omega(x_0)$ and 
$\rho(y, t) = H_\rho(\Phi(p',0), t)$. 
Let $\{x_k\}_{k=1}^K$ and  $\{\zeta_k\}_{k= 1}^K$ be a finite number of points on
$S_R$ and a partition of unity of $S_R$ such that ${\rm supp}\, \zeta_k 
\subset B_\omega(x_k)$ and $\sum_{k=1}^K\zeta_k(y) = 1$
on $S_R$. 
In the following, we represent functions on each $S_R\cap B_\omega(x_k)$, and 
to represent functions globally, we use the formula:
\begin{equation}\label{local:1}
f = \sum_{k=1}^K \zeta^1_kf \quad\text{in $S_R$}.
\end{equation}
Thus, for the detailed calculations, we only consider the domain $B_R \cap B_\omega(x_\ell)$
($\ell=1, \ldots, K$), 
and use the local coordinate system: $y = \Phi_\ell(p)$ for $p \in U$, where
we have written $\Phi=\Phi_\ell$, and 
$U = \{p \in \BR^N \mid 0 < p_N < \omega, \enskip |p'| < \omega\}$.

We write $\rho = \rho(y(p_1, \ldots, p_{N-1}, 0), t)$
in the following. 
By the chain rule, we have
\begin{equation}\label{chain:1}
\frac{\pd \rho}{\pd p_i} 
= \frac{\pd}{\pd p_i}H_\rho(\Phi_\ell(p_1, \ldots, p_{N-1}, 0), t)
= \sum_{m=1}^N\frac{\pd H_\rho}{\pd y_m}\frac{\pd \Phi_{\ell, m}}{\pd p_i}|_{p_N=0},
\end{equation}
where we have set $\Phi_\ell= {}^\top(\Phi_{\ell,1}, \ldots, \Phi_{\ell,N})$, 
and so, $\pd \rho/\pd p_i$ is defined in $B_\omega(x_0)$ 
by letting
\begin{equation}\label{chain:2}
\frac{\pd \rho}{\pd p_i} 
= \sum_{m=1}^N\frac{\pd H_\rho}{\pd y_m}\circ\Phi_\ell
\frac{\pd \Phi_{\ell,m}}{\pd p_i}.
\end{equation}


We first represent $\bn_t$. Since $\Gamma_t$ is given by
$x = y + \rho(y, t)\bn + \xi(t)$ for $y \in S_R$,  
$$\bn_t = a(\bn + \sum_{i=1}^{N-1}b_i\tau_i)
\quad\text{with $\tau_i = \frac{\pd}{\pd p_i}y 
=\frac{\pd}{\pd p_i}\Phi_\ell(p', 0)$}.
$$
The vectors $\tau_i$ ($i=1, \ldots, N-1$) 
form a basis of the tangent space of $S_R$ at $y=y(p_1,\ldots, p_{N-1})$.
Since $|\bn_t|^2 =1$, we have
\begin{equation}\label{norm:1}
1  = a^2(1 + \sum_{i,j=1}^{N-1}g_{ij}b_ib_j)
\quad\text{with $g_{ij} = \tau_i\cdot\tau_j$}
\end{equation}
because $\tau_i\cdot\bn=0$. The vectors $\dfrac{\pd x}{\pd p_i}$
$(i=1, \ldots, N-1)$ form a basis of the tangent space of $\Gamma_t$,
and so  $\bn_t\cdot\dfrac{\pd x}{\pd p_i}=0$.  Thus, we have
\begin{equation}\label{normal:3.1.1}
0 = a(\bn + \sum_{j=1}^{N-1}b_j\tau_j)\cdot
(\frac{\pd y}{\pd p_i} + \frac{\pd \rho}{\pd p_i}\bn 
+ \rho\frac{\pd \bn}{\pd p_i}).
\end{equation}
Since $\bn\cdot\dfrac{\pd y}{\pd p_i} = \bn\cdot\tau_i= 0$,
$\dfrac{\pd \bn}{\pd p_i}\cdot\bn
= 0$ (because of $|\bn|^2=1$), and $\dfrac{\pd y}{\pd p_i}\cdot
\dfrac{\pd y}{\pd p_j} = \tau_i\cdot\tau_j=g_{ij}$,
recalling that $\bn = R^{-1}y= R^{-1}\Phi_\ell$, 
by \eqref{normal:3.1.1} we have
$$
\frac{\pd \rho}{\pd p_i} + \sum_{j=1}^{N-1}(1+R^{-1}\rho)g_{ij}b_j = 0.
$$
Let $G=(g_{ij})$ and $G^{-1} = (g^{ij})$,  and then 
setting $\nabla_\Gamma'\rho = (\pd \rho/\pd p_1, \ldots, \pd\rho/\pd p_{N-1})$,
 we have 
\begin{equation}\label{norm:2}
b_i = -(1+R^{-1}\rho)^{-1}\sum_{k=1}^{N-1}g^{ik}\frac{\pd\rho}{\pd p_k}, 
\quad
b = -(1+R^{-1}\rho)^{-1}G^{-1}\nabla'_\Gamma\rho, 
\end{equation}
which leads to 
\begin{equation}\label{norm:3}
\bn_t  = a\Bigl(\bn -(1+R^{-1}\rho)^{-1} \sum_{i,j=1}^{N-1}g^{ij}\frac{\pd\rho}{\pd p_j}\tau_i\Bigr). 
\end{equation}
Moreover,  combining \eqref{norm:1} and \eqref{norm:2}, we have 
\begin{align*}
a = (1 + (1+R^{-1}\rho)^{-2}<G^{-1}\nabla_\Gamma'\rho, \nabla_\Gamma'\rho>)^{-1/2}.
\end{align*}
Using the formula: 
$$(1 + f)^{-1/2} = 1 - \frac12\int^1_0(1 + \theta f)^{-3/2}\,d\theta\,f,$$
we have
$$a = 1  - V_\Gamma(\rho, \nabla'_\Gamma\rho)
$$
with 
$$V_\Gamma(\rho, \nabla'_\Gamma\rho)
= \frac12\int^1_0(1+\theta(1+R^{-1}\rho)^{-2}<G^{-1}\nabla'_\Gamma\rho, \nabla'_\Gamma\rho>
)^{-3/2}\,d\theta
(1+R^{-1}\rho)^{-2}<G^{-1}\nabla'_\Gamma\rho, \nabla'_\Gamma\rho>.
$$ 
Combining these formulas obtained above gives 
\begin{equation}\label{repr:2.1} \bn_t = \bn
-\sum_{i,j=1}^{N-1}g^{ij}\frac{\pd \rho}{\pd p_j}\tau_i 
+\bV_\bn(\rho, \nabla'_\Gamma\rho)
\end{equation}
where we have set 
\begin{align*}
&\bV_\bn(\rho, \nabla'_\Gamma\rho) 
= \frac{\rho}{R+\rho}\sum_{i,j=1}^{N-1}g^{ij}\frac{\pd \rho}{\pd p_j}\tau_i 
- (\bn - \sum_{i,j=1}^{N-1}(1+R^{-1}\rho)^{-1}g^{ij}\frac{\pd\rho}{\pd p_j}\tau_i)
V_\Gamma(\rho, \nabla'_\Gamma\rho).
\end{align*}

From \eqref{chain:2}, $\nabla'_\Gamma\rho$ is extended to $\BR^N$ 
by the formula: $\nabla'_\Gamma\rho = (\nabla \Phi_\ell)\nabla 
\Psi_\rho\circ\Phi_\ell$,
and so we may write 
\begin{align*}
\bV_\bn(\rho, \nabla'_\Gamma\rho)
= \bV_{\bn, \ell}(\bk)\bar\nabla \Psi_\rho\otimes\bar\nabla\Psi_\rho
\end{align*}
on $B_\omega(x_\ell)$ with some function $\bV_{\bn, \ell}(\bk)
= \bV_{\bn,\ell}(y, \bk)$ defined on 
$B_\omega(x_\ell)\times\{\bk \mid |\bk| \leq \delta\}$ 
with $\bV_{\bn, \ell}(0) = 0$ possessing
the estimate
$$\|(\bV_{\bn, \ell}(\cdot,\bk), \pd_{\bk}
\bV_{\bn, \ell}(\cdot, \bk))
\|_{H^1_\infty(B_\omega(x_\ell))} \leq C
$$
with some constant $C$ independent of $\ell$.  Here and 
in the following $\bk$ are the variables corresponding
to $\bar\nabla H_\rho = (H_\rho, \nabla H_\rho)$. 
In view of \eqref{repr:2.1}, we have 
\begin{equation}\label{repr:2.1*}
\bn_t = \bn - \sum_{i,j=1}^{N-1}g^{ij}\tau_i\frac{\pd\rho}{\pd p_j}
+\bV_{\bn, \ell}(\bk)\bar\nabla\Psi_\rho\otimes\bar\nabla\Psi_\rho
\quad\text{on $B_\omega(x_\ell) \cap S_R$}.
\end{equation}
Thus, in view of \eqref{local:1} and \eqref{chain:2}, we may write 
\begin{equation}\label{normal:3.3}
\bn_t = \bn - \sum_{i,j=1}^{N-1}g^{ij}\pd'_j\rho\tau_i 
+ \bV_{\bn}(\bar\nabla H_\rho)\bar\nabla H_\rho \otimes \bar\nabla H_\rho
\quad\text{on $S_R$}, 
\end{equation}
where $\pd'_j\rho = \pd \rho/\pd p_j$ locally on $B_\omega(x_\ell) \cap S_R$,
$\bar\nabla H_\rho = (H_\rho, \nabla H_\rho)$,  and $\bV_{\bn}(\bk)$
is a matrix of functions defined on $\overline{B_R}\times\{\bk \mid |\bk| < \delta\}$
possessing the estimate:
\begin{equation}\label{normal:3.4}
\|(\bV_{\bn}, \pd_{\bk}\bV_{\bn})(\cdot, \bk)\|_{H^1_\infty(B_R)}
\leq C \quad\text{for $|\bar\bk| \leq \delta$}.
\end{equation}

And also we may write
\begin{equation}\label{normal:3.1}
\bn_t = \bn + \tilde\bV_\bn(\bar\nabla H_\rho)\bar\nabla H_\rho
\end{equation}
where  $\tilde\bV_{\bn}(\bk)$
is a matrix of functions defined on $\overline{B_R}\times\{\bk \mid |\bk| < \delta\}$
possessing the estimate:
\begin{equation}\label{normal:3.2}
\|(\tilde\bV_\bn(\cdot, \bk), \pd_{\bk}\tilde\bV_\bn(\cdot, 
\bk))\|_{H^1_\infty(B_R)}
\leq C \quad\text{for $|\bk| \leq \delta$}.
\end{equation}



We now consider the boundary condition:
\begin{equation}\label{52}
(\mu\bD(\bu) - \fp\bI)\bn_t  = \sigma H(\Gamma_t)\bn_t-p_0\bn_t
\end{equation}
It is convenient  
to divide the formula in \eqref{52} into 
the tangential part and normal part on $\Gamma_t$
as follows:
\begin{gather}
\Pi_t\mu\bD(\bu)\bn_t = 0, \label{55} \\
<\mu\bD(\bv)\bn_t, \bn_t> - \fp 
= \sigma <H(\Gamma_t)\bn_t, \bn_t>-p_0
= h_N(\bv, \rho)
\label{56}
\end{gather}
Here, $\Pi_t$ is defined by 
$\Pi_t\bd = \bd- < \bd, \bn_t>\bn_t$ for any 
$N$-vector of functions $\bd$.
In the last equation in equations \eqref{main:eq1}, we set
$\bh'(\bv, \rho) = \bh(\bv, \rho) - <\bh(\bv, \rho), \bn>\bn$ 
and $h_N(\bv, \rho) = <\bh(\bv, \rho), \bn>$.  
By \eqref{normal:3.1}  and \eqref{div:1}, we see that 
the boundary condition \eqref{55} is transformed to the following 
formula: 
\begin{equation}\label{58*}
(\mu\bD(\bv)\bn)_\tau = \bh'(\bv, \rho)
\quad\text{on $\Gamma\times(0, T)$}, 
\end{equation}
where we have set $\bd_\tau = \bd - <\bd, \bn>\bn$ and 
\begin{equation}\label{58}\begin{aligned}
&\bh'(\bv, \rho) 
 = - \mu\bD(\bv)\tilde \bV_\bn(\bar\nabla H_\rho)\bar\nabla H_\rho \\
&\quad+\mu\{<\bD(\bv)\tilde \bV_\bn(\bar\nabla H_\rho)\bar\nabla H_\rho, 
\bn+\tilde \bV_\bn(\bar\nabla H_\rho)\bar\nabla H_\rho>
(\bn+\tilde \bV_\bn(\bar\nabla H_\rho)\bar\nabla H_\rho)\\
&\quad+<\bD(\bv)\bn, \tilde \bV_\bn(\bar\nabla H_\rho)\bar\nabla H_\rho>
(\bn+\tilde \bV_\bn(\bar\nabla H_\rho)\bar\nabla H_\rho) \\
&\quad+<\bD(\bv)\bn, \bn>\tilde \bV_\bn(\bar\nabla H_\rho)\bar\nabla H_\rho\}
-\mu(\CD_\bD(\bk)\nabla\bv)(\bn+
\tilde \bV_\bn(\bar\nabla H_\rho)\bar\nabla H_\rho) \\
&\quad -\mu<(\CD_\bD(\bk)\nabla\bv)(\bn+
\tilde \bV_\bn(\bar\nabla H_\rho)\bar\nabla H_\rho), 
\bn+\tilde \bV_\bn(\bar\nabla H_\rho)\bar\nabla H_\rho>
(\bn + \tilde \bV_\bn(\bar\nabla H_\rho)\bar\nabla H_\rho).
\end{aligned}\end{equation}



Finally, we derive the nonlinear term 
$h_N(\bu, \rho)$ in \eqref{56}. Recall that 
$\Gamma_t$ is represented by 
$x = (R + \rho)\bn(y) + \xi(t)$  for $y \in S_R$, 
where $\bn = y/|y| \in S_1$.  Then, we have 
$$\frac{\pd x}{\pd p_j}= (R+\rho)\tau_j
+ \frac{\pd\rho}{\pd p_j}\bn
$$
where $\tau_j = \frac{\pd \bn}{\pd p_j}$, which forms
a basis of the tangent space of $S_1$.  Since 
$\tau_j\cdot \bn = 0$, 
the $(i, j)^{\rm th}$ component of the first fundamental form 
$G_t=(g_{tij})$ of $\Gamma_t$ is given by 
$$g_{tij}=\frac{\pd x}{\pd p_i}\cdot\frac{\pd x}{\pd p_j}
= (R+\rho)^2g_{ij} + \frac{\pd \rho}{\pd p_i}\frac{\pd \rho}{\pd p_j},$$
where $g_{ij} = \tau_i\cdot\tau_j$ is the $(i, j)^{\rm th}$ element of 
the  first fundamental form, $G$,  of $S_1$, and so 
\begin{align*}
G_t &= (R+\rho)^2(G+(R+\rho)^{-2}\nabla'_\Gamma\rho\otimes
\nabla'_\Gamma\rho) \\
&= (R+\rho)^2G(\bI + (R+\rho)^{-2}(G^{-1}\nabla'_\Gamma\rho)\otimes \nabla'_\Gamma\rho).
\end{align*}
Since
\begin{equation}\label{38} 
\det(\bI + \ba'\otimes\bb') = 1 + \ba'\cdot\bb',
\quad
(\bI + \ba'\otimes\bb')^{-1} = \bI - 
\frac{\ba'\otimes\bb'}{1+\ba'\cdot\bb'}
\end{equation}
for any $(N-1)$-vectors $\ba'$ and $\bb' \in \BR^{N-1}$, we have
\begin{align*}G_t^{-1}&=(R+\rho)^{-2}
\Bigl(\bI - \frac{(R+\rho)^{-2}(G^{-1}\nabla'_\Gamma\rho)\otimes \nabla'_\Gamma\rho}
{1 + (R+\rho)^{-2}<G^{-1}\nabla'_\Gamma\rho, \nabla'_\Gamma\rho>}
\Bigr)G^{-1} \\
&= (R+\rho)^{-2}G^{-1} + O_2.
\end{align*}
Here and in the following,  $O_2$ denotes a symbol defined by setting
$$O_2 = a_0H_\rho^2 + \sum_{j=1}^Nb_jH_\rho\frac{\pd H_\rho}{\pd y_j}
+ \sum_{i,j=1}^N c_{ij}\frac{\pd H_\rho}{\pd y_i}\frac{\pd H_\rho}{\pd y_j}
$$
with some coefficients $a_0$, $b_j$ and $c_{ij}$ defined on $\overline{B_R}$
 satisfying the estimate:
$|(a_0, b_j, c_{ij})(y, t)| \leq C$ 
and $|\nabla(a_0, b_j, c_{ij})(y, t)| \leq C|\nabla^2 H_\rho(y, t)|$ provided that
$\|H_\rho\|_{L_\infty((0, 2\pi), H^1_\infty(B_R))} \leq \delta$. 
In particular, 
$$g_t^{ij} = (R+\rho)^{-2}g^{ij} + O_2, 
$$
componentwise.

We next calculate the Christoffel symbols of $\Gamma_t$. Since
\begin{align*}
\tau_{ti} &= (R+\rho)\tau_i + \frac{\pd\rho}{\pd p_i}\bn, \\
\tau_{tij} &= (R+\rho)\tau_{ij} + 
\frac{\pd\rho}{\pd p_j}\tau_i + 
\frac{\pd\rho}{\pd p_i}\tau_j
+ \frac{\pd^2\rho}{\pd p_i\pd p_j}\bn,
\end{align*}
we have
\begin{align*}
<\tau_{tij}, \tau_{t\ell}> &= (R+\rho)^2<\tau_{ij}, \tau_\ell>
+ (R+\rho)(\frac{\pd\rho}{\pd p_\ell}\ell_{ij}
+ g_{i\ell}\frac{\pd\rho}{\pd p_j}
+ g_{j\ell}\frac{\pd\rho}{\pd p_i})\\
&+ \frac{\pd^2\rho}{\pd p_i\pd p_j}\frac{\pd \rho}{\pd p_\ell},
\end{align*}
where $\ell_{ij} = <\tau_{ij}, \bn>$, and so
\begin{align*}
\Lambda^k_{tij}  &= g_t^{k\ell}<\tau_{tij}, \tau_{t\ell}> \\
&=\bigg((R+\rho)^{-2}g^{k\ell} + O_2\bigg) \bigg( (R+\rho)^2<\tau_{ij}, \tau_\ell>\\
&+ (R+\rho)(\frac{\pd\rho}{\pd p_\ell}\ell_{ij}
+ g_{i\ell}\frac{\pd\rho}{\pd p_j}
+ g_{j\ell}\frac{\pd\rho}{\pd p_i}) + \frac{\pd^2\rho}{\pd p_i\pd p_j}\frac{\pd \rho}{\pd p_\ell}\bigg)\\
& = \Lambda^k_{ij} + (R+\rho)^{-1}g^{k\ell}(\frac{\pd\rho}{\pd p_\ell}\ell_{ij}
+ \delta^k_i\frac{\pd\rho}{\pd p_j}
+ \delta^k_j\frac{\pd\rho}{\pd p_i})\\
&+((R+\rho)^{-2}g^{k\ell}\frac{\pd\rho}{\pd p_\ell}
+ O_2)\frac{\pd^2\rho}{\pd p_i \pd p_j}
+ O_2.
\end{align*}
Thus, 
\begin{align*}
&\Delta_{\Gamma_t}f  = g_t^{ij}(\pd_i\pd_jf- \Lambda^k_{tij}\pd_kf) \\
& =(R+\rho)^{-2}g^{ij}(\pd_i\pd_jf - \Lambda^k_{ij}\pd_kf) 
+(A^k(\nabla'_p\rho,\nabla'^2_p\rho)\pd_kf + O_2\otimes(\bar\nabla'^2f)
\end{align*}
where   $\bar\nabla'^2f$ is an $((N-1)^2+N)$-vector  of
the form: $\bar\nabla'^2f = (\pd_i\pd_jf, \pd_if, f \mid i, j=1, \ldots, N-1)$,
$\pd_i = \pd/\pd p_i$,  $\nabla'^2_p = (\pd_i\pd_j\rho \mid i, j=1, \ldots, N-1)$, and 
\begin{align*}
A^k(\nabla'_p\rho, \nabla'^2_p\rho)
&= -(R+\rho)^{-3}g^{ij}g^{k\ell}(\frac{\pd\rho}{\pd p_\ell}\ell_{ij}
+ \delta^k_i\frac{\pd\rho}{\pd p_j}
+\delta^k_j\frac{\pd \rho}{\pd p_i}) \\
&\qquad-(R+\rho)^{-2}((R+\rho)^{-2}g^{ij}g^{k\ell}\frac{\pd\rho}{\pd p_\ell}
+ g^{ij}O_2)\frac{\pd^2\rho}{\pd p_i\pd p_j},
\end{align*}
and so
\begin{align*}
H(\Gamma_t)\bn_t &= \Delta_{\Gamma_t}[(R+\rho)\bn+ \xi(t)] \\
& =(R+\rho)^{-2}g^{ij}(\pd_i\pd_j - \Lambda^k_{ij}\pd_k)((R+\rho)\bn)
+ (A^k\nabla^2_p\rho )\pd_k((R+\rho)\bn) \\
&\qquad + O_2\otimes\bar\nabla'^2((R+\rho)\bn)\\
& = (R+\rho)^{-1}g^{ij}(\pd_i\pd_j\bn-\Lambda^k_{ij}\pd_k\bn)
+ (R+\rho)^{-2}g^{ij}(\pd_i\rho\pd_j\bn + \pd_j\rho\pd_i\bn)\\
&\quad + (R+\rho)^{-2}g^{ij}(\pd_i\pd_j\rho-\Lambda^k_{ij}\pd_k\rho)\bn
+ A^k(\nabla'_p\rho, \nabla'^2_p\rho)(\pd_k\rho)\bn \\
& \quad + A^k(\nabla'_p\rho, \nabla'^2_p\rho)(R+\rho)\pd_k\bn 
+ O_2\otimes\bar\nabla'^2(R+\rho)
\end{align*}
Combining this formula with \eqref{repr:2.1},  using $<\pd_i\bn, \bn> = 0$, $<\bn,  \tau_\ell> = 0$,  
$\Delta_{S_1}\bn=
-(N-1)\bn$, and \eqref{chain:1}  gives 
\begin{align*}
&<H(\Gamma_t)\bn_t, \bn_t> \\
&\quad = -(R+\rho)^{-1}(N-1) + (R+\rho)^{-2}
\Delta_{S_1}\rho + (O_1+O_2)\otimes \nabla^2_p\rho + O_2,
\end{align*}
where $O_1$ denotes a symbol defined by setting
$$O_1 = a_0'H_\rho + \sum_{j=1}^Nb_j' \frac{\pd H_\rho}{\pd y_j}$$
with some coefficients $a_0'$ and $b_j'$ defined on $\overline{B_R}$ satisfying
the estimate: $|(a_0', b_j')(y, t)|\leq C$ and $|\nabla(a_0', b_j')(y, t)|
\leq C|\nabla^2H_\rho(y, t)|$ provided that $\|H_\rho\|_{L_\infty((0, 2\pi),
H^1_\infty(B_R))} \leq \delta$. 
Since
\begin{align*}
(R+\rho)^{-1}=&R^{-1} - \rho R^{-2} + O(\rho^2), \\
(R+\rho)^{-2}\Delta_{S_1}\rho &= R^{-2}\Delta_{S_1}\rho 
+2R^{-3}\rho\Delta_{S_1}\rho + O_2\otimes\nabla^2_p\rho,
\end{align*}
we have 
\begin{equation}\label{44}\begin{aligned}
&<H(\Gamma_t)\bn_t, \bn_t>  = -\frac{N-1}{R} + \CB\rho 
+ 
(O_1+ O_2)\otimes\nabla^2_p\rho 
 + O_2.
\end{aligned}\end{equation}
Setting $p_0 = -(N-1)/R$, from \eqref{52} 
we have
$$<\mu\bD(\bv)\bn, \bn> - \fq - \sigma\CB\rho
= h_N(\bv, \rho)
$$
on $S_R\times(0, 2\pi) $.  Here, in view of  \eqref{div:1} and \eqref{44},  
we have defined  $h_N(\bv, \rho)$ by letting 
\begin{equation}\label{non:g.5}
h_N(\bv, \rho) = \bV_{h,N}(\bar\nabla H_\rho)
\bar\nabla H_\rho\otimes\nabla\bv
+ \sigma \tilde \bV'_\Gamma(\bar\nabla H_\rho)\bar\nabla H_\rho\otimes
\bar\nabla^2 H_\rho,
\end{equation}
where $\bV_{h, N}(\bk)$ and $\tilde\bV'_\Gamma(\bk)$ are 
functions defined on $\overline{B_R}\times\{\bk \mid |\bk| < \delta\}$
possessing the estimate:
\begin{align*}
&\sup_{|\bk| < \delta}\|(\bV_{h, N}(\cdot, \bk), 
\pd_{\bk}\bV_{h, N}(\cdot, \bk))\|_{H^1_\infty(B_R)}
\leq C, \\
&\sup_{|\bk| < \delta}\|(\tilde\bV'_{\Gamma}(\cdot, \bk), 
\pd_{\bk}\tilde\bV'_{\Gamma}(\cdot, \bk))\|_{H^1_\infty(B_R)}
\leq C
\end{align*}
for some constant $C$. 

\subsection{Two-phase problem case}
Let $\Omega_+ = B_R$ and $\Omega_- = \Omega\setminus(B_R\cup S_R)$.  
In the two-phase case, we let
$$\Psi_+(y, t) =R^{-1}H_\rho(y, t)y + \xi(t), \quad
\Psi_-(y, t) =  \varphi(y)(R^{-1}H_\rho(y,t) y + \xi(t)).
$$ 
Let $J_\pm(t)$ be the Jacobian of the map: $x  = y+ \Psi_\pm(y, t)$ for $y \in \Omega_\pm$,
which are defined by setting 
\begin{equation}\label{jacob}
\left\{
\begin{aligned}  J_+(t) &= \det(I + R^{-1}\nabla_y (H_\rho(y, t)y)) \quad&\text{for $y \in \Omega_+$}, \\
J_-(t)&=\det(I + \nabla_y(\varphi(y)(R^{-1}(H_\rho(y, t)y + \xi(t))) \quad&\text{for $y \in \Omega_-$}. 
\end{aligned}\right.
\end{equation}
Notice that 
$$\xi(t) = \int^t_0\int_{B_R}\bv_+(y, s)J_{+}(s)\,dyds + c$$
where $c$ is the unique constant for which the following equality holds:
$$\int^{2\pi}_0\xi(t) = 0.$$
We assume that 
\begin{equation}\label{small:3.1}
\sup_{t \in (0, 2\pi)}\|H_\rho(\cdot, t)\|_{H^1_\infty(\Omega_\pm)} \leq \delta,
\quad 
\sup_{t \in (0, 2\pi)} |\xi(t)| \leq \delta
\end{equation}
with suitably small constant $\delta > 0$. Since 
$$|\xi(t)| \leq C\sup_{t \in (0, 2\pi)}\|\bv(\cdot, t)\|_{L_q(B_R)}\sup_{t \in (0, 2\pi)}
|J_+(t)||B_R|,
$$
there exists  a constant $\delta_1 > 0$ such that if 
\begin{equation}\label{small:3.2}
\sup_{t \in (0, 2\pi)}\|\bv_+(\cdot, t)\|_{L_q(B_R)} \leq \delta_1
\end{equation}
then the condition for $\xi(t)$ in \eqref{small:3.1} holds. 
Thus, in the proof of Theorem \ref{main:thm2} below, we assume that 
the conditions \eqref{small:3.1} and  \eqref{small:3.2} hold. 

Set $J_{0\pm}(t) = J_\pm(t) - 1$.  
By the chain rule, we have
$$\nabla_x = (\bI + \bV_{\pm0}(\bk_\pm))\nabla_y, 
\quad \frac{\pd}{\pd x_i} + \sum_{j=1}^N V_{\pm 0ij}(\bk_\pm)\frac{\pd}{\pd y_j}
$$
where $\bV_{\pm0}(\bk_\pm)$ is given by 
$$\bV_{\pm 0}(\bk_\pm) 
= 
\left\{
\begin{aligned}
&(\bI + \nabla_y (R^{-1}H_\rho(y, t) y)^{-1} - \bI &\quad&\text{for $y \in \Omega_+$}, \\
&(\bI + \nabla_y\Psi_{-, \rho}(y, t))^{-1} - \bI
&\quad&\text{for $y \in \Omega_-$}.
\end{aligned}\right. 
$$
Here and in the following,  $\bk_+$ and $\bk_-$ denote
the   variables corresponding to 
$(H_\rho, \nabla H_\rho)$ and $(\Psi_{-, \rho}, \nabla \Psi_{-, \rho})$. 

Employing the same argument as for obtaining the formulas in \eqref{form:g}, we have 
\begin{equation}\label{form:g*}\begin{aligned}
g_\pm(\bv, \rho) &= -(J_{0\pm}(\bk_\pm)\dv\bv_\pm + (1 + J_{0\pm}(\bk_\pm))\bV_{0\pm}(\bk_\pm)
:\nabla\bv_\pm), \\
\bg_\pm(\bv, \rho) &= -(1+J_{0\pm}(\bk_\pm)){}^\top\bV_{0\pm}(\bk_\pm)\bv_\pm.
\end{aligned}\end{equation}
And also, from \eqref{right:1} we have 
\begin{equation}\label{right:1*}
\bG_\pm(y, t) = (\bI + \nabla \Psi_\pm(y, t))\bff(y + \Psi_\pm(y, t), t), \quad
\bF_\pm(\bv, \rho) = {}^\top(F_{1\pm}(\bv, \rho), \ldots, F_{N\pm}(\bv, \rho))
\end{equation}
with
\begin{align*}
F_{i\pm}(\bv, \rho)  &= -\sum_{j,k=1}^N(v_{\pm j} 
- \frac{\pd\Psi_{\pm j}}{\pd t}) 
(\delta_{jk} + V_{0jk}(\bk_\pm))\frac{\pd v_{\pm i}}{\pd y_k} \nonumber
\\
&\quad-\sum_{\ell=1}^N\pd_i\Psi_{\pm \ell}
\Bigl(\frac{\pd v_{\pm\ell}}{\pd t} 
+ \sum_{j,k=1}^N(v_{\pm j} - \frac{\pd\Psi_{\pm j}}{\pd t})
(\delta_{jk} + V_{\pm 0jk}(\bk_\pm))\frac{\pd v_{\pm \ell}}{\pd y_k}\Bigr) \nonumber \\
&\quad + \mu\Bigl(\sum_{j=1}^N \frac{\pd}{\pd y_j}(\CD_\bD(\bk_\pm)\nabla\bv_\pm)_{ij}
+ \sum_{j,k=1}^NV_{0jk}(\bk_\pm)\frac{\pd}{\pd y_k}(\bD(\bv_\pm)_{ij}
+ (\CD_\bD(\bk_\pm)\nabla\bv_\pm)_{ij})  \nonumber \\ 
&\quad + \sum_{j,k,\ell=1}^N\pd_i\Psi_{\pm\ell}(\delta_{jk} + V_{\pm 0jk}(\bk))
\frac{\pd}{\pd y_k}(\bD(\bv_\pm)_{\ell j} + (\CD_\bD(\bk_\pm)\nabla\bv_\pm)_{\ell j})
\Bigr). \label{form:2f}
\end{align*} 
Here and in the following, we have set $\Psi_{\pm}(y, t) = {}^\top(\Psi_{\pm1}(y, t), 
\ldots, \Psi_{\pm N}(y, t))$, $\bv_\pm = {}^\top(v_{\pm1}, \ldots, v_{\pm N})$, and
$$(\CD_\bD(\bk_\pm)\nabla \bv_\pm)_{ij} = \sum_{k=1}^N
\Bigl(V_{\pm0jk}(\bk_\pm)\frac{\pd v_{\pm i}}{\pd y_k} 
+ V_{\pm0ik}(\bk_\pm)\frac{\pd v_{\pm j}}{\pd y_k}\Bigr). 
$$
To define the right hand side of the transmission condition, we use \eqref{58}
and \eqref{non:g.5}.  We first introduce a symbol $((\cdot))$.
For $f_\pm$, let $[f_\pm]$ be a suitable extension of $f_\pm$ to $\Omega_\mp$ such that 
$$\|[f_\pm]\|_{H^k_q(\Omega_\mp)} \leq C_k\|f_\pm\|_{H^k_q(\Omega_\pm)}, 
\quad 
\|\pd_t[f_\pm]\|_{H^k_q(\Omega_\mp)} \leq C_k\|\pd_tf_\pm\|_{H^k_q(\Omega_\pm)}
$$
with some constant $C_k$.  Here, if the right-hand side is finite, then 
$[f_\pm]$ and $\pd_t[f_\pm]$ exist and the estimates above hold. 
In particular, we set $H^0_q(\Omega_\pm) = L_q(\Omega_\pm)$. We set 
$$ex[f_\pm](y, t) = \begin{cases} f_\pm(y, t) & \quad\text{for $y \in \Omega_\pm$}, \\
[f_\pm](y, t) & \quad\text{for $y \in \Omega_\mp$}.
\end{cases}
$$
And then, $((f))$ is defined by setting
$$((f)) = ex[f_+] - ex[f_-]. $$
Using this symbol, we can proceed as for the derivation of \eqref{58} and \eqref{non:g.5}
and define $\tilde\bh'(\bv, \rho)$ and $\tilde h_N(\bv, \rho)$ by setting  
 \begin{equation}\label{3.58}\begin{aligned}
&\tilde\bh'(\bv, \rho) 
 = - \mu((\bD(\bv)))\tilde \bV_\bn(\bar\nabla H_\rho)\bar\nabla H_\rho \\
&\quad+\mu\{<((\bD(\bv)))\tilde \bV_\bn(\bar\nabla H_\rho)\bar\nabla H_\rho, 
\bn+\tilde \bV_\bn(\bar\nabla H_\rho)\bar\nabla H_\rho>
(\bn+\tilde \bV_\bn(\bar\nabla H_\rho)\bar\nabla H_\rho)\\
&\quad+<((\bD(\bv)))\bn, \tilde \bV_\bn(\bar\nabla H_\rho)\bar\nabla H_\rho>
(\bn+\tilde \bV_\bn(\bar\nabla H_\rho)\bar\nabla H_\rho) \\
&\quad+<((\bD(\bv)))\bn, \bn>\tilde \bV_\bn(\bar\nabla H_\rho)\bar\nabla H_\rho\}
-\mu((\CD_\bD(\bk)\nabla\bv))(\bn+
\tilde \bV_\bn(\bar\nabla H_\rho)\bar\nabla H_\rho) \\
&\quad -\mu<((\CD_\bD(\bk)\nabla\bv))(\bn+
\tilde \bV_\bn(\bar\nabla H_\rho)\bar\nabla H_\rho), 
\bn+\tilde \bV_\bn(\bar\nabla H_\rho)\bar\nabla H_\rho>
(\bn + \tilde \bV_\bn(\bar\nabla H_\rho)\bar\nabla H_\rho) \\
&\tilde h_N(\bv, \rho) = \bV_{h,N}(\bar\nabla H_\rho)\bar\nabla H_\rho\otimes((\nabla\bv))
+ \sigma\tilde \bV'_\Gamma(\bar\nabla H_\rho)\bar\nabla H_\rho
\otimes\bar\nabla^2 H_\rho.
\end{aligned}\end{equation}
And then, we set $\tilde \bh(\bv, \rho) =(\tilde\bh'(\bv, \rho), \tilde h_N(\bv, \rho))$. 

\section{On periodic solutions of the linearized equations}\label{sec:4}

In this section, we shall prove the $L_p$-$L_q$ maximal regularity of $2\pi$-periodic 
solutions of the linearized equations.
\subsection{On linearized problem of one-phase problem}
In this subsection, we consider the $L_p$-$L_q$ maximal regularity of 
periodic solutions to  linearized equations:
\begin{equation}\label{4.1}\begin{aligned}
\pd_t\bu + \CL\bu_S-\DV(\mu\bD(\bu)-\fp\bI) = \bF&&\quad&\text{in $B_R\times(0, 2\pi)$}, \\
\dv \bu =G = \dv\bG&&\quad&\text{in $B_R\times(0, 2\pi)$}, \\
\pd_t\rho + \CM\rho-(\CA\bu)\cdot\bn=D&&\quad&\text{on $S_R\times(0, 2\pi)$},\\
(\mu\bD(\bu)-\fp\bI)\bn - (\CB_R\rho)\bn=\bH&&\quad&
\text{on $S_R\times(0, 2\pi)$}, 
\end{aligned}\end{equation}
where $\CL$, $\CM$, and $\CA$ are the linear operators defined in \eqref{op:1}. 
We shall prove the unique existence theorem of $2\pi$-periodic solutions of equations
\eqref{4.1}.  Our main result is this section is stated as follows.
\begin{thm}\label{thm:linear.main1} Let $1 < p, q < \infty$.  Then, for any 
$\bF$, $D$, $G$, $\bG$ and $\bH$ with
\begin{align*}
\bF & \in L_{p, {\rm per}}((0, 2\pi), L_q(B_R)^N), \quad 
D \in L_{p, {\rm per}}((0, 2\pi), W^{2-1/q}_q(S_R)) \\
G & \in L_{p, {\rm per}}((0, 2\pi), H^1_q(B_R)) \cap H^{1/2}_{p, {\rm per}}((0, 2\pi), L_q(B_R)), 
\quad \bG \in H^1_{p, {\rm per}}((0, 2\pi). L_q(B_R)^N), \\
\bH& \in L_{p, {\rm per}}((0, 2\pi), H^1_q(B_R)^N) \cap H^{1/2}_{p, {\rm per}}((0, 2\pi), L_q(B_R)^N), 
\end{align*}
problem \eqref{4.1} admits unique solutions $\bu$, $\fp$ and $\rho$ with 
\begin{align*}
\bu & \in L_{p, {\rm per}}((0, 2\pi), H^2_q(B_R)^N) \cap H^1_{p, {\rm per}}
((0, 2\pi), L_q(B_R)^N), \\
\fp & \in L_{p, {\rm per}}((0, 2\pi), H^1_q(B_R)), \\
\rho & \in L_{p, {\rm per}}((0, 2\pi), W^{3-1/q}_q(S_R)) \cap H^1_{p, {\rm per}}
((0, 2\pi), W^{2-1/q}_q(S_R))
\end{align*}
possessing the estimate:
\begin{equation}\label{main:est1}\begin{aligned}
&\|\bu\|_{L_p((0, 2\pi), H^2_q(B_R))} + \|\pd_t\bu\|_{L_p((0, 2\pi), L_q(B_R))} 
+ \|\nabla \fp\|_{L_p((0, 2\pi), L_q(B_R))}  \\
&\quad + \|\rho\|_{L_p((0, 2\pi), W^{3-1/q}_q(S_R))} + 
\|\pd_t\rho\|_{L_p((0, 2\pi), W^{2-1/q}_q(S_R))} \\
&\qquad \leq C\{\|\bF\|_{L_p((0, 2\pi), L_q(B_R))} + \|D\|_{L_p((0, 2\pi), W^{2-1/q}_q(S_R))} 
+ \|\pd_t\bG\|_{L_p((0, 2\pi), L_q(B_R))} \\
&\quad\qquad + \|(G, \bH)\|_{L_p((0, 2\pi), H^1_q(B_R))}
+ \|(G, \bH)\|_{H^{1/2}_p((0, 2\pi), L_q(B_R))}\}
\end{aligned}\end{equation}
for some constant $C > 0$. 
\end{thm}
To prove Theorem \ref{thm:linear.main1}, 
our approach is to use the $\CR$-solver, Weis' operator-valued Fourier 
multiplier theorem \cite{Weis} and a transference theorem, which is 
created in Eiter, Kyed and Shibata \cite{EKS1}. To introduce the 
notion of $\CR$-solver, we introduce the $\CR$-boundedness of operator families.
\begin{dfn}\label{def:4.1} Let $X$ and $Y$ be two Banach spaces.  A family of operators
$\CT \subset \CL(X, Y)$ is called $\CR$-bounded on $\CL(X, Y)$, if 
there exist a constant $C > 0$ and $p \in [1, \infty)$ such that for each
$n \in \BN$, $\{T_j\}_{j=1}^n \in \CT^n$, 
and $\{f_j\}_{j=1}^n \in X^n$, we have
$$\|\sum_{k=1}^n r_kT_kf_k\|_{L_p((0,1), Y)} \leq 
C\|\sum_{k=1}^nr_kf_k\|_{L_p((0, 1), X)}.
$$
Here, the Rademacher functions $r_k$, $k \in \BN$, are given by
$r_k : [0, 1] \to \{-1, 1\}$, $t \mapsto {\rm sign}\,(\sin 2^k\pi t)$. 
The smallest such $C$ is called $\CR$-bound of $\CT$ on 
$\CL(X, Y)$, which is denoted by $\CR_{\CL(X, Y)}\CT$.
\end{dfn}
We quote Weis' operator-valued Fourier multiplier theorem and the transference theorem
for operator-valued Fourier multipliers.
\begin{thm}[Weis]\label{Weis} Let $X$ and $Y$ be two UMD Banach spaces. Let 
$m \in C^1(\BR\setminus\{0\}, \CL(X, Y))$ satisfies the multiplier condition:
$$\CR_{\CL(X, Y)}\{ (\tau \pd_\tau)^\ell m(\tau) \mid \tau \in \BR\setminus\{0\}\}
\leq r_b
$$
for $\ell = 0, 1$ with some constant $r_b$.  Let $T_m$ be a multiplier defined by 
$T_m[f] = \CF^{-1}[m\CF[f]]$. Then, $T_m \in \CL(L_p(\BR, X), L_p(\BR, Y))$ with
$$\|T_m[f]\|_{L_p(\BR, Y)} \leq C_pr_b\|f\|_{L_p(\BR, X)}$$
for any $p \in (1, \infty)$ with some constant $C_p$ depending on $p$ but independent
of $r_b$. 
\end{thm}
The transference theorem for operator-valued Fourier multipliers obtained in
\cite{EKS1} is stated as follows.
\begin{thm}\label{ESK} Let $X$ and $Y$ be two Banach spaces and $p \in (1, \infty)$.  Assume 
that $Y$ is reflexive.  Let  
$$m \in L_\infty(\BR, \CL(X, Y)) \cap C(\BR, \CL(X, Y)),$$
and let $m|_\BT$ denote the restriction of $m$ on $\BT$. 
We define multipliers on $\BR$ and $\BT$ associated with $m$ by setting 
\begin{align*}
T_{m, \BR}[f](t) = \CF^{-1}[m\CF[f]], \quad T_{m, 
\BT}[f] = \CF^{-1}_\BT[m|_\BT \CF_\BT[f]].
\end{align*}
If $T_{m, \BR} \in \CL(L_p(\BR, X), L_p(\BR, Y))$ possessing 
the estimate:
$$\|T_{m, \BR}[f]\|_{L_p(\BR, Y)} \leq M\|f\|_{L_p(\BR, X)}$$
for any $f \in L_p(\BR, X)$ with some constant $M$, then 
$T_{m, \BT} \in \CL(L_p(\BT, X), L_p(\BT, Y))$ and 
$$\|T_{m, \BT}[f]\|_{L_p(\BT, Y)}
\leq C_pM\|f\|_{L_p(\BT, X)}
$$
for any $f \in L_{p}(\BT, X)$
with some constant $C_p$ depending solely on $p$ and independent of
$M$.
\end{thm}
\begin{remark} In the usual scalar-valued multiplier case, the transference theorem was
proved by de Leeuw \cite{L}, and so this theorem is an extension to the
operator-valued case. 
\end{remark}

We now consider the $\CR$-solver of the generalized resolvent problem:
\begin{equation}\label{4.2}\begin{aligned}
ik\bv -\DV(\mu\bD(\bv)-\fq\bI) = \bff&&\quad&\text{in $B_R$}, \\
\dv \bv =g = \dv\bg&&\quad&\text{in $B_R$}, \\
ik\eta + \CM\eta-(\CA\bv)\cdot\bn=d&&\quad&\text{on $S_R$},\\
(\mu\bD(\bv)-\fq\bI)\bn - (\CB_R\eta)\bn=\bh&&\quad&
\text{on $S_R$}
\end{aligned}\end{equation}
for $k \in \BR$.  From Theorem 4.8 in Shibata \cite{S3} (cf. also Shibata \cite{S1, S2})
we know the following theorem concerned with the existence of an $\CR$-solver
of problem \eqref{4.1}.
\begin{thm}\label{thm:rs.1} Let $1 < q < \infty$ and let $\BR_{k_0}
= \BR\setminus(-k_0, k_0)$.  Let 
\begin{align*}
X_q(B_R) & = \{(\bff, d, \bh, g, \bg) \mid 
\bff \in L_q(B_R)^N, \enskip d \in W^{2-1/q}_q(S_R), \enskip
\bh \in H^1_q(B_R)^N, \enskip g \in H^1_q(B_R), 
\enskip \bg \in L_q(B_R)^N\}, \\
\CX_q(B_R) & = \{F=(F_1, F_2, \ldots, F_7) \mid F_1, F_3, F_7 
\in L_q(B_R)^N, \enskip F_2 \in W^{2-1/q}_q(S_R), \enskip 
F_4 \in H^1_q(B_R)^N, \\
&\phantom{= \{F=(F_1, F_2, \ldots, F_7) \mid F_1, F_3, F_7 
\in L_q(B_R)^N, a}\,\,
 F_5 \in L_q(B_R), \enskip
F_6 \in H^1_q(B_R)\}.
\end{align*}
Then, there exist a constant $k_0 > 0$ and operator families $\CA(ik)$, $\CP(ik)$, and  
$\CH(ik)$ with 
\begin{align*}
\CA(ik) & \in C^1(\BR_{k_0}, \CL(\CX_q(B_R), H^2_q(B_R)^N)), \\
\quad 
\CP(ik) &\in C^1(\BR_{k_0}, \CL(\CX_q(B_R), H^1_q(B_R))), \\
\quad
\CH(ik) &\in C^1(\BR_{k_0}, \CL(\CX_q(B_R), W^{3-1/q}_q(S_R)))
\end{align*}
such that for any $(\bff, d, \bh, g, \bg)$ and $k \in \BR_{k_0}$, 
$\bv = \CA(ik)\CF_k$, $\fq = \CP(ik)\CF_k$ and $\eta
= \CH(ik)\CF_k$, where 
$$\CF_k = (\bff, d, (ik)^{1/2}\bh, \bh, (ik)^{1/2}g, g, ik \bg),$$
are unique solutions of equations \eqref{4.2}, and 
\begin{equation}\label{4.3}\begin{aligned}
\CR_{\CL(\CX_q(B_R), H^{2-m}_q(B_R)^N)}
(\{(k\pd_k)^\ell((ik)^{m/2}\CA(ik))\mid k \in \BR_{k_0}\}) &\leq r_b, \\
\CR_{\CL(\CX_q(B_R), L_q(B_R)^N)}(\{(k\pd_k)^\ell \nabla\CP(ik) \mid 
k \in \BR_{k_0}\}) &\leq r_b, \\
\CR_{\CL(\CX_q(B_R), W^{3-n-1/q}_q(S_R))}
(\{(k\pd_k)^\ell((ik)^n\CH(ik))\mid k \in \BR_{k_0}\}) &\leq r_b
\end{aligned}\end{equation}
for $\ell=0,1$, $m=0,1,2$ and $n=0,1$ with some constant $r_b$. 
\end{thm}
\begin{remark} 
\begin{itemize}
\item[\thetag1]
Here and in the following, for $\theta\in(0,1)$ we set 
$$(ik)^\theta = \begin{cases} e^{i\pi\theta/2}|
k|^\theta &\quad\text{for $k > 0$}, \\
e^{-i\pi\theta/2}|
k|^\theta &\quad\text{for $k < 0$}.
\end{cases}
$$
\item[\thetag2]
The functions $F_1$, $F_2$, $F_3$, $F_4$, $F_5$, $F_6$, and $F_7$ are 
variables corresponding to $\bff$, $d$, $(ik)^{1/2}\bh$, $\bh$, $(ik)^{1/2}g$,
$g$, and $ik\, \bg$, respectively. 
\item[\thetag3]
We define the norm $\|\cdot\|_{\CX_q(B_R)}$ by setting
$$\|(F_1, \ldots, F_7)\|_{\CX_q(B_R)}
= \|(F_1, F_3, F_5, F_7)\|_{L_q(B_R)} 
+ \|F_2\|_{W^{2-1/q}_q(S_R)} + \|(F_4, F_6)\|_{H^1_q(B_R)}.
$$
\end{itemize}
\end{remark}
Let $\varphi(ik)$ be a function in $C^\infty(\BR)$ which equals one
for $k \in \BR_{k_0+2}$ and zero for $k \not \in \BR_{k_0+1}$, and 
let $\psi(ik)$ be a  function in $C^\infty(\BR)$ which equals one
for $k \in \BR_{k_0+4}$ and zero for $k \not \in \BR_{k_0+3}$.  Notice that
$\varphi(ik)\psi(ik) = \varphi(ik)$. Let $\CA(ik)$, $\CP(ik)$ and $\CH(ik)$ be 
the $\CR$-solvers given in Theorem \ref{thm:rs.1}. Then we have
\begin{equation}\label{4.4}\begin{aligned}
\CR_{\CL(\CX_q(B_R), H^{2-m}_q(B_R)^N)}
(\{(k\pd_k)^\ell((ik)^{m/2}(\varphi(ik)\CA(ik)))\mid k \in \BR_{k_0}\}) &\leq C\|\varphi\|_{H^1_\infty(\BR)}r_b, \\
\CR_{\CL(\CX_q(B_R), L_q(B_R)^N)}(\{(k\pd_k)^\ell \nabla(\varphi(ik)\CP(ik))\mid 
k \in \BR_{k_0}\}) &\leq C\|\varphi\|_{H^1_\infty(\BR)}r_b, \\
\CR_{\CL(\CX_q(B_R), W^{3-n-1/q}_q(S_R))}
(\{(k\pd_k)^\ell((ik)^n(\varphi(ik)\CH(ik)))\mid k \in \BR_{k_0}\}) &\leq C\|\varphi\|_{H^1_\infty(\BR)}r_b
\end{aligned}\end{equation}
for $\ell=0,1$, $m=0,1,2$ and $n=0,1$.  To prove \eqref{4.4}, we use the following
lemma concerning the fundamental properties of the $\CR$-bound and scalar-valued Fourier multipliers. 
\begin{lem}\label{lem:rp.2}
\thetag{a} Let $X$ and $Y$ be Banach spaces,
and let $\CT$ and $\CS$ be
$\CR$-bounded families in $\CL(X,Y)$.
Then, $\CT+ \CS = \{T + S \mid
T \in \CT, \enskip S \in \CS\}$ is also an $\CR$-bounded
family in
$\CL(X,Y)$ and
$$\CR_{\CL(X,Y)}(\CT + \CS) \leq
\CR_{\CL(X,Y)}(\CT) + \CR_{\CL(X,Y)}(\CS).$$

\thetag{b} Let $X$, $Y$ and $Z$ be Banach spaces,  and let
$\CT$ and $\CS$ be $\CR$-bounded families in $\CL(X, Y)$ and
$\CL(Y, Z)$, respectively.  Then, $\CS\CT = \{ST \mid
T \in \CT, \enskip S \in \CS\}$ is also an $\CR$-bounded
family in $\CL(X, Z)$ and
$$\CR_{\CL(X, Z)}(\CS\CT) \leq \CR_{\CL(X,Y)}(\CT)\CR_{\CL(Y, Z)}(\CS).$$

\thetag{c} Let $1 < p, \, q < \infty$ and let $D$ be a domain in $\BR^N$.
Let $m=m(\lambda)$ be a bounded function defined on a subset
$U$ of $\BC$ and let $M_m(\lambda)$ be a map defined by
$M_m(\lambda)f = m(\lambda)f$ for any $f \in L_q(D)$.  Then,
$\CR_{\CL(L_q(D))}(\{M_m(\lambda) \mid \lambda \in U\})
\leq C_{N,q,D}\|m\|_{L_\infty(U)}$.

\thetag{d} Let $n=n(\tau)$ be a $C^1$-function defined on $\BR\setminus\{0\}$
that satisfies the conditions $|n(\tau)| \leq \gamma$
and $|\tau n'(\tau)| \leq \gamma$ with some constant $c > 0$ for any
$\tau \in \BR\setminus\{0\}$.  Let $T_n$ be an
 operator-valued Fourier multiplier defined by $T_n f = \CF^{-1}[n \CF[f]]$
for any $f$ with $\CF[f] \in \CD(\BR, L_q(D))$.  Then,
$T_n$ is extended to a bounded linear operator
from $L_p(\BR, L_q(D))$ into itself.  Moreover,
denoting this extension also by $T_n$, we have
$$\|T_n\|_{\CL(L_p(\BR, L_q(D)))} \leq C_{p,q,D}\gamma.
$$
\end{lem}
Here, we only prove the $\CR$-boundedness of $\varphi(ik)ik\CA(ik)$.  
The $\CR$-boundedness of the other terms can be proved by the same
argument.  Let $n \in \BN$, $\{k_\ell\}_{\ell=1}^n \in \BR^n$, 
$\{F_\ell\}_{\ell=1}^n \in \CX_q(B_R)^n$.  Changing the labeling of indices  if necessary, 
we may assume that $\varphi(k_\ell)\not=0$ for $k = 1, \ldots, m$ and 
$\varphi(k_\ell) = 0$ for $\ell=m+1,\ldots, n$. 
And then,  using Lemma \ref{lem:rp.2},
we have
\begin{align*}
&\|\sum_{\ell=1}^n r_\ell\varphi(ik_\ell)(ik_\ell)\CA(ik_\ell)F_\ell
\|_{L_q((0, 1), L_q(B_R))} \\
&\quad = 
\|\sum_{\ell=1}^m r_\ell\varphi(ik_\ell)(ik_\ell)\CA(ik_\ell)F_\ell
\|_{L_q((0, 1), L_q(B_R))} \\
&\quad 
\leq r_b\|\sum_{\ell=1}^m r_\ell\varphi(ik_\ell)F_\ell
\|_{L_q((0, 1), L_q(B_R))} \\
&\quad =
 r_b\|\sum_{\ell=1}^n r_\ell\varphi(ik_\ell)F_\ell
\|_{L_q((0, 1), L_q(B_R))} 
\\
&\quad\leq
C_{q,R}
\|\varphi\|_{H^1_\infty(B_R)}r_b\|\sum_{\ell=1}^n r_\ell F_\ell
\|_{L_q((0, 1), L_q(B_R))},
\end{align*}
which shows that 
$$\CR_{\CL(\CX_q(B_R), L_q(B_R)^N)}
(\{ik\varphi(ik)\CA(ik) \mid k \in \BR_{k_0}\}) \leq C_{q,R}\|\varphi\|_{H^1_\infty(\BR)}r_b.
$$
For $f \in \{\bF, G, \bG,  D, \bH\}$, let 
$$f_\psi = \CF^{-1}_\BT[\psi \CF_\BT[f]].$$
We consider the high frequency part of the equations \eqref{4.1}:
\begin{equation}\label{4.5}\begin{aligned}
\pd_t\bu_\psi -\DV(\mu\bD(\bu_\psi)-\fp_\psi\bI) = \bF_\psi&&\quad&\text{in $B_R\times(0, 2\pi)$}, \\
\dv \bu_\psi =G_\psi = \dv\bG_\psi&&\quad&\text{in $B_R\times(0, 2\pi)$}, \\
\pd_t\rho_\psi + \CM\rho_\psi-(\CA\bu_\psi)\cdot\bn=D_\psi &&\quad&\text{on $S_R\times(0, 2\pi)$},\\
(\mu\bD(\bu_\psi)-\fp_\psi\bI)\bn - (\CB_R\rho_\psi)\bn=\bH_\psi &&\quad&
\text{on $S_R\times(0, 2\pi)$}.
\end{aligned}\end{equation}
By Theorem \ref{Weis}, Theorem \ref{ESK}, and \eqref{4.4},  we have immediately the
following theorem. 
\begin{thm}\label{thm:rs.2} Let $1 < p, q < \infty$.  Then, for any functions $\bF$, $G$, $\bG$, 
$D$, and $\bH$ with 
\begin{align*}
\bF & \in L_{p, {\rm per}}((0, 2\pi), L_q(B_R)^N), \quad 
D \in L_{p, {\rm per}}((0, 2\pi), W^{2-1/q}_q(B_R)), \quad \\
\bH &\in H^{1/2}_{p, {\rm per}}((0, 2\pi), L_q(B_R)^N) \cap 
L_{p, {\rm per}}((0, 2\pi), H^1_q(B_R)^N), \\
G & \in H^{1/2}_{p, {\rm per}}((0, 2\pi), L_q(B_R)) \cap 
L_{p, {\rm per}}((0, 2\pi), H^1_q(B_R)), \quad \bG \in H^1_{p, {\rm per}}((0, 2\pi), L_q(B_R)^N), 
\end{align*}
We let 
\begin{align*}
\bu_\psi & = \CF^{-1}_\BT[\varphi(ik)\CA(ik)\CF_k(\bF_\psi, D_\psi, \bH_\psi, G_\psi, \bG_\psi)](\cdot, t),
\\
\fp_\psi & = \CF^{-1}_\BT[\varphi(ik)\CP(ik)\CF_k(\bF_\psi, D_\psi, \bH_\psi, G_\psi, \bG_\psi)](\cdot, t),
\\
\rho_\psi & = \CF^{-1}_\BT[\varphi(ik)\CA(ik)\CF_k(\bF_\psi, D_\psi, \bH_\psi, G_\psi, \bG_\psi)](\cdot, t),
\end{align*}
where we have set 
\begin{align*}
\CF_k(\bF_\psi, D_\pi, \bH_\psi, G_\psi, \bG_\psi)
= &\psi(ik)(\CF_\BT[\bF](ik), \CF_\BT[D](ik), (ik)^{1/2}\CF_\BT[\bH](ik), 
\CF_\BT[\bH](ik), \\
&\quad (ik)^{1/2}\CF_\BT[G](ik), \CF_\BT[G](ik), ik\CF_\BT[\bG](ik)).
\end{align*}
Then, $\bu_\psi$, $\fp_\psi$ and $\rho_\psi$ are the unique solutions of  equations \eqref{4.5}, 
which possess the following estimate:
\begin{align*}
&\|\bu_\psi\|_{L_p((0, 2\pi), H^2_q(B_R))} + 
\|\pd_t\bu_\psi\|_{L_p((0, 2\pi), L_q(B_R))}  + 
\|\nabla \fp_\psi \|_{L_p((0, 2\pi), L_q(B_R))} \\
 &\hskip0.5cm+ \|\rho_\psi\|_{L_p((0, 2\pi), W^{3-1/q}_q(S_R))} + 
\|\pd_t\rho_\psi\|_{H^1_p((0, 2\pi), W^{2-1/q}_q(S_R))}\\
&\hskip1cm\leq C\{\|\bF_\psi\|_{L_p((0, 2\pi), L_q(B_R))} 
+ \|D_\psi\|_{L_p((0, 2\pi), W^{2-1/q}_q(S_R))} 
+ \|\Lambda^{1/2}(G_\psi, \bH_\psi)\|_{L_p((0, 2\pi), L_q(B_R))} \\
&\hskip1.5cm+ \|(G_\psi, \bH_\psi)\|_{L_p((0, 2\pi), H^1_q(B_R))}
+ \|\pd_t\bG_\psi\|_{L_p((0, 2\pi), L_q(B_R))}\}
\end{align*}
for some constant $C > 0$. Here, we have set
$$\Lambda^{1/2}(G_\psi, \bH_\psi) = \CF^{-1}_\BT[(ik)^{1/2}\psi(ik)(\CF_\BT[G](ik), \CF_\BT[\bH]
(ik))].$$
\end{thm}

We now consider the lower frequency part of solutions of equations \eqref{4.1}. 
Namely, we consider equations \eqref{4.2} for $k \in \BR$ with $1 \leq |k| < k_0+4$.
We shall show the following theorem.
\begin{thm}\label{thm:rs.3} Let $1 < q < \infty$ and $k \in \BZ$ with 
$1 \leq |k| \leq k_0+3$. Then, for any 
$\bff \in L_q(B_R)^N$, $g \in H^1_q(B_R)$, $d \in W^{2-1/q}_q(S_R)$, 
$\bh \in H^1_q(B_R)^N$, and $\bg \in L_q(B_R)^N$, problem \eqref{4.2}
admits unique solutions $\bv \in H^2_q(B_R)^N$, 
$\fq \in H^1_q(B_R)$, and $\eta \in W^{3-1/q}_q(S_R)$ possessing
the estimate:
\begin{equation}\label{4.6}\begin{aligned}
&\|\bv\|_{H^2_q(B_R)} + \|\nabla \fq\|_{L_q(B_R)}
+ \|\eta\|_{W^{3-1/q}_q(S_R)} \\
&\quad 
\leq C(\|\bff\|_{L_q(B_R)} + \|d\|_{W^{2-1/q}_q(S_R)}
+ \|(g, \bh)\|_{H^1_q(B_R)} + \|\bg\|_{L_q(B_R)})
\end{aligned}\end{equation}
for some constant $C > 0$. 
\end{thm}
\begin{proof} From Theorem \ref{thm:rs.1}, problem
\eqref{4.2} with $k = k_0+4$ admits  unique solutions
$\bv_{k_0} \in H^2_q(B_R)^N$, $\fq_{k_0} \in H^1_q(B_R)$, 
and $\eta_{k_0} \in W^{3-1/q}_q(S_R)$ possessing the estimate:
\begin{equation}\label{4.7}\begin{aligned}
&\|\bv_{k_0}\|_{H^2_q(B_R)} + \|\nabla \fq_{k_0}\|_{L_q(B_R)}
+ \|\eta_{k_0}\|_{W^{3-1/q}_q(S_R)} \\
&\quad 
\leq C(\|\bff\|_{L_q(B_R)} + \|d\|_{W^{2-1/q}_q(S_R)}
+ \|(g, \bh)\|_{H^1_q(B_R)} + \|\bg\|_{L_q(B_R)})
\end{aligned}\end{equation}
for some constant $C$.  Thus, for any $k \in \BR$ with $|k| < k_0+4$, we consider the
unique solvability of the equations: 
\begin{equation}\label{4.8}\begin{aligned}
ik\bw - \DV(\mu\bD(\bw) - \fr\bI) = \bff,
\quad \dv\bw = 0&&\quad&\text{in $B_R$}, \\
ik\zeta + \CM\zeta-(\CA\bw)\cdot\bn= d&&\quad&\text{on $S_R$}, \\
(\mu\bD(\bw) - \fr\bI)\bn - \sigma(\CB_R\zeta)\bn= 0&&\quad
&\text{on $S_R$},
\end{aligned}\end{equation}
where we have set $\bff = i(k-k_0)\bv_{k_0}$ and 
$d=i(k_0-k)\eta_{k_0}$.  In fact, if we set $\bv = \bv_{k_0}+\bw$, 
$\fq=\fq_{k_0}+\fr$, and $\eta = \eta_{k_0}+ \zeta$,
then $\bv$, $\fq$ and $\eta$ are unique solutions of equations \eqref{4.2}. 

In what follows, we study the unique solvability of equations
\eqref{4.8} in the case where $\bff \in L_q(B_R)$ and 
$d \in W^{2-1/q}_q(S_R)$ are arbitrary. To solve \eqref{4.8}, 
it is convenient to study the functional analytic form of
\eqref{4.8}, and so we eliminate the pressure term
$\fr$ and the divergence condition $\dv\bw=0$ in $B_R$. 
Given $\bv \in H^2_q(B_R)^N$ and $\zeta \in W^{3-1/q}_q(S_R)$, let 
$K=K(\bv, \zeta) \in H^1_q(B_R)$ be the unique solution of the weak Dirichlet problem:
\begin{equation}\label{4.10}
(\nabla K, \nabla\varphi)_{B_R} = (\DV(\mu\bD(\bv))-\nabla\dv\bv, \nabla\varphi)_{B_R}
\quad\text{for any $\varphi \in \hat H^1_{q',0}(B_R)$}
\end{equation}
subject to 
\begin{equation}\label{4.10*}
K = <\mu\bD(\bv)\bn, \bn> - \sigma\CB\zeta - \dv\bv
\quad\text{on $S_R$},
\end{equation} 
where we have set
$$\hat H^1_{q', 0}(B_R) = \{\varphi \in L_{q, {\rm loc}}(B_R) \mid 
\nabla\varphi \in L_q(B_R)^N, \quad \varphi|_{S_R}=0\}$$
and $q' = q/(q-1)$.  In view of Poincar\'e's inequality, $\hat H^1_{q',0}(B_R) = 
H^1_{q', 0}(B_R) = \{\varphi \in H^1_{q'}(B_R) \mid \varphi|_{S_R}=0\}$. 
Instead of \eqref{4.8}, we consider the equations:
\begin{equation}\label{4.9}\begin{aligned}
ik\bw - \DV(\mu\bD(\bw) - K(\bw, \zeta)\bI) = \bff&&\quad&\text{in $B_R$}, \\
ik\zeta + \CM\zeta-(\CA\bw)\cdot\bn= d&&\quad&\text{on $S_R$}, \\
(\mu\bD(\bw) - K(\bw, \zeta)\bI)\bn - \sigma(\CB_R\zeta)\bn= 0&&\quad
&\text{on $S_R$}.
\end{aligned}\end{equation}
In view of the boundary condition \eqref{4.10*} for $K(\bw, \zeta)$, that $\bw$ and $\zeta$ satisfy
the third equation of equations \eqref{4.9} is equivalent to 
\begin{equation}\label{4.9*}
(\mu \bD(\bw)\bn)_\tau = 0 \quad\text{and}\quad \dv \bw = 0 \quad\text{on $S_R$},
\end{equation}
where $\bd_\tau = \bd - <\bd, \bn>\bn$ for any $N$-vector $\bd$. 
Let $J_q(B_R)$ be a solenoidal space defined by setting
$$J_q(B_R) = \{\bv \in L_q(B_R) \mid (\bv, \nabla \varphi)_{B_R} = 0
\quad\text{for any $\varphi \in \hat H^1_{q',0}(B_R)$}\}. 
$$
Obviously,  for $\bv \in H^1_q(B_R)$, in order that $\dv\bv= 0$ in $B_R$,
it is necessary and sufficient that $\bv \in J_q(B_R)$.  For any $\bff \in L_q(B_R)^N$, 
let $\psi \in H^1_{q,0}(B_R)$ be a unique solution of the weak Dirichlet problem:
$$(\nabla\psi, \nabla\varphi)_{B_R} = (\bff, \nabla\varphi)_{B_R}
\quad\text{for any $\varphi \in \hat H^1_{q', 0}(B_R)$}.
$$
Let $\bg =\bff -\nabla\psi$ and inserting this formula into equations \eqref{4.8},
we have
$$\begin{aligned}
ik\bw - \DV(\mu\bD(\bw) - (\fr-\psi)\bI) = \bg,
\quad \dv\bw = 0&&\quad&\text{in $B_R$}, \\
ik\zeta + \CM\zeta-(\CA\bw)\cdot\bn= d&&\quad&\text{on $S_R$}, \\
(\mu\bD(\bw) - (\fr-\psi)\bI)\bn - \sigma(\CB_R\zeta)\bn= 0&&\quad
&\text{on $S_R$}.
\end{aligned}$$
where we have used the fact that $\psi|_{S_R}=0$.  Therefore, we shall solve
equations \eqref{4.8} for $\bff \in J_q(B_R)$ and $d \in W^{2-1/q}_q(S_R)$. 
When $\bff \in J_q(B_R)$, the equations \eqref{4.8} and  \eqref{4.9}
are equivalent. In fact, if $\bw \in H^2_q(B_R)^N$ and $\zeta \in W^{3-1/q}_q(S_R)$ 
satisfy equations \eqref{4.8} with some $\fr \in H^1_q(B_R)$.  Then, for any 
$\varphi \in \hat H^1_{q',0}(B_R)$, we have
\begin{align*}
0 & = (\bff, \nabla\varphi)_{B_R} = (ik\bw - \DV(\mu\bD(\bw)), \nabla\varphi)_{B_R}
+ (\nabla\fr, \nabla\varphi)_{B_R} = (\nabla(\fr- K(\bw, \zeta)), \nabla\varphi)_{B_R},
\end{align*}
where we have used  the fact that $\dv\bw = 0$.  Moreover, from the boundary conditions 
in equations \eqref{4.8} and \eqref{4.10*}, it follows that
$$\fr- K(\bw, \zeta) = <\mu\bD(\bw) \bn, \bn> - \sigma \CB_R\zeta- K(\bw, \zeta)
= \dv\bw = 0
$$
on $S_R$ because $\dv \bw=0$.  Thus, the uniqueness of the solutions to his weak Dirichlet problem 
yields that 
$\fr = K(\bw, \zeta)$, and so $\bw$ and $\zeta$ satisfy equations \eqref{4.9}.
Conversely, let $\bw \in H^2_q(B_R)^N$ and $\zeta \in W^{3-1/q}_q(S_R)$ be 
solutions of equations \eqref{4.9}.  For any $\varphi \in \hat H^1_{q', 0}(B_R)$, 
we have
\begin{align*}
0 &= (\bff, \nabla\varphi)_{B_R} = ik(\bw, \nabla\varphi)_{B_R} 
-(\DV(\mu\bD(\bw)), \nabla\varphi)_{B_R} + (\nabla K(\bw, \zeta), \nabla\varphi)_{B_R} \\
& = -ik(\dv\bw, \varphi)_{B_R} - (\nabla\dv\bw, \nabla\varphi)_{B_R}
\end{align*}
Moreover, from the boundary condition \eqref{4.9*} it follows that
$\dv\bw=0$ on $S_R$.  The uniqueness implies that $\dv\bw=0$ in
$B_R$. Thus, $\bw$, $\fr = K(\bw, \zeta)$ and $\zeta$ are solutions 
of equations \eqref{4.8}.  In particular,  for solutions $\bw$ and $\zeta$ of
equations \eqref{4.9}, we see that $\bw$ satisfies the divergence
condition: $\dv\bw=0$ in $B_R$ automatically. 

From now on, we study the unique existence theorem for equations \eqref{4.9}
for any $\bff \in J_q(B_R)$ and $d\in W^{2-1/q}_q(S_R)$.  To formulate 
problem \eqref{4.9} in a functional analytic setting, we define the 
spaces $\CH_q$, $\CD_q$ and the operator $\bA$ by setting
\begin{align*}
\CH_q &= \{(\bff, d) \mid \bff \in J_q(B_R), \quad d \in W^{2-1/q}_q(S_R)\}, \\
\CD_q &= \{(\bw, \zeta) \in \CH_q \mid \bw \in H^2_q(B_R)^N, \quad
\zeta \in W^{3-1/q}_q(S_R), \quad
(\mu\bD(\bw))_\tau|_{S_R} = 0\}, \\
\bA U &= (\DV(\mu\bD(\bw)-K(\bw, \zeta)\bI), (-\CM\zeta +(\CA\bw)\cdot\bn)|_{S_R})
\quad\text{for $U = (\bw, \zeta) \in \CD_q$},
\end{align*}
where we have used \eqref{4.9*} and $\dv \bw=0$ in the definition of $\CD_q$. 
We write equations \eqref{4.9}  as
\begin{equation}\label{4.14} ik U - \bA U = F \quad\text{in $\CH_q$}.
\end{equation}
In view of Theorem \ref{thm:rs.1}, we see that $k=k_0+4$ is  an element of the resolvent set 
of the operator
$\bA$, and so $(i(k_0+4)\bI - \bA)^{-1}$ exists in 
$\CL(\CH_q, \CD_q)$.  Since $B_R$ is a compact set, it follows from the Rellich compactness
theorem that $(i(k_0+4)\bI - \bA)^{-1}$ is a compact operator from $\CH_q$ into itself.
Thus, in view of Riesz-Schauder theory, in particular, Fredholm alternative principle,
that $k$ belongs to the resolvent set if and only if uniqueness holds for $k$. 
Thus, our task is to prove the uniqueness of solutions to equations \eqref{4.14}.  
Let $U = (\bw, \zeta) \in \CD_q$ satisfy the homogeneous equations:
\begin{equation}\label{4.22} ik U - \bA U = 0 \quad\text{in $\CH_q$}.
\end{equation}
Namely, $(\bw, \zeta) \in \CD_q$ satisfies equations:
\begin{equation}\label{null:1}\begin{aligned}
ik\bw - \DV(\mu\bD(\bw) - K(\bw, \zeta)\bI) = 0
&&\quad&\text{in $B_R$}, \\
ik\zeta + \CM\zeta-(\CA\bw)\cdot\bn= 0&&\quad&\text{on $S_R$}, \\
(\mu\bD(\bw) - K(\bw, \zeta)\bI)\bn - \sigma(\CB_R\zeta)\bn= 0&&\quad
&\text{on $S_R$}.
\end{aligned}
\end{equation}
We first prove that
\begin{equation}\label{null:2}
(\zeta, 1)_{S_R} = 0, \quad (\zeta, x_j)_{S_R} = 0 \quad\text{for $j=1, \ldots, N$}.
\end{equation}
Integrating the second equation of equations \eqref{null:1} and applying the 
divergence theorem of Gauss gives that
$$0 = ik(\zeta, 1)_{S_R} + (\zeta, 1)_{S_R}|S_R|-\int_{B_R} \dv \CA\bw \,dx  
=(ik + |S_R|)(\zeta, 1)_{S_R},$$
where we have set $|S_R| = \int_{S_R}\,d\omega$ and we have used the fact that
$\dv\bw=0$ in $B_R$. Thus, we have $(\zeta, 1)_{S_R} = 0$. 
Multiplying the second equation of equations \eqref{null:1} with $x_j$,  
integrating the resultant formula over $S_R$ and using the divergence
theorem of Gauss gives that
\begin{align}\label{null:3}
0 = ik(\zeta, x_\ell)_{S_R} + (\zeta, x_\ell)_{S_R}(x_\ell, x_\ell)_{S_R} 
-\int_{B_R} \dv(x_\ell\CA\bw)\,dx,
\end{align}
because  $(x_j, x_\ell)_{S_R}=0$ for $j\not=\ell$.  Since
$$\int_{B_R} \dv(x_\ell\CA\bw)\,dx = \int_{B_R} (\bw_\ell - \frac{1}{|B_R|}\int_{B_R} \bw_\ell\,dx)\,dx
= 0, 
$$
we have $(\zeta, x_\ell)_{S_R} = 0$, because $(x_\ell, x_\ell)_{S_R} = (R^2/N)|S_R| > 0$.
Thus, we have proved \eqref{null:2}.  In particular, $\CM\zeta = 0$ in \eqref{null:1}. 

We now prove that $\bw = 0$.  For this purpose, we first consider the case where
$2 \leq q < \infty$. Since $B_R$ is bounded, $\CD_q \subset \CD_2$.  
Multiplying the first equation of \eqref{null:1} with $\bw$ and integrating the 
resultant formula over $B_R$ and using the divergence theorem of Gauss gives
that  
\begin{align*}
0 = ik\|\bw\|_{L_2(B_R)}^2 - \sigma(\CB_R\zeta, \bn\cdot\bw)_{S_R} + \frac{\mu}{2}
\|\bD(\bw)\|_{L_2(B_R)}^2,
\end{align*}
because $\dv\bw = 0$ in $B_R$.  By the second equation of \eqref{null:1}
with $\CM\zeta=0$, we have
$$\sigma(\CB_R\zeta, \bn\cdot\bw)_{S_R}
= \sigma (\CB_R\zeta, ik\zeta)_{S_R}
+ \sum_{k=1}^N\frac{1}{|B_R|}\int_{B_R}w_j\,dt(\CB_R\zeta, R^{-1}x_j)_{S_R}
$$
where we have used $\bn =R^{-1}x= R^{-1}(x_1,\ldots, x_N)$ for $x \in S_R$.
Thus, 
$$(\CB_R\zeta, x_j)_{S_R} = (\zeta, (\Delta_{S_R}+\frac{N-1}{R^2})x_j)_{S_R} =0.
$$
Moreover, since $\zeta$ satisfies \eqref{null:2}, we know that
$$-(\CB_R\zeta, \zeta)_{S_R} \geq c\|\zeta\|_{L_2(S_R)}^2
$$
for some positive constant $c$, and therefore \eqref{null:3} implies $\bw=0$. 

Now the first equation of \eqref{null:1} yields $\nabla K(\bw,\zeta)=0$, so that $K(\bw,\zeta)$ is 
constant. Integration of the third equation of \eqref{null:1} over $S_R$ combined with \eqref{null:2}
shows that this constant is $0$, that is, $K(\bw,\zeta)=0$.

Finally, the third equation of  \eqref{null:1} yields that $\CB_R\zeta=0$ on 
$S_R$, and so by \eqref{null:2} we have $\zeta=0$.  This completes the 
proof of the uniqueness in the case where $2 \leq q < \infty$.
In particular, we have the unique existence theorem of solutions to
equation \eqref{4.14}.  

We now consider the case where $1 < q < 2$.  Let $\bff$ be any element in  $J_{q'}(B_R)$ 
and let $V = (\bv, \eta) \in \CD_{q'}$ be a solution of the equation:
$$-ik V - \bA V= (\bff, 0) \quad\text{in $\CH_{q'}$}.
$$
The existence of such $V$ has already been proved above.  Since $d=0$, we see that 
$\eta$ satisfies the relations:
$$(\eta, 1)_{S_R} = 0, \quad (\eta, x_j)_{S_R} = 0 \quad\text{for $j=1, \ldots, N$},
$$
and so $\CM\eta = 0$.  Using the divergence theorem of Gauss, we have
\begin{align*}
(\bw, \bff)_{B_R} &= (\bw, -ik \bv - \DV(\mu\bD(\bv) - K(\bv, \eta)\bI))_{B_R} \\
&= (ik\bw, \bv)_{B_R} -(\bw, (\mu\bD(\bv) - K(\bv, \eta)\bI)\bn)_{S_R} 
+\frac{\mu}{2}(\bD(\bw), \bD(\bv))_{B_R} \\
&= (\DV(\mu(\bD(\bw) - K(\bw, \zeta)\bI), \bv)_{B_R} 
-\sigma(\bw\cdot\bn, \CB_R\eta)_{S_R} +\frac{\mu}{2}
(\bD(\bw), \bD(\bv))_{B_R}\\
& = \sigma(\CB_R \zeta, \bn\cdot\bv)_{S_R} -
\sigma(\bw\cdot\bn, \CB_R\eta)_{S_R} \\
& = \sigma(\CB_R\zeta, -ik\eta + \frac{1}{|B_R|}\int_{B_R}\bv\,dy\cdot\bn)_{S_R}
	-\sigma(ik\zeta + \frac{1}{|B_R|}\int_{B_R}\bw\,dy\cdot\bn, \CB_R\eta)_{S_R}.
\end{align*}
Using the fact that $(\CB_R\zeta, x_j)_{S_R} = (x_j, \CB_R\eta)_{S_R} =0$, we have
\begin{align*}(\bw, \bff)_{B_R} & = \sigma ik(\CB_R\zeta, \eta)_{S_R} - \sigma ik(\zeta, \CB_R\eta)_{S_R}
\\
&=\sigma ik\Bigl\{\frac{N-1}{R^2}(\zeta, \eta)_{S_R} - (\nabla_{S_R}\zeta, \nabla_{S_R}\eta)_{S_R}
-\frac{N-1}{R^2}(\zeta, \eta)_{S_R} + (\nabla_{S_R}\zeta, \nabla_{S_R}\eta)_{S_R}\Bigr\}
= 0. 
\end{align*}
For any $\bg \in L_{q'}(B_R)^N$, let $\psi \in \hat H^1_{q', 0}(B_R)$ be a unique solution of the
weak Dirichlet problem:
$$(\nabla\psi, \nabla\varphi)_{B_R} = (\bg, \nabla\varphi)_{B_R}
\quad\text{for any $\varphi \in \hat H^1_{q,0}(B_R)$}.$$
Let $\bff = \bg - \nabla\psi$, and then $\bff \in J_{q'}(B_R)$, and so using
the fact that $\bw \in J_q(B_R)$, we have $(\bw, \bg)_{B_R} = 
(\bw, \bff)_{B_R} + (\bw, \nabla\psi)_{B_R} = 0$.  The arbitrariness of 
$\bg\in L_{q'}(B_R)^N$ implies that $\bw=0$.  Thus, the second equation of \eqref{null:1}
and \eqref{null:2} leads to $\zeta=0$.  This completes the proof of the 
uniqueness in the case where $1 < q < 2$, and therefore the proof of 
Theorem \ref{thm:rs.3}. 
\end{proof}
We now consider the linearized stationary problem:
\begin{equation}\label{eq:s1}\begin{aligned}
\CL\bv-\DV(\mu\bD(\bv) - \fp\bI) = \bff&&\quad&\text{in $B_R$}, \\
\dv\bv = g = \dv\bg&&\quad&\text{in $B_R$}, \\
\CM\rho-(\CA\bv)\cdot\bn=d&&\quad&\text{on $S_R$}, \\
(\mu\bD(\bv) - \fp\bI)\bn - \sigma(\CB_R\rho)\bn=\bh
&&\quad&\text{on $S_R$}.
\end{aligned}\end{equation}
We shall prove the following theorem.
\begin{thm}\label{thm:s1} Let $1 < q < \infty$.  Then, for any 
$\bff \in L_q(B_R)^N$, $d \in W^{2-1/q}_q(S_R)$, $g \in H^1_q(B_R)$, 
$\bg \in L_q(B_R)^N$, and $\bh \in H^1_q(B_R)^N$, problem \eqref{eq:s1}
admits unique solutions $\bv \in H^2_q(B_R)^N$, $\fp \in H^1_q(B_R)$, 
and $\rho \in W^{3-1/q}_q(S_R)$ possessing the estimate:
\begin{equation}\label{eq:s2}\begin{aligned}
&\|\bv\|_{H^2_q(B_R)} + \|\fp\|_{H^1_q(B_R)} + \|\rho\|_{W^{3-1/q}_q(S_R)}
\\
&\quad \leq C(\|\bff\|_{L_q(B_R)} + \|d\|_{W^{2-1/q}_q(S_R)} 
+ \|(g, \bh)\|_{H^1_q(B_R)} + \|\bg\|_{L_q(B_R)})
\end{aligned}\end{equation}
for some constant $C > 0$. 
\end{thm}
\begin{proof}
The strategy of the proof is the same as that  of Theorem \ref{thm:rs.3}. 
Since $\CL\bv$, $\CM\rho$, and $|B_R|^{-1}\int_{B_R}\bv\,dy$ are 
lower order perturbations, choosing $k_0 > 0$ large enough, 
the generalized resolvent problem:
\begin{equation}\label{eq:s3}\begin{aligned}
ik_0\bv + \CL\bv-\DV(\mu\bD(\bv) - \fp\bI) = \bff&&\quad&\text{in $B_R$}, \\
\dv\bv = g = \dv\bg&&\quad&\text{in $B_R$}, \\
ik_0\rho + \CM\rho-(\CA\bv)\cdot\bn=d&&\quad&\text{on $S_R$}, \\
(\mu\bD(\bv) - \fp\bI)\bn - \sigma(\CB_R\rho)\bn=\bh
&&\quad&\text{on $S_R$}.
\end{aligned}\end{equation}
admits unique solutions: $\bv \in H^2_q(B_R)^N$, $\fp \in H^1_q(B_R)$, 
and $\rho \in W^{3-1/q}_q(S_R)$ possessing the estimate \eqref{eq:s2}.
Of course, the constant $C$ in \eqref{eq:s2} depends on $k_0$ in this case,
but $k_0$ is fixed, and so we can say that  $C$ in \eqref{eq:s2} is 
some fixed constant. The essential part of the proof is to show the unique
existence of solutions to equations \eqref{eq:s1} with $g=\bg=\bh=0$,
that is 
\begin{equation}\label{eq:s1**}\begin{aligned}
\CL\bv-\DV(\mu\bD(\bv) - \fp\bI) = \bff&&\quad&\text{in $B_R$}, \\
\dv\bv = 0&&\quad&\text{in $B_R$}, \\
\CM\rho-(\CA\bv)\cdot\bn=d&&\quad&\text{on $S_R$}, \\
(\mu\bD(\bv) - \fp\bI)\bn - \sigma(\CB_R\rho)\bn=0
&&\quad&\text{on $S_R$}.
\end{aligned}\end{equation}
And then, the uniqueness of the reduced problem in the $L_2$ framework
implies the unique existence of solutions as was studied in the proof 
Theorem \ref{thm:rs.3}.  Thus, we define the reduced problem corresponding
to equations \eqref{eq:s1}.  For $\bv \in H^2_q(B_R)^N$ and 
$\rho \in W^{3-1/q}_q(S_R)$, let $K = K(\bv, \rho) \in H^1_q(B_R)$
be the unique solution of the weak Dirichlet problem:
\begin{equation}\label{eq:s4}
(\nabla K, \nabla\varphi)_{B_R} = (\DV(\mu\bD(\bv)-\CL\bv-
\nabla\dv\bv, \nabla\varphi)_{B_R}
\quad\text{for any $\varphi \in \hat H^1_{q',0}(B_R)$}, 
\end{equation}
subject to the boundary condition:
\begin{equation}\label{eq:s5}
K = <\mu\bD(\bv)\bn, \bn> - \sigma\CB_R\rho - \dv\bv
\quad\text{on $B_R$}.
\end{equation}
Then the reduced problem corresponding to 
problem \eqref{eq:s1} with $g=\bg=\bh=0$ is given by
the following equations:
\begin{equation}\label{eq:s1*}\begin{aligned}
\CL\bv-\DV(\mu\bD(\bv) - K(\bv, \rho)\bI) = \bff&&\quad&\text{in $B_R$}, \\
\CM\rho-(\CA\bv)\cdot\bn=d&&\quad&\text{on $S_R$}, \\
(\mu\bD(\bv) - K(\bv, \rho)\bI)\bn - \sigma(\CB_R\rho)\bn=0
&&\quad&\text{on $S_R$}.
\end{aligned}\end{equation}
Then, for $\bff \in J_q(B_R)$ and $d \in W^{2-1/q}_q(S_R)$, problems
\eqref{eq:s1**} and \eqref{eq:s1*} are equivalent. In fact,
if problem \eqref{eq:s1**} admits unique solutions
$\bv \in H^2_q(B_R)^N$, $\fp \in H^1_q(B_R)$ and $\rho \in W^{3-1/q}_q(S_R)$, 
then for any $\varphi \in \hat H^1_{q',0}(B_R)$, we have
\begin{align*}
0 = (\bff, \nabla\varphi)_{B_R} = (\CL\bv - \DV(\mu\bD(\bv)), 
\nabla\varphi)_{B_R}
+ (\nabla\fp, \nabla\varphi)_{B_R} 
= (\nabla(\fp - K(\bv, \rho)), \nabla\varphi)_{B_R} 
\end{align*}
because $\dv\bv =0$ in $B_R$. Moreover, from the boundary conditions in \eqref{eq:s1**}
and \eqref{eq:s5} it follows that 
$$\fp-K(\bv, \rho) = <\mu\bD(\bv)\bn, \bn> - \sigma\CB_R\rho - 
K(\bv, \rho) = \dv \bv = 0
$$
on $S_R$.  The uniqueness of the weak Dirichlet problem leads to 
$\fp = K(\bv, \rho)$, and therefore $\bv$ and $\rho$ satisfy
equations \eqref{eq:s1*}.  Conversely, if $\bv \in H^2_q(B_R)^N$ and
$\rho \in W^{3-1/q}_q(S_R)$ satisfy the equations \eqref{eq:s1*}, 
then for any $\varphi \in \hat H^1_{q', 0}(B_R)$ we have
\begin{align*}
0 = (\bff, \nabla\varphi)_{B_R} = (\CL\bv-\DV(\mu\bD(\bv), \nabla\varphi)_{B_R}
+ (\nabla K(\bv, \rho), \nabla\varphi)_{B_R}
= (\nabla\dv\bv, \nabla\varphi)_{B_R}.
\end{align*}
Moreover, the boundary conditions of \eqref{eq:s1*} and \eqref{eq:s5} gives that
$$\dv\bv = <\mu\bD(\bv)\bn, \bn>- \sigma\CB_R\rho - K(\bv, \rho) = 0.
$$
The uniqueness of the weak Dirichlet problem yields that $\dv\bv=0$, 
and therefore, $\bv$, $\fp = K(\bv, \rho)$ and $\rho$ are solutions of 
equations \eqref{eq:s1**}. 

Finally, we show the uniqueness of equations \eqref{eq:s3} in the $L_2$-framework,
which yields Theorem \ref{thm:s1}.  Let $\bv \in H^2_2(B_R)^N$ and 
$\rho \in W^{5/2}_2(S_R)$ satisfy the homogeneous equations:
\begin{equation}\label{eq:s6}\begin{aligned}
\CL\bv-\DV(\mu\bD(\bv) - K(\bv, \rho)\bI) = 0&&\quad&\text{in $B_R$}, \\
\CM\rho-(\CA\bv)\cdot\bn=0&&\quad&\text{on $S_R$}, \\
(\mu\bD(\bv) - K(\bv, \rho)\bI)\bn - \sigma(\CB_R\rho)\bn=0
&&\quad&\text{on $S_R$}.
\end{aligned}\end{equation}
Note that $\dv\bv=0$ in $B_R$. Employing the same argument as in
the proof of Theorem \ref{thm:rs.3}, we have
\begin{equation}\label{null:s1} (\rho, 1)_{S_R} = 0, 
\quad (\rho, x_j)_{S_R}=0 \quad\text{for $j=1, \ldots, N$}.
\end{equation}
In particular, $\CM\rho=0$.  Multiplying the first equation with $\bv$,
integrating the resultant formula on $B_R$ and using the 
divergence theorem of Gauss gives that
$$0 = (\CL\bv, \bv)_{B_R} + 
\sigma(\CB_R\rho, \bn \cdot\bv)_{S_R} + \frac{\mu}{2}\|\bD(\bv)\|_{L_2(B_R)}^2,
$$
because $(K(\bv, \rho), \dv\bv) = 0$ as follows from $\dv\bv=0$ in $B_R$. 
From \eqref{op:1} it follows that 
$$(\CL\bv, \bv)_{B_R} = \sum_{k=1}^M|(\bv, \bp_k)_{B_R}|^2.
$$
From the second equation of  \eqref{eq:s6} with $\CM\rho=0$ it follows that
$$(\CB_R\rho, \bn\cdot\bv)_{S_R} = \sum_{j=1}^N R^{-1}
(\CB_R\rho, x_j)_{S_R}\frac{1}{|B_R|}\int_{B_R}v_j\,dy = 0.
$$
Combining these formulas yields that 
$$0 = \sum_{k=1}^M|(\bv, \bp_k)_{B_R}|^2 + \frac{\mu}{2}\|\bD(\bv)\|_{L_2(B_R)}^2,
$$
which leads to $\bD(\bv) = 0$ and $(\bv, \bp_k)_{B_R} = 0$ for $k = 1, \ldots, M$.
Thus, we have $\bv=0$.  From the first equation of \eqref{eq:s6}, 
we have $\nabla K(\bv, \rho) = 0$, and so $K(\bv, \rho) = c$ with some constant $c$.
From the boundary condition of \eqref{eq:s6}, we have
$\sigma\CB\rho=-c$ on $B_R$.  Integrating this formula on $S_R$ and using 
the fact $(\rho, 1)_{S_R}=0$ in 
\eqref{null:s1} gives that $c=0$.  Thus, $\CB_R\rho=0$ on $S_R$, but 
we know \eqref{null:s1}, and so 
$$0 = -(\CB_R\rho, \rho)_{S_R} \geq c\|\rho\|_{L_2(S_R)}^2
$$
for some constant $c > 0$, which shows that $\rho=0$.    This completes the proof of the 
uniqueness in the $L_2$ framework, the proof of Theorem \ref{thm:s1}.
\end{proof}
{\bf Proof of Theorem \ref{thm:linear.main1}.} We now prove Theorem \ref{thm:linear.main1}.  Let 
$\bu_\psi$, $\fp_\psi$ and $\rho_\psi$ be functions given in Theorem \ref{thm:rs.2}
which are solutions of equations \eqref{4.5}.  Notice that $\psi(ik) = 1$ for 
for $|k| \geq k_0+4$ and $\psi(ik) = 0$ for $|k| \leq k_0+3$.  For 
$k \in \BZ$ with $1\leq |k| \leq k_0+3$, let 
\begin{align*}
\bff = \CF_\BT[\bF](ik), \quad g=\CF_\BT[G](ik), \quad \bg = \CF_\BT[\bG](ik), 
\quad d = \CF_\BT[D](ik), \quad \bh = \CF_\BT[\bH](ik)
\end{align*}
in  equations \eqref{4.2}, and we write solutions $\bv$, $\fq$ and $\eta$ as
$\bv_k=\bv$, $\fq_k=\fq$ and $\eta_k = \eta$.  Let 
\begin{align*}
\bu_k &= e^{ikt}\bv_k, \quad \fp_k = e^{ikt}\fq_k, \quad \rho_k = e^{ikt}\eta_k, 
\end{align*}
and then, $\bu_k$, $\fp_k$ and $\rho_k$ satisfy the equations:
\begin{equation}\label{4.1k}\begin{aligned}
\pd_t\bu_k -\DV(\mu\bD(\bu_k)-\fp_k\bI) = e^{ikt}\CF_\BT[\bF](ik)&&\quad&\text{in $B_R$}, \\
\dv \bu_k =e^{ikt}\CF_\BT[G](ik) = \dv(e^{ikt}\CF_\BT[\bG](ik))&&\quad&\text{in $B_R$}, \\
\pd_t\rho_k + \CM\rho_k-(\CA\bu_k)\cdot\bn=e^{ikt}\CF_\BT[D](ik)&&\quad&\text{on $S_R$},\\
(\mu\bD(\bu_k)-\fp_k\bI)\bn - (\CB_R\rho_k)\bn=e^{ikt}\CF_\BT[\bH](ik)&&\quad&
\text{on $S_R$}.
\end{aligned}\end{equation}
Let $\bff = \bF_S$, $d=D_S$, $g=G_S$, $\bg = \bG_S$ and $\bh = \bH_S$ in equations
\eqref{eq:s1}, and let $\bv$, $\fp$ and $\rho$ be unique solutions of equations \eqref{eq:s1}.
We write $\bu_S = \bv$, $\fp_S = \fp$ and $\rho_S = \rho$.  Under these preparations, 
we set
\begin{align*}
\bu &= \bu_S + \sum_{1 \leq |k| \leq k_0+3} \bu_k + \bu_\psi, \\
\fp & = \fp_S + \sum_{1 \leq |k| \leq k_0+3} \fp_k + \fp_\psi, \\
\rho & = \rho_S + \sum_{1 \leq |k| \leq k_0+3} \rho_k + \rho_\psi
\end{align*}
and then $\bu$, $\fp$ and $\rho$ are unique solutions of equations \eqref{4.1}.
Moreover, by Theorem \ref{thm:rs.2}, Theorem \ref{thm:rs.3}, and 
Theorem \ref{thm:s1}, we see that $\bu$, $\fp$ and $\rho$ satisfy the estimate
\eqref{main:est1}.  In fact, for $f = f_S + \sum_{1\leq |k| \leq k_0+3}
e^{ikt}f_k + f_\psi$, we have the following estimates: 
\begin{align*}
\|f\|_{L_p((0, 2\pi), X)} &\leq \|f_S\|_{L_p((0, 2\pi), X)} 
+ \sum_{1 \leq |k| \leq k_0+3}\|e^{ikt}f_k\|_{L_p((0, 2\pi), X)}
+ \|f_\psi\|_{L_p((0, 2\pi), X)}\\
&\leq (2\pi)^{1/p}\|f_S\|_X + (2\pi)^{1/p}\sum_{1 \leq |k| \leq k_0+3}\|f_k\|_X + 
\|f_\psi\|_{L_p((0, 2\pi), X)}, \\
\|\pd_t f\|_{L_p((0, 2\pi), X)} &\leq 
\sum_{1 \leq |k| \leq k_0+3}\|(ik)e^{ikt}f_k\|_{L_p((0, 2\pi), X)}
+ \|\pd_tf_\psi\|_{L_p((0, 2\pi), X)}\\
&\leq (2\pi)^{1/p}(k_0+3)\sum_{1 \leq |k| \leq k_0+3}\|f_k\|_X + 
\|\pd_tf_\psi\|_{L_p((0, 2\pi), X)}.
\end{align*}
By H\"older's inequality, we have
$$\|f_S\|_{L_p((0, 2\pi), X)} \leq 2\pi\|f\|_{L_p((0, 2\pi), X)}, \quad
\|e^{ikt}\CF_\BT[f](ik)\|_{L_p((0, 2\pi), X)}  \leq 2\pi\|f\|_{L_p((0, 2\pi), X)},
$$
and for any UMD Banach space $X$, 
using Lemma \ref{lem:rp.2} and transference theorem, 
Theorem \ref{ESK},  we have 
\begin{align*}
\|f_\psi\|_{L_p((0, 2\pi), X)}, & \leq C\|\psi\|_{H^1_\infty}\|f\|_{L_p((0, 2\pi), X)}, \\
\|\pd_t f_\psi\|_{L_p((0, 2\pi), X)} &\leq C\|\psi\|_{H^1_\infty}\|\pd_tf\|_{L_p((0, 2\pi), X)}, \\
\|\Lambda^{1/2}f_\psi\|_{L_p((0, 2\pi), X)}
& \leq \|\CF^{-1}_\BT[((ik)^{1/2}/(1+k^2)^{1/4})\psi(ik)(1+k^2)^{1/4}\CF_\BT[f](ik)]\|_{L_p
((0, 2\pi), X)} \\
& \leq C\bigl(\sum_{\ell=0,1}\sup_{\lambda \in \BR}|(\lambda\frac{d}{d\lambda})^\ell(
((i\lambda)^{1/2}/(1+\lambda^2)^{1/4})\psi(i\lambda))|\bigr) \|f\|_{H^{1/2}_p((0, 2\pi), X)}.
\end{align*}
\subsection{On linearized problem of two-phase problem}
In this subsection, we consider the linear equations: 
\begin{equation}\label{4.21}\begin{aligned}
\pd_t\bu_\pm-\DV(\mu\bD(\bu_\pm)-\fp_\pm\bI) = \bF_\pm 
&&\quad&\text{in $\Omega_\pm\times(0, 2\pi)$}, \\
\dv \bu_\pm =G_\pm = \dv\bG_\pm&&\quad&\text{in $\Omega_\pm\times(0, 2\pi)$}, \\
\pd_t\rho + \CM\rho-(\CA\bu)\cdot\bn=D&&\quad&\text{on $S_R\times(0, 2\pi)$},\\
[[\mu\bD(\bu)-\fp\bI)]]\bn - (\CB_R\rho)\bn=\bH&&\quad&
\text{on $S_R\times(0, 2\pi)$}, \\
[[\bu]]=0 &&\quad&
\text{on $S_R\times(0, 2\pi)$}, \\
\bu_-=0&&\quad&\text{on $S\times(0, 2\pi)$}.
\end{aligned}\end{equation}
where $\Omega_+ = B_R$, $\Omega_- = \Omega\setminus(B_R\cup S_R)$, and 
$\CM$, $\CA$ and $\CB_R$ are the linear operators defined in \eqref{op:1}. 
We shall prove the unique existence theorem of $2\pi$-periodic solutions of equations
\eqref{4.21}.  Our main result in this section is stated as follows.
\begin{thm}\label{thm:linear.main2} Let $1 < p, q < \infty$.  Then, for any 
$\bF_\pm$, $D$, $G_\pm$, $\bG_\pm$ and $\bH$ with
\begin{align*}
\bF_\pm & \in L_{p, {\rm per}}((0, 2\pi), L_q(\Omega_\pm)^N), \quad 
D \in L_{p, {\rm per}}((0, 2\pi), W^{2-1/q}_q(S_R)) \\
G_\pm & \in L_{p, {\rm per}}((0, 2\pi), H^1_q(\Omega_\pm)) 
\cap H^{1/2}_{p, {\rm per}}((0, 2\pi), L_q(\Omega_\pm)), 
\quad \bG_\pm \in H^1_{p, {\rm per}}((0, 2\pi),  L_q(\Omega_\pm)^N), \\
\bH& \in L_{p, {\rm per}}((0, 2\pi), H^1_q(\Omega)^N) \cap H^{1/2}_{p, {\rm per}}((0, 2\pi), L_q(\Omega)^N), 
\end{align*}
problem \eqref{4.1} admits unique solutions $\bu_\pm$, $\fp_\pm$ and $\rho$ with 
\begin{align*}
\bu_\pm & \in L_{p, {\rm per}}((0, 2\pi), H^2_q(\Omega_\pm)^N) \cap H^1_{p, {\rm per}}
((0, 2\pi), L_q(\Omega_\pm)^N), \\
\fp_\pm & \in L_{p, {\rm per}}((0, 2\pi), H^1_q(\Omega_\pm)), \quad 
\sum_{\pm}\int_{\Omega_\pm}\fp_\pm(x,t)\,dx=0 \ \text{ for }t\in(0,2\pi), \\
\rho & \in L_{p, {\rm per}}((0, 2\pi), W^{3-1/q}_q(S_R)) \cap H^1_{p, {\rm per}}
((0, 2\pi), W^{2-1/q}_q(S_R))
\end{align*}
possessing the estimate:
\begin{equation}\label{main:est2}\begin{aligned}
&\sum_\pm\{ \|\bu_\pm\|_{L_p((0, 2\pi), H^2_q(\Omega_\pm))} 
+ \|\pd_t\bu_\pm\|_{L_p((0, 2\pi), L_q(\Omega_\pm))} 
+ \|\nabla \fp_\pm\|_{L_p((0, 2\pi), L_q(\Omega_\pm))}\}  \\
&\quad + \|\rho\|_{L_p((0, 2\pi), W^{3-1/q}_q(S_R))} + 
\|\pd_t\rho\|_{L_p((0, 2\pi), W^{2-1/q}_q(S_R))} \\
&\qquad \leq C\{\sum_{\pm}\|\bF_\pm\|_{L_p((0, 2\pi), L_q(\Omega_\pm))} 
+ \|D\|_{L_p((0, 2\pi), W^{2-1/q}_q(S_R))} 
+ \sum_\pm\|\pd_t\bG_\pm\|_{L_p((0, 2\pi), L_q(\Omega_\pm))} \\
&\quad\qquad + \sum_\pm \|G_\pm\|_{L_p((0, 2\pi), H^1_q(\Omega_\pm))}
+ \|G_\pm\|_{H^{1/2}_p((0, 2\pi), L_q(\Omega_\pm))} \\
&\qquad\qquad 
+ \|\bH\|_{L_p((0, 2\pi), H^1_q(\Omega))}
+ \|\bH\|_{H^{1/2}_p((0, 2\pi), L_q(\Omega))} 
\}\end{aligned}\end{equation} 
for some constant $C > 0$. 
\end{thm}
To prove Theorem \ref{thm:linear.main2}, 
the strategy  is the same as in the proof of Theorem \ref{thm:linear.main1}.
Therefore, we first consider
the $\CR$-solver of the generalized resolvent problem:
\begin{equation}\label{4.22}\begin{aligned}
ik\bv_\pm -\DV(\mu\bD(\bv_\pm)-\fq_\pm\bI) = \bff_\pm&&\quad&\text{in $\Omega_\pm$}, \\
\dv \bv_\pm =g_\pm = \dv\bg_\pm&&\quad&\text{in $\Omega_\pm$}, \\
ik\eta + \CM\eta-(\CA\bv_+)\cdot\bn=d&&\quad&\text{on $S_R$},\\
[[\mu\bD(\bv)-\fq\bI]]\bn - (\CB_R\eta)\bn=\bh&&\quad&
\text{on $S_R$}, \\
[[\bv]]=0&&\quad&
\text{on $S_R$}, \\
\bv_- = 0&&\quad&\text{on $S$}
\end{aligned}\end{equation}
for $k \in \BR$.  From Theorem 2.1.4  in Shibata and Saito \cite{SS1} 
we know the following theorem concerned with the existence of an $\CR$-solver
of problem \eqref{4.21}.
\begin{thm}\label{thm:rs.21} Let $1 < q < \infty$ and let $\BR_{k_0}
= \BR\setminus(-k_0, k_0)$.  Let 
\begin{align*}
X_q(\dot\Omega) & = \{(\bff, d, \bh, g, \bg) \mid 
\bff \in L_q(\dot\Omega), \enskip d \in W^{2-1/q}_q(S_R), \enskip
\bh \in H^1_q(\Omega)^N, \enskip g \in H^1_q(\dot\Omega), 
\enskip \bg \in L_q(\dot\Omega)^N\}, \\
\CX_q(\dot\Omega) & = \{F=(F_1, F_2, \ldots, F_7) \mid F_1,  F_7 
\in L_q(\dot\Omega)^N, \enskip F_2 \in W^{2-1/q}_q(S_R), \enskip 
F_3 \in L_q(\Omega)^N, \enskip 
F_4 \in H^1_q(\Omega)^N, \\
&\phantom{= \{F=(F_1, F_2, \ldots, F_7) \mid F_1, F_3, F_7 
\in L_q(B_R)^N, a}\,\,
 F_5 \in L_q(\dot\Omega), \enskip
F_6 \in H^1_q(\dot\Omega)\}.
\end{align*}
Then, there exist a constant $k_0 > 0$ and operator families $\CA(ik)$, $\CP(ik)$, and  
$\CH(ik)$ with 
\begin{align*}
\CA(ik) & \in C^1(\BR_{k_0}, \CL(\CX_q(\dot\Omega), H^2_q(\dot\Omega)^N)), \\
\quad 
\CP(ik) &\in C^1(\BR_{k_0}, \CL(\CX_q(\dot\Omega), \dot H^1_q(\dot\Omega))), \\
\quad
\CH(ik) &\in C^1(\BR_{k_0}, \CL(\CX_q(\dot\Omega), W^{3-1/q}_q(S_R)))
\end{align*}
such that for any $(\bff, d, \bh, g, \bg)$ and $k \in \BR_{k_0}$, 
$\bv = \CA(ik)\CF_k$, $\fq = \CP(ik)\CF_k$ and $\eta
= \CH(ik)\CF_k$, where 
$$\CF_k = (\bff, d, (ik)^{1/2}\bh, \bh, (ik)^{1/2}g, g, ik \bg),$$
are unique solutions of equations \eqref{4.22}, and 
\begin{equation}\label{4.23}\begin{aligned}
\CR_{\CL(\CX_q(\dot\Omega), H^{2-m}_q(\dot\Omega)^N)}
(\{(k\pd_k)^\ell((ik)^{m/2}\CA(ik))\mid k \in \BR_{k_0}\}) &\leq r_b, \\
\CR_{\CL(\CX_q(\dot\Omega), L_q(\dot\Omega)^N)}(\{(k\pd_k)^\ell \nabla\CP(ik) \mid 
k \in \BR_{k_0}\}) &\leq r_b, \\
\CR_{\CL(\CX_q(\dot\Omega), W^{3-n-1/q}_q(S_R))}
(\{(k\pd_k)^\ell((ik)^n\CH(ik))\mid k \in \BR_{k_0}\}) &\leq r_b
\end{aligned}\end{equation}
for $\ell=0,1$, $m=0,1,2$ and $n=0,1$ with some constant $r_b$. 
\end{thm}
\begin{remark} 
\begin{itemize}
\item[\thetag1] Here $f \in L_q(\dot\Omega)$ means that $f_\pm \in L_q(\Omega_\pm)$, 
and $f \in H^1_q(\dot\Omega)$ means that $f_\pm \in H^1_q(\Omega_\pm)$, and 
we set 
$$\|f\|_{L_q(\dot\Omega)} = \sum_\pm\|f_\pm\|_{L_q(\Omega_\pm)}, \quad 
\|f\|_{H^1_q(\dot\Omega)} = \sum_\pm\|f_\pm\|_{H^1_q(\Omega_\pm)}.
$$
Moreover, we define
$$
\dot H^1_q(\dot\Omega) = \Big\{\theta \in H^1_q(\dot\Omega) \mid
\int_{\dot\Omega} \theta \,dx = 0\Big\}.
$$
\item[\thetag2] For $f$ defined on $\dot\Omega$, we set $f_\pm = f|_{\Omega_\pm}$ and for 
$f_\pm$ defined on $\Omega_\pm$, we set $f = f_\pm$ on $\Omega_\pm$. 
The functions $F_1$, $F_2$, $F_3$, $F_4$, $F_5$, $F_6$, and $F_7$ are 
variables corresponding to $\bff$, $d$, $(ik)^{1/2}\bh$, $\bh$, $(ik)^{1/2}g$,
$g$, and $ik\, \bg$, respectively. 
\item[\thetag3]
We define the norm $\|\cdot\|_{\CX_q(\Omega)}$ by setting
$$\|(F_1, \ldots, F_7)\|_{\CX_q(\Omega)}
= \|(F_1, F_5, F_7)\|_{L_q(\dot\Omega)} 
+ \|F_2\|_{W^{2-1/q}_q(S_R)} + \|F_6\|_{H^1_q(\dot\Omega)}
+ \|F_3\|_{L_q(\Omega)} + \|F_4\|_{H^1_q(\Omega))}.
$$
\end{itemize}
\end{remark}
Let $\varphi(ik)$ be a function in $C^\infty(\BR)$ which equals one
for $k \in \BR_{k_0+2}$ and zero for $k \not \in \BR_{k_0+1}$, and 
let $\psi(ik)$ be a  function in $C^\infty(\BR)$ which equals one
for $k \in \BR_{k_0+4}$ and zero for $k \not \in \BR_{k_0+3}$.  
For $f \in \{\bF_\pm, G_\pm, \bG_\pm, D, \bH\}$, we set  
$$f_\psi = \CF^{-1}_\BT[\psi \CF_\BT[f]].$$
We consider the high frequency part of the equations \eqref{4.21}:
\begin{equation}\label{4.25}\begin{aligned}
\pd_t\bu_{\pm\psi} -\DV(\mu\bD(\bu_{\pm\psi})-\fp_{\pm\psi}\bI) = \bF_{\pm\psi}
&&\quad&\text{in $\Omega_\pm\times(0, 2\pi)$}, \\
\dv \bu_{\pm\psi} =G_{\pm\psi} = \dv\bG_{\pm\psi}&&\quad&\text{in $\Omega_\pm\times(0, 2\pi)$}, \\
\pd_t\rho_\psi + \CM\rho_\psi-(\CA\bu_{+\psi})\cdot\bn=D_\psi &&\quad&\text{on $S_R\times(0, 2\pi)$},\\
[[\mu\bD(\bu_\psi)-\fp_\psi\bI)\bn - (\CB_R\rho_\psi)\bn=\bH_\psi &&\quad&
\text{on $S_R\times(0, 2\pi)$}, \\
[[\bu_\psi]]=0 &&\quad&
\text{on $S_R\times(0, 2\pi)$}, \\
\bu_{-\psi} = 0 
&&\quad&
\text{on $S\times(0, 2\pi)$}. \\
\end{aligned}\end{equation}
By Theorem \ref{Weis}, Theorem \ref{ESK}, and the analogue of \eqref{4.4} resulting from \eqref{4.23},
we have immediately the
following theorem. 
\begin{thm}\label{thm:rs.22} Let $1 < p, q < \infty$.  Then, for any functions $\bF$, $G$, $\bG$, 
$D$, and $\bH$ with 
\begin{align*}
\bF & \in L_{p, {\rm per}}((0, 2\pi), L_q(\dot\Omega)^N), \quad 
D \in L_{p, {\rm per}}((0, 2\pi), W^{2-1/q}_q(S_R)), \quad \\
\bH &\in H^{1/2}_{p, {\rm per}}((0, 2\pi), L_q(\Omega)^N) \cap 
L_{p, {\rm per}}((0, 2\pi), H^1_q(\Omega)^N), \\
G & \in H^{1/2}_{p, {\rm per}}((0, 2\pi), L_q(\dot\Omega)) \cap 
L_{p, {\rm per}}((0, 2\pi), H^1_q(\dot\Omega)), \quad 
\bG \in H^1_{p, {\rm per}}((0, 2\pi), L_q(\dot\Omega)^N), 
\end{align*}
We let 
\begin{align*}
\bu_\psi & = \CF^{-1}_\BT[\varphi(ik)\CA(ik)\CF_k(\bF_\psi, D_\psi, \bH_\psi, G_\psi, \bG_\psi)](\cdot, t),
\\
\fp_\psi & = \CF^{-1}_\BT[\varphi(ik)\CP(ik)\CF_k(\bF_\psi, D_\psi, \bH_\psi, G_\psi, \bG_\psi)](\cdot, t),
\\
\rho_\psi & = \CF^{-1}_\BT[\varphi(ik)\CA(ik)\CF_k(\bF_\psi, D_\psi, \bH_\psi, G_\psi, \bG_\psi)](\cdot, t),
\end{align*}
where we have set 
\begin{align*}
\CF_k(\bF_\psi, D_\pi, \bH_\psi, G_\psi, \bG_\psi)
= &\psi(ik)(\CF_\BT[\bF](ik), \CF_\BT[D](ik), (ik)^{1/2}\CF_\BT[\bH](ik), 
\CF_\BT[\bH](ik), \\
&\quad (ik)^{1/2}\CF_\BT[G](ik), \CF_\BT[G](ik), ik\CF_\BT[\bG](ik)).
\end{align*}
Then, $\bu_\psi$, $\fp_\psi$ and $\rho_\psi$ are the unique solutions of  equations \eqref{4.25}, 
which possess the following estimate:
\begin{align*}
&\|\bu_\psi\|_{L_p((0, 2\pi), H^2_q(\dot\Omega))} + 
\|\pd_t\bu_\psi\|_{L_p((0, 2\pi), L_q(\dot\Omega))} + 
\|\nabla \fp_\psi \|_{L_p((0, 2\pi), L_q(\dot\Omega))} \\
 &\hskip0.5cm+ \|\rho_\psi\|_{L_p((0, 2\pi), W^{3-1/q}_q(S_R))} + 
\|\pd_t\rho_\psi\|_{H^1_p((0, 2\pi), W^{2-1/q}_q(S_R))}\\
&\hskip1cm\leq C\{\|\bF_\psi\|_{L_p((0, 2\pi), L_q(\dot\Omega))} 
+ \|D_\psi\|_{L_p((0, 2\pi), W^{2-1/q}_q(S_R))} 
+ \|\pd_t\bG_\psi\|_{L_p((0, 2\pi), L_q(\dot\Omega))} \\
&\hskip1.5cm+ \|\Lambda^{1/2}G_\psi\|_{L_p((0, 2\pi), L_q(\dot\Omega))} 
+ \|G_\psi\|_{L_p((0, 2\pi), H^1_q(\dot\Omega))}\\
&\hskip2cm +\|\Lambda^{1/2}\bH_\psi\|_{L_p((0, 2\pi), L_q(\Omega))}
+ \|\bH_\psi\|_{L_p((0, 2\pi), H^1_q(\Omega))}\}
\end{align*}
for some constant $C > 0$. Here, we have set
$$\Lambda^{1/2}(G_\psi, \bH_\psi) = \CF^{-1}_\BT[(ik)^{1/2}\psi(ik)(\CF_\BT[G](ik), \CF_\BT[\bH]
(ik))].$$
\end{thm}

We now consider the lower frequency part of solutions of equations \eqref{4.21}. 
Namely, we consider equations \eqref{4.22} for $k \in \BR$ with $1 \leq |k| < k_0+4$.
We shall show the following theorem.
\begin{thm}\label{thm:rs.23} Let $1 < q < \infty$ and $k \in \BZ$ 
with $|k| \leq k_0+3$.   Then, for any 
$\bff_\pm \in L_q(\Omega_\pm)^N$, $g_\pm \in H^1_q(\Omega_\pm)$, $d \in W^{2-1/q}_q(S_R)$, 
$\bh \in H^1_q(\Omega)^N$, and $\bg_\pm \in L_q(\Omega_\pm)^N$, problem \eqref{4.22}
admits unique solutions $\bv_\pm \in H^2_q(\Omega_\pm)^N$, 
$\fq_\pm \in H^1_q(\Omega_\pm)$ with $\int_\Omega\fq\,dx=0$, and $\eta \in W^{3-1/q}_q(S_R)$ possessing
the estimate:
\begin{equation}\label{4.6}\begin{aligned}
&\|\bv\|_{H^2_q(\dot\Omega)} + \|\nabla \fq\|_{L_q(\dot\Omega)}
+ \|\eta\|_{W^{3-1/q}_q(S_R)} \\
&\quad 
\leq C(\|\bff\|_{L_q(\dot\Omega)} + \|d\|_{W^{2-1/q}_q(S_R)}
+ \|g\|_{H^1_q(\dot\Omega)} + \|\bg\|_{L_q(\dot\Omega)} + \|\bh\|_{H^1_q(\Omega)})
\end{aligned}\end{equation}
for some constant $C > 0$. 
\end{thm}
\begin{proof} The strategy of the proof is the same as that in 
Theorem \ref{thm:rs.3}.  The only difference is the reduced problem.
First, we can reduce equations \eqref{4.22} to equations: 
\begin{equation}\label{4.23}\begin{aligned}
ik\bv -\DV(\mu\bD(\bv)-\fp\bI) = \bff &&\quad&\text{in $\dot\Omega$}, \\
\dv \bv =0&&\quad&\text{in $\dot\Omega$}, \\
ik\rho + \CM\rho-(\CA\bv_+)\cdot\bn=d&&\quad&\text{on $S_R$},\\
[[\mu\bD(\bv)-\fp\bI]]\bn - (\CB_R\rho)\bn=0&&\quad&
\text{on $S_R$}, \\
[[\bv]]=0&&\quad&
\text{on $S_R$}, \\
\bv_- = 0&&\quad&\text{on $S$}.
\end{aligned}\end{equation}
For any $\bv_\pm \in H^2_q(\Omega_\pm)^N$ and $\rho\in W^{3-1/q}_q(S_R)$, 
let $K=K(\bv, \rho) \in \dot H^1_q(\dot\Omega)$ be the unique solution of the weak Neumann problem:
\begin{equation}\label{4.210}
(\nabla K, \nabla\varphi)_{\dot\Omega} = (\DV(\mu\bD(\bv))-\nabla\dv\bv, \nabla\varphi)_{\dot\Omega}
\quad\text{for any $\varphi \in \dot H^1_{q'}(\Omega)$}
\end{equation}
subject to the transmission condition:
\begin{equation}\label{4.210*}
[[K]]=<[[\mu\bD(\bv)]]\bn, \bn> - \sigma(\CB_R\zeta)\bn - [[\dv\bv]]
\quad\text{on $S_R$},
\end{equation}
where $\mu$ is piecewise constant defined by $\mu|_{\Omega_\pm} = \mu_\pm$. 
Here and in the following, $\dot H^1_q(\Omega)$ is defined by
setting
$$\dot H^1_q(\Omega) 
= \Big\{\varphi \in H^1_q(\Omega) \mid \int_\Omega \varphi\, dx = 0\Big\}.
$$
The reduced problem corresponding to equations \eqref{4.23} is 
\begin{equation}\label{4.29}\begin{aligned}
ik\bv - \DV(\mu\bD(\bv) - K(\bv, \rho)\bI) = \bff&&\quad&\text{in $\dot\Omega$}, \\
ik\rho + \CM\rho-(\CA\bv_+)\cdot\bn= d&&\quad&\text{on $S_R$}, \\
[[\mu\bD(\bv) - K(\bv, \rho)\bI]]\bn - \sigma(\CB_R\rho)\bn= 0&&\quad
&\text{on $S_R$}, \\
[[\bv]]=0 &&\quad
&\text{on $S_R$}, \\
\bv_-=0 &&\quad
&\text{on $S$}. 
\end{aligned}
\end{equation}
Let $J_q(\dot\Omega)$ be the solenoidal space defined by setting
$$J_q(\dot\Omega) = \{\bu \in L_q(\dot\Omega) \mid 
(\bu, \nabla\varphi)_{\dot\Omega} = 0 \quad\text{for any
$\varphi \in \dot H^1_{q'}(\Omega)$}\}.
$$
For any $\bff \in J_q(\dot\Omega)$ and $d \in W^{2-1/q}_q(S_R)$, problems
\eqref{4.23} and \eqref{4.29} are equivalent.  In fact, 
if problem \eqref{4.23} admits unique solutions 
$\bv \in H^2_q(\dot\Omega)^N$, $\fp \in \dot H^1_q(\dot\Omega)$ and 
$\rho \in W^{3-1/q}_q(S_R)$, then using the divergence theorem of Gauss
and noting that $[[\varphi]]=0$ on $S_R$ gives that
for any $\varphi \in \dot H^1_{q'}(\Omega)$, 
\begin{align*}
0 = (\bff, \nabla\varphi)_{\dot\Omega} = ik(\bv, \nabla\varphi)_{\dot\Omega} 
-(\nabla\dv\bv, \nabla\varphi)_{\dot\Omega} 
+(\nabla(\fp - K(\bv, \rho)), \nabla\varphi)_{\dot\Omega}
 = (\nabla(\fp - K(\bv, \rho)), \nabla\varphi)_{\dot\Omega}
\end{align*}
because $\dv \bv=0$ on $\dot\Omega$.  Moreover, the transmission conditions
in \eqref{4.23} and \eqref{4.210*} gives that
$$[[\fp-K(\bv, \rho)]] = [[\dv\bv]]=0 \quad\text{on $S_R$}.
$$
Thus, the uniqueness of the weak Neumann problem in $\dot H^1_q(\dot\Omega)$ yields that
$\fp - K(\bv, \rho) = 0$ in $\Omega$. Thus, $\bv$ and $\rho$ satisfy
the equations \eqref{4.29}.  

Conversely, if $\bv\in H^2_q(\dot\Omega)^N$ and $\rho \in W^{3-1/q}_q(S_R)$
satisfy equations \eqref{4.29}, then the divergence theorem of Gauss gives that 
for any $\varphi \in \dot H^1_{q'}(\Omega)$ we have
\begin{align*}
0 = (\bff, \nabla\varphi)_{\dot\Omega} = ik(\bv, \nabla\varphi)_{\dot\Omega} 
-(\nabla\dv\bv, \nabla\varphi)_{\dot\Omega}
 = -\{ik(\dv\bv, \varphi)_{\dot\Omega} +(\nabla\dv\bv, \nabla\varphi)_{\dot\Omega}\}.
\end{align*}
Moreover, the transmission conditions in \eqref{4.29} and \eqref{4.210*} give that
$$[[\dv \bv]]= <[[\mu\bD(\bv)]]\bn, \bn> - \sigma(\CB_R\zeta) - [[K]]=0 
\quad\text{on $S_R$}.
$$
Thus, the uniqueness of this weak Neumann problem yields that
$\dv \bv=c$ in $\dot\Omega$ for some global constant $c$.
Now the divergence theorem of Gauss and the boundary conditions in \eqref{4.29} 
yield $c=0$, that is, $\dv\bv=0$, which shows that 
$\bv$, $\fp = K(\bv, \rho)$ and $\rho$ satisfy equations \eqref{4.23}. 

Employing the same argument as that in the proof of Theorem
\ref{thm:rs.3}, we see that to prove Theorem \ref{thm:rs.23}, it is
sufficient to prove the uniqueness of solutions to equations
\eqref{4.29} in the $L_2$ framework.  Thus, we choose 
 $\bv \in H^2_2(\dot\Omega)^N$ and $\rho \in 
W^{5/2}_2(S_R)$ be solutions of the homogeneous equations: 
\begin{equation}\label{4.230}\begin{aligned}
ik\bv - \DV(\mu\bD(\bv) - K(\bv, \rho)\bI) = 0&&\quad&\text{in $\dot\Omega$}, \\
ik\rho + \CM\rho-(\CA\bv_+)\cdot\bn= 0&&\quad&\text{on $S_R$}, \\
[[\mu\bD(\bv) - K(\bv, \rho)\bI]]\bn - \sigma(\CB_R\rho)\bn= 0&&\quad
&\text{on $S_R$}, \\
[[\bv]]=0 &&\quad
&\text{on $S_R$}, \\
\bv_-=0 &&\quad
&\text{on $S$}, 
\end{aligned}
\end{equation}
and we shall show that $\bv=0$ and $\rho=0$.  Notice that 
$\dv\bv= 0$ on $\dot\Omega$.  Moreover, by $[[\bv]]=0$, 
we have $\bv \in H^1_q(\Omega) \cap H^2_q(\dot\Omega)$. Integrating the second equation
in \eqref{4.230} over $S_R$ and using the divergence theorem of Gauss on
$\Omega_+ = B_R$ gives that
\begin{align*}
0 =ik(\rho, 1)_{S_R} +  \int_{S_R}\rho\,d\omega |S_R| - \int_{B_R}\dv(\bv_+-\frac{1}{|B_R|}\int_{B_R}\bv_+\,dy)\,dx
= (ik+|S_R|)\int_{S_R}\rho\,d\omega |S_R|
\end{align*}
because $\dv\bv_+ = 0$ on $B_R$, and so $(\rho, 1)_{S_R}=0$.  Moreover, multiplying the second equation
in \eqref{4.230} by $x_j$ and integrating over $S_R$, similar arguments lead to
\begin{align*}
0 &=ik(\rho, x_j)_{S_R} +  \int_{S_R}\rho x_j\,d\omega(x_j, x_j)_{S_R}
- \int_{B_R}\dv\{x_j(\bv_+(x)-\frac{1}{|B_R|}\int_{B_R}\bv_+\,dy)\}\,dx \\
& =ik(\rho, x_j)_{S_R} +  \int_{S_R}\rho x_j\,d\omega(x_j, x_j)_{S_R}
-\int_{B_R}(v_{+j}(x) - \frac{1}{|B_R|}\int_{B_R}v_{+j}\,dy)\,dx\\
& = ik(\rho, x_j)_{S_R} +  \int_{S_R}\rho x_j\,d\omega(x_j, x_j)_{S_R},
\end{align*}
because $(1, x_j)_{S_R}=0$, and $(x_k, x_j)_{S_R} = 0$ for $j\not=k$. Since 
$(x_j, x_j)_{S_R} = (R^2/N)|S_R| > 0$, we have $(\rho, x_j) = 0$.
Summing up, we have proved 
\begin{equation}\label{null:21} 
(\rho, 1)_{S_R} = 0, \quad(\rho, x_j)_{S_R} = 0 \quad(j=1\ldots, N).
\end{equation}
In particular, $\CM\rho=0$.

We now prove that $\bv = 0$.  
Multiplying the first equation of \eqref{4.230} with $\bv$ and integrating the 
resultant formula over $\dot\Omega$ and using the divergence theorem of Gauss gives
that  
\begin{align*}
0 = ik\|\bv\|_{L_2(\dot\Omega)}^2 - \sigma(\CB_R\rho, \bn\cdot\bv)_{S_R} + \frac{\mu}{2}
\|\bD(\bv)\|_{L_2(\dot\Omega)}^2,
\end{align*}
because $\dv\bv = 0$ in $\dot\Omega$.  By the second equation of \eqref{4.230}
with $\CM\rho=0$, we have
$$\sigma(\CB_R\rho, \bn\cdot\bv)_{S_R}
= \sigma (\CB_R\rho, ik\rho)_{S_R}
+ \sum_{k=1}^N\frac{1}{|B_R|}\int_{B_R}w_j\,dt(\CB_R\rho, R^{-1}x_j)_{S_R}
$$
where we have used $\bn =R^{-1}x= R^{-1}(x_1,\ldots, x_N)$ for $x \in S_R$.
This also yields 
$$(\CB_R\rho, x_j)_{S_R} = (\rho, (\Delta_{S_R}+\frac{N-1}{R^2})x_j)_{S_R} =0.
$$
Moreover, since $\rho$ satisfies \eqref{null:21}, we know that
$$-(\CB_R\rho, \rho)_{S_R} \geq c\|\rho\|_{L_2(S_R)}^2
$$
for some positive constant $c$, and therefore we have $\bD(\bv)=0$.
Since $\bv \in H^1_q(\Omega)$ and $\bv=0$
on $S_-$, we have $\bv=0$. 

Finally, the first equation of \eqref{4.230} yields that
$\nabla K(\bv, \rho) = 0$, which shows that 
$K(\bv, \rho)$ is constant in $\dot\Omega$.  Thus, 
$[[K(\bv, \rho)]]$ is constant. Integrating the third equation of  \eqref{4.230} yields that
$$
[[K(\bv, \rho)]]\int_{S_R}\,d\omega = \sigma(\Delta_{S_R}\rho,1)_{S_R}
+ \frac{N-1}{R^2}(\rho, 1)_{S_R} = 0
$$
where we have used \eqref{null:21}.  In particular, $K(\bv, \rho)$ is a constant globally
in $\Omega$.  Finally, we have $\CB_R\rho = 0$ on $S_R$, which, combined with
\eqref{null:21} leads to $\rho=0$.  This completes the proof of uniqueness for equations
\eqref{4.29} in the $L_2$ framework.  Therefore, we have proved 
Theorem \ref{thm:rs.23}.
\end{proof}



{\bf Proof of Theorem \ref{thm:linear.main2}.} 
Employing the same argument as in the proof of Theorem \ref{thm:linear.main1}
and using Theorem \ref{thm:rs.22} and Theorem \ref{thm:rs.23}, we can prove 
Theorem \ref{thm:linear.main2}. We may omit the detailed proof. 

\section{Proofs of main results} \label{sec:5}

In this section, we shall prove Theorem \ref{main:thm1}.  
The proof of Theorem \ref{main:thm2} is parallel to that of Theorem \ref{main:thm1},
and so we may omit it.  We prove Theorem \ref{main:thm1}
with the help of the usual Banach fixed-point argument, and we define
an underlying space $\CI_\epsilon$ with some small constant $\epsilon > 0$
determined later by setting
\begin{align}
\CI_\epsilon = \{&(\bv, h) \mid\quad 
\bv \in L_{p, {\rm per}}((0, 2\pi), H^2_q(B_R)^N) \cap 
H^1_{p, {\rm per}}((0, 2\pi), L_q(B_R)^N), \nonumber \\
& h \in L_{p, {\rm per}}((0, 2\pi), W^{3-1/q}_q(S_R)) \cap 
H^1_{p, {\rm per}}((0, 2\pi), W^{2-1/q}_q(S_R))
\cap H^1_{\infty, {\rm per}}((0, 2\pi), W^{1-1/q}_q(S_R)), \nonumber \\
&\sup_{t \in (0, 2\pi)}\|H_h(\cdot, t)\|_{H^1_\infty(B_R)} \leq \delta, \quad 
 E(\bv, h ) \leq \epsilon \label{space:1}\},
\end{align}
where we have set
\begin{align*}
E(\bv, h) &= \|\bv\|_{L_p((0, 2\pi), H^2_q(B_R))} + \|\bv\|_{H^1_p((0, 2\pi), L^2_q(B_R))} 
\\
&\,+ \|h\|_{L_p((0, 2\pi), W^{3-1/q}_q(B_R))}  + 
\|h\|_{H^1_p((0, 2\pi), W^{2-1/q}_q(B_R))}
+\|\pd_th\|_{L_\infty((0, 2\pi), W^{1-1/q}_q(S_R))}. 
\end{align*} 
In view of \eqref{bary:3}, we define $\xi(t)$ by setting 
\begin{equation}\label{xi:1}
\xi(t) = \int^t_0\xi'(s)\,ds + c = \frac{1}{|B_R|} \int^t_0\int_{B_R}\bv(x, s)(1+ J_0(x,s))\,dxds + c
\end{equation}
where $c$ is a constant for which 
\begin{equation}\label{xi:2}
\int^{2\pi}_0 \xi(s)\,ds = 0, \quad\text{that is}, \quad 
c = -\frac{1}{2\pi|B_R|}\int^{2\pi}_0\Bigl(\int^t_0\int_{B_R}(\bv(x, s)(1+ J_0(x,s))\,dxds\Bigr)\,dt.
\end{equation}
We choose $\delta > 0$ so small that the map $x =\Phi(y, t)= y + \Psi(y, t)$ with
$\Psi(y, t)=\Psi_h(y, t) = R^{-1}H_h(y, t)y + \xi(t)$ is one to one.  In particular, we may assume
that $\delta > 0$ and the inverse map: $y = \Xi(y, t)$ is well-defined and has the same regularity
property as $\Phi(y, t)$.  In particular, we may assume that 
\begin{equation}\label{non:0}
\Xi(D) \subset B_R.
\end{equation}

Since $\epsilon > 0$ will be chosen small eventually, we may assume that $0 < \epsilon < 1$,
 and so for example, we estimate $\epsilon^2 < \epsilon$ if necessary. Let $(\bv, h)
\in \CI_\epsilon$ and let $\bu$ and $\rho$ be solutions of linearized equations: 
\begin{equation}\label{5.1}\left\{\begin{aligned}
&\pd_t\bu + \CL\bu_S- \DV(\mu(\bD(\bu) - \fp\bI) = \bG + \bF(\bv, h)
&\quad&\text{in $B_R\times(0, 2\pi)$}, \\
&\dv \bu = g(\bv, h) = \dv \bg(\bv, h)
&\quad&\text{in $B_R\times(0, 2\pi)$}, \\
&\pd_t\rho + \CM\rho 
-\CA\bu\cdot\bn
= \tilde d(\bv, h) 
&\quad&\text{on $S_R\times(0, 2\pi)$}, \\
&(\mu\bD(\bu)-\fp\bI)\bn - (\CB_R\rho) \bn
= \bh(\bv, h)
&\quad&\text{on $S_R\times(0, 2\pi)$}.
\end{aligned}\right.\end{equation}
In view of Theorem \ref{thm:linear.main1}, we shall show that
\begin{equation}\label{main:ne:1}\begin{aligned}
&\|\bF(\bv, h)\|_{L_p((0, 2\pi), L_q(B_R))} +
\|\tilde d(\bv, h)\|_{L_p((0, 2\pi), W^{2-1/q}_q(S_R))} 
+ \|(g(\bv, h), \bh(\bv, h)\|_{H^{1/2}_p((0, 2\pi), L_q(B_R))} \\
&\quad + \|(g(\bv, h), \bh(\bv, h)\|_{L_p((0, 2\pi), H^1_q(B_R))} 
+ \|\pd_t\bg(\bv, h)\|_{L_p((0, 2\pi), L_q(B_R))}
\leq C\epsilon^2,
\end{aligned}\end{equation}
for some constant $C > 0$ independent of $\epsilon > 0$.  In the following,
$C$ denotes generic constants independent of $\epsilon > 0$, the value of
which may change from line to line.  Before starting with the estimates of the nonlinear
terms, we summarize some inequalities which are useful for our estimations.
The following inequalities follow from  Sobolev's inequality and the estimate of the boundary
trace: 
\begin{equation}\label{sob:1}\begin{aligned}
\|f\|_{L_\infty(B_R)} & \leq C\|f\|_{H^1_q(B_R)}, \\
\|fg\|_{H^1_q(B_R)} & \leq C\|f\|_{H^1_q(B_R)}\|g\|_{H^1_q(B_R)},  \\
\|fg\|_{H^2_q(B_R)} & \leq C(\|f\|_{H^2_q(B_R)}\|g\|_{H^1_q(B_R)} 
+ \|f\|_{H^1_q(B_R)}\|g\|_{H^2_q(B_R)}), \\
\|fg\|_{W^{1-1/q}_q(S_R)} & \leq C\|f\|_{W^{1-1/q}_q(S_R)}\|g\|_{W^{1-1/q}_q(S_R)},  \\
\|fg\|_{W^{2-1/q}_q(S_R)} & \leq C(\|f\|_{W^{2-1/q}_q(S_R)}\|g\|_{W^{1-1/q}_q(S_R)} 
+ \|f\|_{W^{1-1/q}_q(S_R)}\|g\|_{W^{2-1/q}_q(S_R)})
\end{aligned}\end{equation}
for $N < q < \infty$ with some constant $C$.  The following inequalities
follow from real interpolation theorem and the periodicity of functions, 
which will be used to estimate
the $L_\infty$ norm with respect to the time variable of lower order regularity
terms with respect to the space variable $x$: 
\begin{equation}\label{real:1}\begin{aligned}
\|\bv\|_{L_\infty((0, 2\pi), B^{2(1-1/q)}_{q,p}(B_R))}
& \leq C(\|\bv\|_{L_p((0, 2\pi), H^2_q(B_R))} 
+ \|\pd_t\bv\|_{L_p((0, 2\pi), L_q(B_R))}),  \\
\|h\|_{L_\infty((0, 2\pi), B^{3-1/p-1/q}_{q,p}(S_R))}
& \leq C(\|h\|_{L_p((0, 2\pi), W^{3-1/q}_q(S_R))} 
+ \|\pd_th\|_{L_p((0, 2\pi), W^{2-1/q}_q(S_R))}).
\end{aligned}\end{equation}
In fact, to obtain \eqref{real:1} we use the following well-known result: 
Let $X$ and $Y$ be two Banach spaces such that $Y$ is continuously embedded
into $X$, and then $C([0, \infty), (X, Y)_{1-1/p,p})$
is continuously embedded into $H^1_p((0, \infty), X) \cap
L_p((0, \infty), Y)$ and 
$$\|f\|_{L_\infty((0, \infty), (X, Y)_{1-1/p,p})}
\leq \|f\|_{L_p((0, \infty), Y)} + \|f\|_{H^1_p((0, \infty), X)}.
$$
For its proof, we refer to \cite{Lu, Tri}.  

We start with the estimate of $\bF(\bv, h)$. 
From \eqref{form:f}, we have
\begin{align*}
&\|\bF_1(\bv, h)\|_{L_q(B_R)}  \leq C\{\|\bv\|_{L_\infty(B_R)}\|\nabla\bv\|_{L_q(B_R)}
+ \|\pd_t\Psi_h\|_{L_\infty(B_R)}\|\nabla\bv\|_{L_q(B_R)} \\
&\quad + \|\nabla\Psi_h\|_{L_\infty(B_R)}\|\pd_t\bv\|_{L_q(B_R)}
+ \|\nabla\Psi_h\|_{L_\infty(B_R)}\|\nabla^2\bv\|_{L_q(B_R)}
+ \|\nabla^2\Psi_h\|_{L_q(B_R)}\|\nabla\bv\|_{L_\infty(B_R)}).
\end{align*}
By \eqref{sob:1} and  \eqref{ext:1}, we have
\begin{align*}
\|\bF_1(\bv, h)\|_{L_p((0, 2\pi), L_q(B_R))}&
 \leq C\{\|\bv\|_{L_\infty((0, 2\pi), H^1_q(B_R))}
\|\bv\|_{L_p((0, 2\pi), H^1_q(B_R))} \\
&+\|\pd_th\|_{L_p((0, 2\pi), W^{1-1/q}_q(S_R))}\|\bv\|_{L_\infty((0, 2\pi), H^1_q(B_R))}\\
&+ \|h\|_{L_\infty((0, 2\pi), W^{2-1/q}_q(S_R))}(\|\pd_t\bv\|_{L_q((0, 2\pi), L_q(B_R))}
+ \|\bv\|_{L_p((0, 2\pi), H^2_q(B_R))}),
\end{align*}
which, combined with \eqref{real:1} and \eqref{space:1}, leads to 
\begin{equation}\label{non:1}
\|\bF_1(\bv, h)\|_{L_p((0, 2\pi), L_q(B_R))} \leq C\epsilon^2,
\end{equation}
because $1 < 2(1-1/p)$ and $2-1/q < 3-1/p-1/q$.  From \eqref{form:f*},
it follows that
\begin{align*}
&\|\bF_2(\bv, h)(\cdot, t)\|_{L_q(B_R)} \\
&\leq C\int^{2\pi}_0\|\bv(\cdot,t)\|_{L_q(B_R)}(\|J_0(\cdot,t)\|_{L_\infty(B_R)}
+\|\Psi(\cdot, t)\|_{L_\infty(B_R)}(1+\|J_0(\cdot,t)\|_{L_\infty(B_R)}))\,dt \\
&\quad+ \int^{2\pi}_0\|\bv(\cdot, t)\|_{L_q(B_R)}(1+\|\Psi(\cdot, t)\|_{L_\infty(B_R)})
(1+\|J_0(\cdot,t)\|_{L_\infty(B_R)})\,dt\|\Psi(\cdot, t)\|_{L_q(B_R)} \\
&\quad + \|\nabla\Psi(\cdot, t)\|_{L_q(B_R)}\int^{2\pi}_0
\|\bv(\cdot, t)\|_{L_q(B_R)}(1 + \|\Psi(\cdot, t)\|_{L_\infty(B_R)})
(1 + \|J_0(\cdot,t)\|_{L_\infty(B_R)})\,dydt\\
&\qquad\qquad\times(1+\|\Psi(\cdot, t)\|_{L_\infty(B_R)}).
\end{align*}

To estimate $\bF_2(\bv, h)$, we recall 
$$
J_0(y,t) = \det\Bigl(\delta_{ij} + R^{-1}\frac{\pd}{\pd y_j}H_h(y, t)y_i\Bigr) - 1
$$
and that 
$\Psi(y, t) = R^{-1}H_h(y, t)y + \xi(t)$, where $\xi(t)$ is given by 
\begin{equation}\label{xi:3} \begin{aligned}
\xi(t)& = \int^t_0\frac{1}{|B_R|}\int_{B_R}(\bv(y, s)(1 + J_0(y,s))\,dyds + c, \\
 c&=-\int^{2\pi}_0\int^t_0\frac{1}{|B_R|}\int_{B_R}(\bv(y, s)(1 + J_0(y,s))\,dydsdt. 
\end{aligned}\end{equation}
By \eqref{sob:1} and \eqref{ext:1}
we obtain
\begin{equation}\label{5.2}\begin{aligned}
\|H_h(\cdot, t)\|_{L_\infty(B_R)} \leq C\|h(\cdot, t)\|_{W^{1-1/q}_q(S_R)}
\leq C\epsilon, \\
\|\nabla H_h(\cdot, t)\|_{L_\infty(B_R)} \leq C\|h(\cdot, t)\|_{W^{2-1/q}_q(S_R)}
\leq C\epsilon, 
\end{aligned}\end{equation} 
By \eqref{sob:1}, \eqref{ext:1}, \eqref{real:1}, the fact that $2-1/q < 3-1/p-1/q$, and 
\eqref{space:1}, we have
\begin{equation}\label{5.3}\begin{aligned}
\|J_0(\cdot,t)\|_{L_\infty(B_R)}& \leq C\|\nabla H_h(\cdot, t)\|_{L_\infty(B_R)}(1 + 
\|\nabla H_h(\cdot, t)\|_{L_\infty(B_R)})^{N-1} \\
&\leq C\|h(\cdot, t)\|_{W^{2-1/q}_q(S_R)}(1+
\|h(\cdot, t)\|_{W^{2-1/q}_q(S_R)})^{N-1} \\
&\leq C\epsilon. 
\end{aligned}\end{equation}
From \eqref{xi:3} and \eqref{space:1}, it follows that 
\begin{equation}\label{5.4}
|\xi(t)| \leq C\|\bv\|_{L_p((0, 2\pi), L_q(B_R))} \leq C\epsilon.
\end{equation}
In particular, by \eqref{5.2} and  \eqref{5.4}, we have
\begin{equation}\label{5.5}
\|\Psi(\cdot, t)\|_{L_\infty(B_R)} \leq C\epsilon, 
\quad \|\nabla\Psi(\cdot, t)\|_{L_\infty(B_R)} \leq C\epsilon.
\end{equation}
Combining \eqref{space:1} and \eqref{5.5}  gives that
$$
\|\bF_2(\bv, h)\|_{L_p((0, 2\pi), L_q(B_R))}
 \leq C\epsilon\|\bv\|_{L_p((0, 2\pi), L_q(B_R))}
\leq C\epsilon^2,
$$
which, combined with \eqref{non:1}, leads to 
\begin{equation}\label{non:2}
\|\bF(\bv, h)\|_{L_p((0, 2\pi), L_q(B_R))}
\leq C\epsilon^2.
\end{equation}
By \eqref{non:0} and \eqref{5.5}, we have 
\begin{equation}\label{non:3}
\|\bG\|_{L_p((0, 2\pi), L_q(B_R))} \leq C\|\bff\|_{L_p((0, 2\pi), L_q(D))}.
\end{equation}

We next estimate $\tilde d(\bv, h)$.  By \eqref{normal:3.1} and \eqref{space:1}, 
\begin{align*}
\|\bn_t-\bn\|_{W^{1-1/q}_q(S_R)} & \leq C\|H_h(\cdot, t)\|_{H^2_q(B_R)}
\leq C\epsilon, \\
\|\bn_t-\bn\|_{W^{2-1/q}_q(S_R)} &\leq C(\|H_h(\cdot, t)\|_{H^3_q(B_R)}
+ \|H_h(\cdot, t)\|_{H^2_q(B_R)}\|H_h(\cdot, t)\|_{H^2_\infty(B_R)}).
\end{align*}
Since we assume that $2/p + N/q < 1$, we can choose $\kappa > 0$ so small that
 $2+N/q + \kappa -1/q < 3-1/p-1/q$ and $1 + N/q + \kappa < 2(1-1/p)$,
 and then by Sobolev's inequality and \eqref{real:1} we have 
\begin{equation}\label{real:2}\begin{aligned}
&\sup_{t \in (0, 2\pi)}\|\bv(\cdot, t)\|_{H^{1}_\infty(B_R)}
\leq C\sup_{t \in (0, 2\pi)}\|\bv(\cdot, t)\|_{B^{2(1-1/p)}_{q,p}(B_R)}
\leq C\epsilon; \\
&\sup_{t \in (0, 2\pi)}\|H_h(\cdot, t)\|_{H^2_\infty(B_R)} 
 \leq C\sup_{t \in (0, 2\pi)}\|h(\cdot, t)\|_{B^{3-1/p-1/q}_{q,p}(S_R)}
\leq C\epsilon, 
\end{aligned}\end{equation}
where we have used \eqref{ext:1} in the last inequality. Then, in particular, using again \eqref{ext:1},
we have
$$
\|\bn_t-\bn\|_{W^{2-1/q}_q(S_R)} \leq C\|H_h(\cdot, t)\|_{H^3_q(B_R)} 
\leq C\|h(\cdot, t)\|_{W^{3-1/q}_q(S_R)}.
$$
Thus, applying \eqref{5.3} to the formula in \eqref{kin:1*}
and using \eqref{space:1} and \eqref{sob:1} gives that
\begin{equation}\label{nd:1}\begin{aligned}
\|d(\bv, h)\|_{L_p((0, 2\pi), W^{2-1/q}_q(S_R))}
&\leq C\epsilon(\|\bv\|_{L_\infty((0, 2\pi), H^1_q(B_R))}
+\|\bv\|_{L_p((0, 2\pi), H^2_q(B_R))}\\
&\qquad+ \|\pd_th\|_{L_\infty((0, 2\pi), W^{1-1/q}_q(S_R))}
+\|\partial_t h\|_{L_p((0, 2\pi), W^{2-1/q}_q(S_R))})\\
& \leq C\epsilon^2.
\end{aligned}\end{equation}
On the other hand, by \eqref{5.2}, 
$$\|h(\cdot, t)\|_{L_\infty(S_R)} \leq C \|H_h(\cdot, t)\|_{L_\infty(B_R)}
\leq C\epsilon,
$$
and so
$$\Bigl|\int_{S_R}h^k\,d\omega\Bigr| \leq C\epsilon^2, \quad 
\Bigl|\int_{S_R}h^k\omega\,d\omega\Bigr| \leq C\epsilon^2
\quad\text{for $k \geq 2$}, 
$$
which, combined with \eqref{nd:1}, leads to 
\begin{equation}\label{non:4}
\|\tilde d(\bv, h)\|_{L_p((0, 2\pi), W^{2-1/q}_q(S_R))}
\leq C\epsilon^2.
\end{equation}

We next consider $\bg(\bv, h)$ given in \eqref{form:g}, where $\rho$ is replaced by
$h$.  We may write
$$\bg(\bv, h) = \bV_\bg(\bk)(H_h, \nabla H_h)\otimes\bv.
$$
where $\bk$ denotes variables corresponding to $(H_h, \nabla H_h)$ and $\bV_\bg$
is a $C^\infty$ function defined on $|\bk| < \delta$.  We write
$$\pd_t\bg(\bv, h) = \bV_\bg'(\bk)\pd_t(H_h, \nabla H_h)\otimes(H_h, \nabla H_h)\otimes\bv
+ \bV_\bg(\bk)\pd_t(H_h, \nabla H_h)\otimes\bv +
\bV_\bg(\bk)(H_h, \nabla H_h)\otimes\pd_t\bv,
$$
and so, by \eqref{5.2}, \eqref{ext:1}, we have 
\begin{equation}\label{non:5}\begin{aligned}
\|\pd_t\bg(\bv, h)\|_{L_p((0, 2\pi), L_q(B_R))}
& \leq C(\|\bv\|_{L_\infty((0, 2\pi), H^1_q(B_R))}+\|h\|_{L_\infty((0, 2\pi), W^{2-1/q}_q(S_R))}) \\
&\hskip1cm \times (\|h\|_{H^1_p((0, 2\pi), W^{2-1/q}_q(S_R))}
+ \|\pd_t\bv\|_{L_p((0, 2\pi), L_q(B_R))}) \\
&\leq C\epsilon^2.
\end{aligned}\end{equation}

We next estimate $g(\bv, h)$ and $\bh(\bv, h)=(\bh'(\bv,h),h_N(\bv,h))$ given in \eqref{form:g}, \eqref{58} and \eqref{non:g.5},  where $\rho$ is replaced by $h$.
We may write
$$g(\bv, h) = V_g(\bk)(H_h, \nabla H_h)\otimes\nabla\bv,
$$
where $\bk$ are variables corresponding to $(H_h, \nabla H_h)$ and 
$V_g(\bk)$ is some matrix of $C^\infty$ functions
defined on $|\bk| < \delta$.  To estimate $g$, we
use the following two lemmas. 
\begin{lem}\label{lem:half1} Let $1 <p < \infty$ and $N < q < \infty$. 
Let 
\begin{align*}
f & \in H^1_{\infty, {\rm per}}((0, 2\pi), L_q(B_R)) \cap L_{\infty, {\rm per}}((0, 2\pi), H^1_q(B_R)),
\\
g & \in H^{1/2}_{p, {\rm per}}((0, 2\pi), L_q(B_R)) \cap L_{p, {\rm per}}((0, 2\pi), H^1_q(B_R)).
\end{align*}
Then we have
\begin{equation}\label{complex:0}\begin{aligned}
&\|fg\|_{H^{1/2}_p((0, 2\pi), L_q(B_R))} + \|fg\|_{L_p((0, 2\pi), H^1_q(B_R))} \\
&\quad \leq C(\|f\|_{H^1_\infty((0, 2\pi), L_q(B_R))} + \|f\|_{L_\infty((0, 2\pi), H^1_q(B_R))})^{1/2}
\|f\|_{L_\infty((0, 2\pi), H^1_q(B_R))}^{1/2}\\
&\qquad\times(\|g\|_{H^{1/2}_p((0, 2\pi), L_q(B_R))} + \|g\|_{L_p((0, 2\pi), H^1_q(B_R))})
\end{aligned}\end{equation}
for some constant $C > 0$.
\end{lem}
\begin{proof}
By \eqref{sob:1}, we have
\begin{equation}\label{complex:2}
\|fg\|_{L_p((0, 2\pi), H^1_q(B_R))}
\leq \|f\|_{L_\infty((0, 2\pi), H^1_q(B_R))}\|g\|_{L_p((0, 2\pi), H^1_q(B_R))}.
\end{equation}
To estimate the $H^{1/2}$ norm, we use the complex interpolation relation:
\begin{equation}\label{complex:1}\begin{aligned}
&H^{1/2}_{p, {\rm per}}((0, 2\pi), L_q(B_R)) \cap L_{p, {\rm per}}((0, 2\pi), H^{1/2}_q(B_R))
\\
&\quad =
\big(L_{p, {\rm per}}((0, 2\pi), L_q(B_R)), 
H^1_{p, {\rm per}}((0, 2\pi), L_q(B_R)) \cap L_{p, {\rm per}}((0, 2\pi), H^1_q(B_R))\big)_{1/2}
\end{aligned}\end{equation}
where $(\cdot, \cdot)_{1/2}$ denotes a complex interpolation of order $1/2$. 
By \eqref{sob:1}, we have 
\begin{align*}
\|fg\|_{H^1_p((0, 2\pi), L_q(B_R))} &\leq C(\|\pd_tf\|_{L_\infty((0, 2\pi), L_q(B_R))}
+ \|f\|_{L_\infty((0, 2\pi), H^1_q(B_R))})\\
&\qquad\times
(\|g\|_{L_p((0, 2\pi), H^1_q(B_R))} + \|\partial_t g\|_{L_p((0, 2\pi), L_q(B_R))}), \\
\|fg\|_{L_p((0, 2\pi), L_q(B_R))}
&\leq C\|f\|_{L_\infty((0, 2\pi), H^1_q(B_R))}\|g\|_{L_p((0, 2\pi), L_q(B_R))}.&
\end{align*}
Thus, by \eqref{complex:1}, we have
\begin{equation}\label{complex:3}\begin{aligned}
\|fg\|_{H^{1/2}_p((0, 2\pi), L_q(B_R))} 
& \leq C
(\|f\|_{H^1_\infty((0, 2\pi), L_q(B_R))} + \|f\|_{L_\infty((0, 2\pi), H^1_q(B_R))})^{1/2}
 \|f\|_{L_\infty((0, 2\pi), H^1_q(B_R))}^{1/2}\\
&\qquad \times(\|g\|_{H^{1/2}_p((0, 2\pi), L_q(B_R))} + \|g\|_{L_p((0, 2\pi), H^{1/2}_q(B_R))})
\end{aligned}\end{equation}
Since $\|g\|_{L_p((0, 2\pi), H^{1/2}_q(B_R))} \leq C\|g\|_{L_p((0, 2\pi), H^1_q(B_R))}$, 
combining \eqref{complex:2} and \eqref{complex:3} leads to \eqref{complex:0},
which completes the proof of Lemma \ref{lem:half1}.
\end{proof}
\begin{lem}\label{lem:half2} Let $1 < p, q < \infty$.  Then, there exists a constant $C$ such that
for any $u$ with
$$u \in H^1_{p, {\rm per}}((0, 2\pi), L_q(B_R)) \cap 
L_{p, {\rm per}}((0, 2\pi), H^2_q(B_R)),$$
we have
\begin{equation}\label{complex:4}
\|u\|_{H^{1/2}_p((0, 2\pi), H^1_q(B_R))}
\leq C(\|u\|_{H^{1}_p((0, 2\pi), L_q(B_R))} + 
\|u\|_{L_p((0, 2\pi), H^2_q(B_R))})
\end{equation}
for some constant $C > 0$.
\end{lem}
\begin{proof}
As was proved in the proof of Proposition 1 in Shibata \cite{S4},
there exist two operators $\Phi_1$ and $\Phi_2$ with
$$\Phi_1 \in C^1(\BR\setminus\{0\}, \CL(L_q(B_R), L_q(B_R)^N)), \quad
\Phi_2 \in C^1(\BR\setminus\{0\}, \CL(H^2_q(B_R), L_q(B_R)^N)
$$
such that for any $g \in H^2_q(B_R)$, we have
$$(1+\lambda^2)^{1/4}\nabla g =\Phi_1(\lambda)(1+\lambda^2)^{1/2}g + 
\Phi_2(\lambda)g,$$
and 
\begin{align*}
&\CR_{\CL(L_q(B_R), L_q(B_R)^N)}
(\{(\lambda\pd_\lambda)^{\ell}\Phi_1(\lambda) \mid \lambda \in \BR\setminus\{0\}\})
\leq r_b, \\
&\CR_{\CL(H^2_q(B_R), L_q(B_R)^N)}
(\{(\lambda\pd_\lambda)^{\ell}\Phi_1(\lambda) \mid \lambda \in \BR\setminus\{0\}\})
\leq r_b,
\end{align*}
for $\ell=0,1$ with some constant $r_b$. Thus, by Weis' operator-valued
Fourier multiplier theorem, Theorem \ref{Weis},  and transference theorem, Theorem \ref{ESK}, 
we have \eqref{complex:4}, which completes the proof of Lemma \ref{lem:half2}. 
\end{proof}
By \eqref{space:1}, \eqref{ext:1}, \eqref{sob:1} and \eqref{real:2}, we have
\begin{align*}
&\|\pd_tV_g(\bk)(H_h, \nabla H_h)\|_{L_\infty((0, 2\pi), L_q(B_R))}
\leq C\|h\|_{H^1_p((0, 2\pi), W^{1-1/q}_q(B_R))} \leq C\epsilon, \\
&\|V_g(\bk)(H_h, \nabla H_h)\|_{L_\infty((0, 2\pi), H^1_q(B_R))}
\leq C\|H_h\|_{L_\infty((0, 2\pi), H^2_q(B_R))}
\leq C\epsilon.
\end{align*}
Thus, by Lemma \ref{lem:half1}, Lemma \ref{lem:half2},
and \eqref{space:1}, we have
\begin{equation}\label{non:6}\begin{aligned}
&\|g(\bv, h)\|_{H^{1/2}_p((0, 2\pi), L_q(B_R))} 
+ \|g(\bv, h)\|_{L_p((0, 2\pi), H^1_q(B_R))}
\\
&\quad\leq
C\epsilon(\|\nabla\bv\|_{H^{1/2}_p((0, 2\pi), L_q(B_R))}
+ \|\nabla\bv\|_{L_p((0, 2\pi), H^1_q(B_R))})\\
&\quad \leq C\epsilon(\|\bv\|_{L_p((0, 2\pi), H^2_q(B_R))}
+ \|\pd_t\bv\|_{L_p((0, 2\pi), L_q(B_R))}) \\
&\quad \leq C\epsilon^2.
\end{aligned}\end{equation}

Analogously, recalling the definition of $\bh(\bv, h)=(\bh'(\bv,h),h_N(\bv,h))$ given in
\eqref{58} and \eqref{non:g.5}, where $\rho$ is replaced by $h$, and using
Lemma \ref{lem:half1}, we have
\begin{align*}
&\|\bh(\bv, h)\|_{H^{1/2}_p((0, 2\pi), L_q(B_R))}
+ \|\bh(\bv, h)\|_{L_p((0, 2\pi), H^1_q(B_R))} \\
&\quad 
\leq C\epsilon(\|\nabla\bv\|_{H^{1/2}_p((0, 2\pi), L_q(B_R))}
+ \|\nabla\bv\|_{L_p((0, 2\pi), H^1_q(B_R))} \\
&\qquad+
\|\overline\nabla^2H_h\|_{H^{1/2}_p((0, 2\pi), L_q(B_R))}
+ \|\overline\nabla^2H_h\|_{L_p((0, 2\pi), H^1_q(B_R))}).
\end{align*}
Since $H^{1/2}_p((0, 2\pi), L_q(B_R)) \supset H^1_p((0, 2\pi), L_q(B_R))$, 
we have
$$\|\overline{\nabla}^2H_h\|_{H^{1/2}_p((0, 2\pi), L_q(B_R))}
\leq C\|\overline{\nabla}^2H_h\|_{H^1_p((0, 2\pi), L_q(B_R))},
$$
and so using Lemma \ref{lem:half2}, \eqref{ext:1}, and \eqref{space:1},  we have
\begin{equation}\label{non:7}\begin{aligned}
&\|\bh(\bv, h)\|_{H^{1/2}_p((0, 2\pi), L_q(B_R))}
+ \|\bh(\bv, h)\|_{L_p((0, 2\pi), H^1_q(B_R))} \\
&\quad 
\leq C\epsilon(\|\bv\|_{H^1_p((0, 2\pi), L_q(B_R))}
+ \|\bv\|_{L_p((0, 2\pi), H^2_q(B_R))} \\
&\qquad+
\|\pd_t H_h\|_{L_p((0, 2\pi), H^2_q(B_R))}
+ \|H_h\|_{L_p((0, 2\pi), H^3_q(B_R))})
\\
&\quad
\leq C\epsilon^2.
\end{aligned}\end{equation}
Combining \eqref{non:2}, \eqref{non:4}, \eqref{non:5}, 
\eqref{non:6}, and \eqref{non:7} gives \eqref{main:ne:1}. 
Applying Theorem \ref{thm:linear.main1} to equations \eqref{5.1}
and using \eqref{main:ne:1} and \eqref{non:3} gives that
\begin{equation}\label{non:8}
\begin{aligned}
&\|\bu\|_{L_p((0, 2\pi), H^2_q(B_R))} + \|\pd_t\bu\|_{L_p((0, 2\pi), L_q(B_R))}\\
&\quad + \|\rho\|_{L_p((0, 2\pi), W^{3-1/q}_q(S_R))}
+ \|\pd_t\rho\|_{L_p((0, 2\pi), W^{2-1/q}_q(S_R))} 
\leq M_1\|\bff\|_{L_p((0, 2\pi), L_q(D))} + M_2\epsilon^2
\end{aligned}\end{equation}
for some constants $M_1$ and $M_2$ independent of $\epsilon \in (0, 1)$. 
Finally, we estimate $\|\pd_t\rho\|_{L_\infty((0, 2\pi), W^{1-1/q}_q(S_R))}$.  From 
the third equation in equations \eqref{5.1}, we have
$$\|\pd_t\rho\|_{W^{1-1/q}_q(S_R)}
\leq \|\CM\rho\|_{W^{1-1/q}_q(S_R)}
+ \|\CA\bu\|_{W^{1-1/q}_q(S_R)} +\| \tilde d(\bv, h)\|_{W^{1-1/q}_q(S_R)}.
$$
Therefore,  by \eqref{space:1}, \eqref{sob:1}, \eqref{real:1}, \eqref{5.2},
\eqref{5.3}, and \eqref{5.4}, we have
\begin{align*}
|\pd_t\rho\|_{L_\infty((0, 2\pi), W^{1-1/q}_q(S_R))}
&\leq C(\|\bu\|_{L_p((0, 2\pi), H^2_q(B_R))} 
+ \|\pd_t\bu\|_{L_p((0, 2\pi), L_q(B_R))}\\
&\qquad+ \|\rho\|_{L_p((0, 2\pi), W^{3-1/q}_q(S_R))}
+ \|\pd_t\rho\|_{L_p((0, 2\pi), W^{2-1/q}_q(S_R))}  +  \epsilon^2),
\end{align*}
which, combined with \eqref{non:8},  leads to 
\begin{equation}\label{non:9}
\begin{aligned}
E(\bu, \rho) 
\leq M_1'\|\bff\|_{L_p((0, 2\pi), L_q(D))} + M_2'\epsilon^2
\end{aligned}\end{equation}
for some constants $M_1'$ and $M_2'$ independent of $\epsilon \in (0, 1)$. 
We choose $\epsilon > 0$ so small that $M'_2\epsilon < 1/2$ and 
we assume that 
$M_1'\|\bff\|_{L_p((0, 2\pi), L_q(D))} \leq \epsilon/2$.
Then we have
\begin{equation}\label{non:10}
E(\bu, \rho) \leq \epsilon.
\end{equation}
Moreover, by \eqref{ext:1} and \eqref{real:1}, we have
$$\sup_{t \in (0, 2\pi)} \|H_\rho\|_{H^1_\infty(B_R))} 
\leq C\|\rho\|_{L_\infty((0, 2\pi), W^{1-1/q}_q(S_R))}
\leq M_3E(\bu, \rho)
\leq M_3\epsilon.
$$
Choosing $\epsilon > 0$ smaller if necessary, we may assume that
$0 < M_3\epsilon < \delta$, and so $(\bu, \rho) \in \CI_\epsilon$.
If we define a map $\Phi$ acting on $(\bv, h) \in \CI_\epsilon$
by setting $\Phi(\bv, h) = (\bu, \rho)$, and then 
$\Phi$ is a map from $\CI_\epsilon$ into itself. 
Employing a similar argument as for proving \eqref{non:10}, 
we see that for any $(\bv_i, h_i) \in \CI_\epsilon$ ($i = 1,2$),  
$$E(\Phi(v_1, h_1) - \Phi(v_2, h_2))
\leq M_4\epsilon E((\bv_1, h_1) - (\bv_2, h_2)).
$$
Choosing $\epsilon > 0$ smaller if necessary, we may assume that
$M_4\epsilon \leq 1/2$, and so $\Phi$ is a contraction map on
$\CI_\epsilon$.   The Banach fixed-point theorem yields the unique
existence of a fixed point $(\bv, \rho) \in \CI_\epsilon$ of the 
map $\Phi$, that is $(\bv, \rho) = \Phi(\bv, \rho)$, which is the
required solution of equations \eqref{main:eq1}.  This completes the 
proof of Theorem \ref{main:thm1}.

\end{document}